\documentclass[a4paper,10pt]{amsart}
\usepackage{amsmath,amsthm,amssymb,enumerate}
\usepackage{epsfig}
\usepackage{amssymb}
\usepackage{amsmath}
\usepackage{amssymb}
\usepackage{amsmath,amsthm}
\usepackage[latin1]{inputenc}
\usepackage{bm}
\usepackage{bbm}
\usepackage{esint}
\usepackage{tikz-cd}
\usetikzlibrary{decorations.pathreplacing}
\usepackage{amsfonts}
\usepackage{amsxtra}
\usepackage{euscript,mathrsfs}
\usepackage{color}
\usepackage[left=3.75cm,right=3.75cm,top=4cm,bottom=3cm]{geometry}
\usepackage[colorlinks=true, linktocpage=true, linkcolor=red!70!black, citecolor=green!50!black]{hyperref}
\allowdisplaybreaks

\usepackage{tikz-cd}
\usetikzlibrary{decorations.pathreplacing}

\usepackage{enumitem}
\setenumerate{label={\rm (\alph{*})}}

\usepackage{amsfonts}
\usepackage{amsxtra}

\numberwithin{equation}{section}

\newtheorem{theorem}{Theorem}[section]
\newtheorem{lemma}[theorem]{Lemma}
\newtheorem{proposition}[theorem]{Proposition}

\newtheorem{corollary}[theorem]{Corollary}

\newtheorem{remark}[theorem]{Remark}

\numberwithin{equation}{section}

\usepackage{xcolor}
\colorlet{ColorPink}{red!30}
\usepackage{graphicx}

\newcommand{\R}{\mathbb R}

\newcommand{\E}{\operatorname{E}\!}

\newcommand{\dista}{\operatorname{dist}}
\newcommand{\dif}{\mathrm{d}}

\DeclareMathOperator{\dist}{dist}

\newcommand{\sym}{\operatorname{sym}}

\renewcommand{\dif}{\operatorname{d}\!}
\newcommand{\lebe}{\operatorname{L}}
\newcommand{\sobo}{\operatorname{W}}

\newcommand{\locc}{\operatorname{loc}}

\newcommand{\hold}{\operatorname{C}}

\newcommand{\sg}{\varepsilon}

\newcommand{\bv}{\operatorname{BV}}

\newcommand{\ball}{\operatorname{B}}
\newcommand{\di}{\operatorname{div}}
\newcommand{\bd}{\operatorname{BD}}
\newcommand{\gm}{\operatorname{GM}}
\newcommand{\A}{\mathbb{A}}
\newcommand{\rsym}{\mathbb{R}_{\operatorname{sym}}^{n\times n}}
\newcommand{\ld}{\operatorname{LD}}

\newcommand{\mres}{%
  \,\raisebox{-.127ex}{\reflectbox{\rotatebox[origin=br]{-90}{$\lnot$}}}\,%
}

\newcommand{\D}{\operatorname{D}}

\newcommand{\dashint}{\fint}
\newcommand{\spt}{\operatorname{spt}}

\renewcommand{\mres}{\mathbin{\vrule height 1.6ex depth 0pt width
0.13ex\vrule height 0.13ex depth 0pt width 1.3ex}}
\newcommand{\rscew}{\mathbb{R}_{\operatorname{scew}}^{n\times n}}
\newcommand{\excenew}{\Phi}
\newcommand{\excesso}{\widetilde{\Phi}}
\newcommand{\devi}{\operatorname{dev}_{\alpha}}
\newcommand{\Lip}{\operatorname{Lip}}
\newcommand{\comp}{\operatorname{comp}}
\newcommand{\con}{\operatorname{con}}
\newcommand{\scew}{\operatorname{scew}}
\newcommand{\m}{\operatorname{m}}

\newcommand{\decay}{\operatorname{dec}}

\allowdisplaybreaks
\newcommand{\M}{\operatorname{M}}

\newcommand{\trace}{\operatorname{Tr}}

\begin{document}


\title{Regularity for the Dirichlet problem on $\bd$}
\author[F.~Gmeineder]{Franz Gmeineder}
\address{Author's Address: Mathematical Institute, University of Bonn, Endenicher Allee 60, 53115 Bonn, Germany. }


\begin{abstract}
We establish that the Dirichlet problem for convex linear growth functionals on $\bd$, the functions of bounded deformation, gives rise to the same unconditional Sobolev and partial $\hold^{1,\alpha}$-regularity theory as presently available for the full gradient Dirichlet problem on $\bv$. By \textsc{Ornstein}'s Non-Inequality, $\bv$ is a proper subspace of $\bd$, and full gradient techniques known from the $\bv$-situation do not apply here. In particular, applying to all generalised minima (i.e., minima of a suitably relaxed problem) despite their non-uniqueness and reaching the ellipticity ranges known from the $\bv$-case, this paper extends previous results by \textsc{Kristensen} and the author \cite{GK1} in an optimal way.
\end{abstract}

\maketitle

\setcounter{tocdepth}{1}
\tableofcontents

\section{Introduction}
A variety of physically relevant convex variational problems that describe the displacements of bodies subject to external forces are posed in the space $\bd$ of \emph{functions of bounded deformation}, see \cite{Anze1,FS2,ST,Suquet,Temam} for overviews. For a given open set $\Omega\subset\R^{n}$, this space consists of all $u\in\lebe^{1}(\Omega;\R^{n})$ such that the distributional symmetric gradient $\sg(u):=\frac{1}{2}(\D\!u+\D\!u^{\top})$ is a finite, matrix-valued Radon measure on $\Omega$. By \textsc{Ornstein}'s Non-Inequality \cite{Ornstein,CFM,KiKr,KaWoj}, there exists \emph{no} constant $c>0$ such that $\|D\varphi\|_{\lebe^{1}(\Omega;\R^{n\times n})}\leq c\|\sg(\varphi)\|_{\lebe^{1}(\Omega;\R^{n\times n})}$ holds for all $\varphi\in\hold_{c}^{\infty}(\Omega;\R^{n})$. In consequence, $\bd(\Omega)$ is in fact larger than $\bv(\Omega;\R^{n})$, and the full distributional gradients of $\bd$-maps in general do not need to exist as (locally) finite $\R^{n\times n}$-valued Radon measures. Yet, by coerciveness considerations as outlined below, this space displays the natural function space setup for a vast class of variational integrals. For minima of such functionals, the present paper aims to develop a regularity theory which -- from a Sobolev regularity and partial H\"{o}lder continuity perspective -- essentially yields the \emph{same results which are presently known for the Dirichlet problem on $\bv$}.  

This task can be viewed as a borderline case of a theory having emerged over the past decades. Namely, considering variational integrals 
\begin{align}\label{eq:W1pvariationalintegrals}
v\mapsto \int_{\Omega}g(\nabla v)\dif x,\qquad v\colon\Omega\to\R^{N},
\end{align}
over suitable Dirichlet classes, an abundance of criteria for improved regularity of minima is available subject to convexity, smoothness and the growth behaviour of $g\colon\R^{N\times n}\to\R$. It is only possible to give an incomplete list of the wealth of contributions, and instead we refer to \textsc{Mingione} \cite{Mingione1,Mingione2} and the references therein for more detail. When linear growth functionals are considered -- i.e., $c_{1}|z|-\gamma\leq g(z)\leq c_{2}(1+|z|)$ for some $c_{1},c_{2},\gamma>0$ and all $z\in\R^{N\times n}$ -- then compactness considerations lead to the study of minima of a suitably relaxed problem on $\bv$, cf. \cite{GMSI,Bi1,Bi2,BS1}. In both linear and superlinear growth regimes, these contributions crucially utilise at various steps that \emph{the full gradients of minimising sequences} are uniformly bounded in some $\lebe^{p}$-space, $p\geq 1$. When \eqref{eq:W1pvariationalintegrals} is modified to act on the symmetric gradients exclusively, convexity and $1<p<\infty$-growth of $g$ still allow to work on $\sobo^{1,p}$ by \textsc{Korn}'s inequality. Also, in the borderline case of $L\log L$-growth integrands as considered in the seminal works by \textsc{Fuchs \& Seregin} \cite{FS1,FS2}, one may essentially still work on $\sobo^{1,1}$ (cf. Section~\ref{sec:Orlicz}). In the \emph{linear growth}, symmetric gradient situation, however, \textsc{Ornstein}'s Non-Inequality neither allows to \emph{a priori} consider $\sobo^{1,1}$- or $\bv$-regular minima nor to employ the usual full-gradient techniques. A key question in this setting thus is \emph{to which extent the results from corresponding full gradient theory on $\bv$} continue to hold for the Dirichlet problem on $\bd$, too. 
\subsection{Aim and scope}
Toward a unifying regularity theory for the Dirichlet problem on $\bd$, we begin by giving the underlying setup first. Let $\Omega\subset\R^{n}$ be open and bounded with Lipschitz boundary $\partial\Omega$. We consider (generalised) minima of variational principles
\begin{align}\label{eq:varprin}
\text{to minimise}\;\;\;F[v]:=\int_{\Omega}f(\sg(v))\dif x \;\;\;\text{over}\;v\in\mathscr{D}_{u_{0}}, 
\end{align}
where $\mathscr{D}_{u_{0}}$ is a suitable Dirichlet class. As a key feature, we suppose that the convex integrand $f\colon\rsym\to\R$ is of \emph{linear growth}, by which we understand that 
there exist constants $c_{1},c_{2},\gamma>0$ such that there holds 
\begin{align}\label{eq:lingrowth1}\tag{LG}
c_{1}|z|-\gamma\leq f(z) \leq c_{2}(1+|z|)\qquad\text{for all}\;z\in\rsym.  
\end{align}
In this situation, we put $\ld(\Omega):=\{v\in\lebe^{1}(\Omega;\R^{n})\colon\;\sg(v)\in\lebe^{1}(\Omega;\rsym)\}$ to be endowed with the canonical norm $\|v\|_{\ld(\Omega)}:=\|v\|_{\lebe^{1}(\Omega;\R^{n})}+\|\sg(v)\|_{\lebe^{1}(\Omega;\rsym)}$, and define $\ld_{0}(\Omega)$ as the closure of $\hold_{c}^{\infty}(\Omega;\R^{n})$ with respect to $\|\cdot\|_{\ld(\Omega)}$. With this terminology, we pick $u_{0}\in\ld(\Omega)$ and set $\mathscr{D}_{u_{0}}:=u_{0}+\ld_{0}(\Omega)$. Subject to \eqref{eq:lingrowth1}, $F$ is bounded below on $\mathscr{D}_{u_{0}}$ and minimising sequences are bounded in $\ld(\Omega)$; note that this is \emph{not necessarily} the case in $\sobo^{1,1}(\Omega;\R^{n})$. By non-reflexivity of $\ld(\Omega)$ and possible concentration effects, minimising sequences do not need to be weakly relatively compact in $\ld(\Omega)$ but \emph{can be shown to be weak*-relatively compact in $\bd(\Omega)$} (cf.~Section~\ref{sec:prelims} for the requisite background terminology). As a routine consequence, for $F$ to be defined for $\bd$-maps, it must be suitably relaxed. For $u,v\in\bd_{\locc}(\Omega)$ and an open Lipschitz subset $\omega\subseteq\Omega$ we put 
\begin{align}\label{eq:relaxed}
\begin{split}
\overline{F}_{v}[u;\omega]  = \int_{\omega}f(\mathscr{E}u)\dif\mathscr{L}^{n} & + \int_{\omega}f^{\infty}\Big(\frac{\dif\E^{s}u}{\dif |\!\E^{s}u|}\Big)\dif|\!\E^{s}u| \\ & \;\;\;\;\;\;\;\;\;\;\;\;+ \int_{\partial\omega}f^{\infty}(\trace_{\partial\omega}(v-u)\odot\nu_{\partial\omega})\dif\mathscr{H}^{n-1}.
\end{split}
\end{align}
Following the by now classical works \cite{GoSe,GMSI}, $\overline{F}_{u_{0}}[u]:=\overline{F}_{u_{0}}[u;\Omega]$ then coincides with the weak*-relaxation (or weak*-Lebesgue-Serrin extension) of $F$ to $\bd(\Omega)$ subject to the Dirichlet constraint $u|_{\partial\Omega}=u_{0}$. Here, for $u\in\bd(\Omega)$ we denote the Lebesgue-Radon-Nikod\'{y}m decomposition $\E u=\E^{a}u+\E^{s}u=\mathscr{E}u\mathscr{L}^{n}+\E^{s}u$ of\footnote{From now, if the symmetric gradient of an integrable map is a measure, we write $\E u$ instead of $\sg(u)$.} $\E u$ into its absolutely continuous and singular parts for $\mathscr{L}^{n}$. Moreover, $f^{\infty}(z):=\lim_{t\searrow 0}tf(z/t)$ denotes the \emph{recession function} of $f$, capturing the integrand's behaviour at infinity. Consequently, we call a map $u\in\bd(\Omega)$ a \emph{generalised minimiser} if $\overline{F}_{u_{0}}[u]\leq \overline{F}_{u_{0}}[v]$ for all $v\in\bd(\Omega)$. Similarly, we call $u\in\bd_{\locc}(\Omega)$ a \emph{local generalised minimiser} if $\overline{F}_{u}[u;\omega]\leq \overline{F}_{u}[v;\omega]$ for all open subsets $\omega\Subset\Omega$ with Lipschitz boundary $\partial\omega$ and all $v\in\bd_{\locc}(\Omega)$. Subject to the Dirichlet datum $u_{0}$, the set of all generalised minima is denoted $\gm(F;u_{0})$ and, similarly, the set of all local generalised minima is denoted $\gm_{\locc}(F)$. As a consequence of \cite[Sec.~5]{GK1}, generalised minimisers always exist in this framework. For future reference, we remark that even if $f$ is strictly convex, generalised minima are \emph{not unique} in general; see Section~\ref{sec:obstructions} for more detail. 
\begin{figure}
\begin{centering}
\begin{tikzpicture} 
\node [above] at (0,1) {$a=1$}; 
\node [above] at (2,1) {$a=\frac{n+1}{n}$}; 
\node [above] at (4,1) {$a=\frac{n}{n-1}$};
\node [above] at (6,1) {$a=1+\frac{2}{n}$};
\node [above] at (8,1) {$a\to\infty$}; 
\node [left] at (0,0.5) {Partial Regularity, Thm.~\ref{thm:PR}}; 
\node [left] at (0,0) {$\exists\,p>1\colon\,\gm\subset\sobo_{\locc}^{1,p}$, Thm.~\ref{thm:W11reg}};
\node [left] at (0,-0.5) {$\exists\,\widetilde{p}>1\colon\,\gm\subset\sobo_{\locc}^{2,\widetilde{p}}$, Cor.~\ref{cor:2ndorder}}; 
\node [left] at (0,-1) {$\gm\subset\sobo_{\locc}^{1,1}$, cf.~\cite[Thm.~1.2]{GK1} }; 
\draw [|-|] (0,1) -- (2,1);
\draw [|-|] (2,1) -- (4,1);
\draw [|-|] (4,1) -- (6,1);
\draw [|-] (6,1) -- (8,1);
\draw [->] (0,0.5) -- (8,0.5); 
\draw [->] (0,-0.5) -- (4,-0.5); 
\draw [->] (0,-0) -- (6,-0); 
\draw [->] (0,-1) -- (2,-1); 
\end{tikzpicture} 
\caption{The regularity theory for the Dirichlet problem on $\bd$ in the framework of $a$-ellipticity (cf.~\eqref{eq:ellipticity}), contextualising the results obtained in this paper with previous work. The arrows indicate '\emph{up to, not including}'. \label{fig:1}}
\end{centering}
\end{figure}
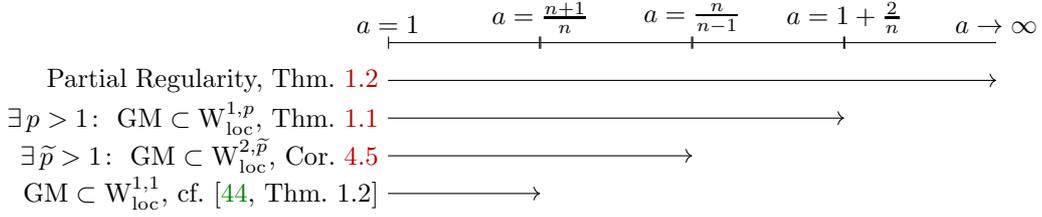

In view of the main theme of the paper, we shall focus on \emph{higher Sobolev and partial regularity} for generalised minima of the variational principle \eqref{eq:varprin}, even leading to novel results in the radially symmetric case $f=\widetilde{f}(|\cdot|)$. The corresponding results crucially rely on the degenerate elliptic behaviour of the integrands $f$, being roughly depicted in Figure~\ref{fig:1}, and let us retrieve what is unconditionally known for the Dirichlet problem on $\bv$. As such, we particularly obtain criteria for the \emph{full gradients} of generalised minima to exist as locally finite Radon measures. To explain why the results given below are close to optimal, we briefly pause to introduce the relevant ellipticity scale.
\vspace{-0.075cm}
\subsection{$\sobo_{\locc}^{1,1}$-regularity of minima}
As it is customary in the linear growth context and motivated by \textsc{Bernstein}'s genre \cite{Bernstein,GMSI,Serrin} and the conditions considered by \textsc{Ladyzhenskaya \& Ural'ceva} \cite{LU}, a natural scale of $\hold^{2}$-integrands is given by those $f\colon\rsym\to\R$ that satisfy for some $a>1$ and $0<\lambda\leq\Lambda<\infty$
\begin{align}\label{eq:ellipticity}
\lambda\frac{|\xi|^{2}}{(1+|z|^{2})^{\frac{a}{2}}}\leq \langle f''(z)\xi,\xi\rangle\leq \Lambda \frac{|\xi|^{2}}{(1+|z|^{2})^{\frac{1}{2}}}\qquad\text{for all}\;z,\xi\in\rsym.
\end{align}
For such integrands, \eqref{eq:ellipticity} precisely describes the degeneration of the ellipticity ratio of $f''$. From a more systematic viewpoint, this scale has been studied by \textsc{Bildhauer, Fuchs \& Mingione} \cite{BiFu1,Bi1,Bi2,FuMi} in the $(p,q)$-growth or $\bv$-context, respectively, under the name of $\mu$-\emph{ellipticity}, where $\mu=a$ in our terminology; also see \cite{GMSI}. Note that $a=1$ is \emph{excluded here} as then the corresponding integrands are not of linear growth. Even though convex integrands $f$ with \eqref{eq:lingrowth1} have the same growth behaviour from above and below, this is not the case on the level of second derivatives. To some extent, such problems thus have some resemblance with $(p,q)$-growth type problems. Higher integrability of the minimisers' gradients can only be expected when $p$ and $q$ are not too far apart or additional hypotheses are imposed, see the seminal work \cite{ELM} by \textsc{Esposito, Leonetti \& Mingione} (also cf. \cite[Thm.~6.2]{Mingione1} and \textsc{Carozza, Kristensen \& Passarelli di Napoli} \cite{CKP}). More precisely, for suitably regular, convex $(p,q)$-type problems the critical exponent ratio to yield $\sobo_{\locc}^{1,q}$-regular minima was determined in \cite{ELM} as 
\begin{align}
\frac{q}{p}<1+\frac{2}{n}, 
\end{align}
a threshold ratio getting in line with others identified earlier in a slightly different context by \textsc{Marcellini} \cite{Mar1,Mar2}. Beyond this threshold, one usually imposes additional hypotheses -- such as local boundedness, cf. \cite{CKP} -- on minima in order to obtain regularity results, and such conditions in fact can be justified for a variety of situations, so for instance by maximum principles or, in the radial situation, Moser-type $\lebe^{\infty}$-bounds.

This distinction of ellipticity regimes also enters in the $\bv$-theory for full gradient functionals. In fact, it is known from \cite{Bi1,BS1} that if $1<a<1+\frac{2}{n}$, then generalised minima of the corresponding full gradient functionals belong to some $\sobo_{\locc}^{1,p}$ with $p>1$ whereas in the regime $1+\frac{2}{n}\leq a \leq 3$, the only $\sobo^{1,1}$-regularity results \cite{Bi1,BS1} are subject to the additional $\lebe_{\locc}^{\infty}$-hypothesis on the generalised minima. For variational principles of the form \eqref{eq:varprin} subject to \eqref{eq:ellipticity}, a first result has been given by \textsc{Kristensen} and the author \cite{GK1} by passing to fractional estimates. While striving for the optimal ellipticity $1<a<1+\frac{2}{n}$, the method as employed therein only yields the $\sobo_{\locc}^{1,1}$-regularity for $1<a<1+\frac{1}{n}$, revealing a crucial ellipticity gap of size $\frac{1}{n}$. The first main result of this paper precisely closes this gap:

\begin{theorem}[Universal $\sobo_{\locc}^{1,1}$-regularity estimates]\label{thm:W11reg}
Let $u_{0}\in\ld(\Omega)$ and suppose that $f\in\hold^{2}(\rsym)$ satisfies \eqref{eq:lingrowth1} and \eqref{eq:ellipticity} with $1<a<1+\frac{2}{n}$. If
\begin{enumerate}
\item $n=2$, then every generalised minimiser $u\in\gm(F;u_{0})$ is of class $\ld(\Omega)\cap\sobo_{\locc}^{1,q}(\Omega;\R^{n})$ for \emph{any} $1\leq q < \infty$. More precisely, $u$ has locally exponentially integrable gradients in the following sense: There exists $c=c(a,c_{1},c_{2},\gamma,\lambda,\Lambda)>0$ such that for any $x_{0}\in\Omega$ and $0<r<1$ with $\ball(x_{0},5r)\subset\Omega$ there holds 
\begin{align}\label{eq:Kornmain0}
\|\nabla u\|_{\exp\lebe^{\frac{2-a}{3-a}}(\ball(x_{0},r);\R^{n\times n})} \leq c\Big(\Big(1+\dashint_{\ball(x_{0},5r)}|\!\E u|\Big)^{\frac{1}{2-a}} + \frac{1}{r}\dashint_{\ball(x_{0},r)}|u|\dif x\Big).
\end{align}
\item $n\geq 3$, then every generalised minimiser $u\in\gm(F;u_{0})$ is of class $\ld(\Omega)\cap\sobo_{\locc}^{1,q}(\Omega;\R^{n})$ for $q=\frac{2-a}{n-2}n$. 
More precisely, there exists $c=c(n,a,c_{1},c_{2},\gamma,\lambda,\Lambda)>0$ such that for any $x_{0}\in\Omega$ and $0<r<1$ with $\ball(x_{0},5r)\subset\Omega$ there holds 
\begin{align}\label{eq:Kornmain}
\Big(\dashint_{\ball(x_{0},r)}|\nabla u|^{q}\dif x \Big)^{\frac{1}{q}} \leq c\Big(\Big(1+\dashint_{\ball(x_{0},5r)}|\!\E u|\Big)^{\frac{1}{2-a}} + \frac{1}{r} \dashint_{\ball(x_{0},r)}|u|\dif x \Big).
\end{align}
\end{enumerate}
\end{theorem}
Theorem~\ref{thm:W11reg} thus \emph{gives exactly the same Sobolev regularity in the $\bd$-situation as is presently known for the autonomous Dirichlet problem on $\bv$}. As mentioned above, for the autonomous Dirichlet problem on $\bv$ it is possible to establish $\sobo_{\locc}^{1,L\log^{2}L}$-regularity of \emph{locally bounded} generalised minima for the wider ellipticity range $1+\frac{2}{n}\leq a \leq 3$; note that for $a>3$, no $\sobo^{1,1}$-regularity results are available at present\footnote{The only systematic $\sobo^{1,1}$-regularity theory for $a>3$ is available for Neumann-type problems on $\bv$, cf. \textsc{Beck, Bul\'{i}\v{c}ek} and the author \cite{BBG}, being conceptually different from the Dirichlet problem.}. While, in principle, the strategy underlying Theorem~\ref{thm:W11reg} can be modified to work in the $\lebe_{\locc}^{\infty}$-constrained case, too, no method is known to us that would provide locally bounded generalised minima at all. In fact, whereas maximum principles and Moser-type $\lebe_{\locc}^{\infty}$-bounds can be employed in the full gradient setting subject to specific structural conditions on the integrands (cf.~\cite[Thm.~1.11, Thms.~D.1--3]{BS1}), the symmetric gradient seems to destroy the impact of any such good structural hypotheses (so e.g. radial dependence on the arguments). In order not to produce a possibly vacuous result, we thus stick to the ellipticity range $1<a<1+\frac{2}{n}$ for which the additional local boundedness is \emph{not} required. Deferring the precise discussion to Section~\ref{sec:W11reg}, let us now briefly outline the underlying chief obstructions that make Theorem~\ref{thm:W11reg} considerably harder to obtain than its BV-analogue. 

To establish the regularity assertions of Theorem~\ref{thm:W11reg}, one might consider a vanishing viscosity sequence and then derive uniform \emph{second order} estimates. Essentially inspired by the foundational works of \textsc{Seregin} \cite{Seregin1,Seregin2,Seregin3,Seregin4}, in the $\bv$-case  a difference quotient approach yields the requisite estimates as a consequence of the fact that \emph{the full gradients of the single viscosity approximations} are uniformly bounded in $\lebe^{1}(\Omega;\R^{n\times n})$; cf.~\cite{Bi1,Bi2}. Within the framework of Theorem~\ref{thm:W11reg}, however, the latter boundedness cannot be assumed and $\lebe^{1}$-estimates on the full gradients must be avoided throughout. On the other hand, generalised minima are in general non-unique -- despite strict convexity of the integrands $f$. Hence, even if it were applicable, the vanishing viscosity approach would only apply to one particular generalised minimiser. The claimed \emph{universal} regularity estimates (i.e., for \emph{all} generalised minima) thus require another argument.

In doing so, we modify and extend the \textsc{Ekeland} viscosity approximation scheme as introduced by \textsc{Beck \& Schmidt} \cite{BS1} in the $\bv$-context and generalised to the $\bd$-situation by \textsc{Kristensen} and the author \cite{GK1}; see \cite{Mar0,AcerbiFusco0} for the first applications of the \textsc{Ekeland} variational principle in the regularity context.  Here, on the one hand, the \textsc{Ekeland}-type approximations must be strong enough for the (perturbed) Euler-Lagrange equations to permit a splitting strategy, thereby implying the requisite second order estimates for the corresponding almost-minima. Simultaneously, they must be weak enough to be treatable by the a priori information on the minimising sequences. By our arguments below -- and contrary to the $\sobo^{-1,1}$-perturbations in the $\bv$-context \cite{BS1} -- the correct perturbation space now turns out to be $\sobo^{-2,1}$ (see Section~\ref{sec:prelimsauxiliaryFS} for the definition). Without the aforementioned splitting strategy, in turn inspired by \textsc{Seregin} et al. \cite{Seregin4,FS1}, we would be bound to argue as in \cite{GK1}, and then the desired ellipticity range $1<a<1+\frac{2}{n}$ would not be reached. By the degenerate elliptic behaviour of the integrands, non-uniqueness of generalised minima and the overall lack of Korn's inequality, the proof of Theorem~\ref{thm:W11reg} requires to overcome both technical and conceptual difficulties and is given in Section~\ref{sec:W11reg} below. 

Once the presence of the singular parts $\E^{s}u$ is ruled out \emph{for all} $u\in\gm(F;u_{0})$, the boundary integrals in \eqref{eq:relaxed} are identified as the only source of non-uniqueness. This admits to apply more general principles (to be established in the Appendix, Section~\ref{sec:uniquenessappendix}, with emphasis on the two-dimensional case) to draw conclusions on the structure of $\gm(F;u_{0})$, cf. Section~\ref{sec:cors}.  
\subsection{Partial $\hold^{1,\alpha}$-regularity of minima}
The second part of this paper is devoted to the partial (H\"{o}lder) regularity of generalised minima of $F$. We note that, essentially because the minimisation of $F$ constitutes a \emph{vectorial} problem, full H\"{o}lder continuity in general is not to be expected; see \cite{Giaquinta,Giusti,Mingione2,Mingione3} and the references therein. To streamline terminology, in this paper we say that a map $v\in\lebe_{\locc}^{1}(\Omega;\R^{n})$ is \emph{partially regular}  if there exists a relatively open subset $\Omega_{u}\subset\Omega$ such that $v$ is of class $\hold^{1,\alpha}$ in a neighbourhood of any of the elements of $\Omega_{u}$ for any $0<\alpha<1$. 

There is an extensive literature on the topic of partial regularity and proof strategies, most notably the (indirect) blow-up method with roots in \textsc{De Giorgi}'s work \cite{DeGiorgi} and the $\mathcal{A}$-harmonic approximation method with roots in \textsc{Almgren}'s and \textsc{Allard}'s work in geometric measure theory \cite{Allard,Almgren}. These proof strategies have been adapted to the setting of functionals of the type \eqref{eq:varprin} with $\sg$ replaced by the full gradient, see \cite{AcerbiFusco0,AcerbiFusco,EvansPR,Du1,Du3,Mingione1} for an incomplete list. For instance, even in the convex full-gradient linear growth case, indirect methods such as blow-up are difficult to implement by the relatively weak compactness properties of $\bv$ \emph{as long as no additional Sobolev regularity is available}. Appealing to Theorem~\ref{thm:W11reg}, this is e.g. the case in the very degenerate regime $a\geq 1+\frac{2}{n}$. On the other hand, should an integrand degenerate completely for large values of the argument, one might still aim for a \emph{local-in-phase-space} regularity result (in the terminology of \textsc{Schmidt} \cite{Schmidt2}). 

To establish such a regularity theorem, in turn being able to cover all degenerate ellipticities, we make use of a direct strategy using mollifications as comparison maps. Since, by Jensen's inequality, mollifications can be suitably controlled by convex functions, this method is particularly designed for \emph{convex} problems. Originally employed by \textsc{Anzellotti \& Giaquinta} \cite{AG} in the full gradient context (also see the related result by \textsc{Schmidt} \cite{Schmidt1} for the model integrands $\m_{p}(\cdot)=(1+|\cdot|^{p})^{\frac{1}{p}}$, $p\neq 2$), functionals \eqref{eq:relaxed} require a different treatment. First, now the decay of the comparison maps must appear as a consequence of a careful linearisation and hereafter \textsc{Korn}'s inequality in $\lebe^{2}$. More importantly, the comparison argument forces us to control $V$-function type distances from a given generalised minimiser to its mollifications \emph{by the symmetric gradients only}. While this is a consequence of the fundamental theorem of calculus in the $\bv$-context, the requisite estimates now must be accessed without appealing to the full gradients. This motivates the derivation of a novel family of convolution-type Poincar\'{e} inequalities in Section~\ref{sec:poincare}, which might be of independent interest. Lastly, the estimates of Section~\ref{sec:poincare} necessitate a refined construction of \emph{good} annuli in the partial regularity proof, where the key parts of the comparison are performed. A combination of these tools in Section~\ref{sec:PR} then yields an $\varepsilon$-regularity result (cf.~Corollary~\ref{cor:epsreg}), and implies the following second main result of the paper:
\begin{theorem}[Local-in-phase-space regularity]\label{thm:PR}
Let $f\in \hold^{2}(\R_{\sym}^{n\times n})$ be convex and satisfy \eqref{eq:lingrowth1}. Given $u_{0}\in\ld(\Omega)$, let $u\in\gm(F;u_{0})$. If $(x_{0},z_{0})\in\Omega\times\R_{\sym}^{n\times n}$ is such that 
\begin{align}\label{eq:theimportantcondition}
\lim_{R\searrow 0}\left[\dashint_{\ball(x_{0},R)}|\mathscr{E}u-z_{0}|\dif x + \frac{|\E^{s}u|(\ball(x_{0},R))}{\mathscr{L}^{n}(\ball(x_{0},R))}\right]=0
\end{align}
and $f''(z_{0})$ is positive definite, then there holds $u\in \hold^{1,\alpha}(U;\R^{n})$ for a suitable open neighbourhood $U$ of $x_{0}$ for all $0<\alpha<1$. In consequence, if $f''$ is positive definite everywhere on $\rsym$, then the \emph{singular set} $\Sigma_{u}$ of points in whose neighourhood $u$ is not of class $\hold^{1,\alpha}$ for any $0<\alpha<1$ satisfies $\mathscr{L}^{n}(\Sigma_{u})=0$, is relatively closed and is given by 
\begin{align}
\Sigma_{u}=\Big\{x_{0}\in \Omega\colon\;\text{there exists no $z_{0}\in\rsym$ with \eqref{eq:theimportantcondition}} \Big\}.
\end{align} 
\end{theorem}
Similar as in $\bv$-theory, the importance of the previous theorem is manifested by its minimal assumptions regarding locality and (degenerate) ellipicity; in fact, no global uniform strong convexity needs to be imposed on $f$ in order to yield the corresponding partial $\hold^{1,\alpha}$-regularity result. Recalling the $a$-ellipticity scale \eqref{eq:ellipticity}, Theorem~\ref{thm:PR} thus particularly complements Theorem~\ref{thm:W11reg}
in the very degenerate ellipticity regime $1+\frac{2}{n}\leq a <\infty$, cf. Figure~\ref{fig:1}. As a routine matter, however, strengthening the ellipticity to $1<a<\frac{n}{n-1}$, Theorem~\ref{thm:W11reg} can be invoked to yield bounds on $\dim_{\mathscr{H}}(\Sigma_{u})$ -- cf. Corollary~\ref{cor:dimbound}. We moreover note that the previous theorem equally proves interesting for \emph{radially symmetric integrands}. Indeed, techniques to arrive at \emph{full $\hold^{1,\alpha}$-regularity results} in the full gradient setting (cf. \textsc{Uhlenbeck} \cite{Uhlenbeck}, \textsc{Ural'ceva} \cite{Uralceva} or \textsc{Beck \& Schmidt} \cite{BS2} in the $\bv$-context) are hard to be implemented: The symmetric gradient seems to destroy the beneficial structure of the corresponding Euler-Lagrange equations. As such, Theorem~\ref{thm:PR} seems hard to be generalised to the model integrands $\m_{p}$ (revealing $p$-Laplacean type behaviour at the origin) for $p\neq 2$, cf. Section~\ref{sec:exampleintegrands} for a discussion. Finally, recalling the aim of a regularity result in the very degenerate ellipticity regime, Theorem~\ref{thm:PR} proves independent of the recent companion theorem~\cite{G3} for strongly symmetric quasiconvex integrals by the author. Whereas the main difficulties in \cite{G3} stem from the weakened convexity notion, its application to convex integrands only yields a partial regularity theorem for $a$-elliptic integrands, $1<a\leq 3$. A discussion of these matters, together with possible generalisations of Theorems~\ref{thm:W11reg} and \ref{thm:PR} is given in Section~\ref{sec:extensions}.

\subsection{Organisation of the paper}
In Section~\ref{sec:prelims} we fix notation, record basic definitions and auxiliary estimates. After a discussion of sample integrands in Section~\ref{sec:exampleintegrands}, we provide the proof of Theorem~\ref{thm:W11reg} and selected implications in Section~\ref{sec:W11reg}. Section~\ref{sec:poincare} provides convolution-type Poincar\'{e} inequalities to crucially enter the proof of Theorem~\ref{thm:PR} in Section~\ref{sec:PR}. Section~\ref{sec:extensions}  discusses generalisations of the results of the paper, and the appendices, Sections~\ref{sec:uniquenessappendix} and \ref{sec:proofaux}, comprise selected uniqueness assertions and proofs of auxiliary results.

{\small \subsection*{Acknowledgments} 
I am grateful to \textsc{Jan Kristensen} for a fruitful interaction on the theme of the paper.  Moreover, I am indebted to \textsc{Gianni Dal Maso} and \textsc{Gregory Seregin} for commenting on previous results of mine and thereby making valuable suggestions, which motivated the Sobolev regularity improvement compared to \cite{GK1}. I am also thankful to \textsc{Lars Diening} for discussions related to the convolution-type Poincar\'{e} inequalities of Section~\ref{sec:poincare}. Financial support through the Hausdorff Center in Mathematics, Bonn, is gratefully acknowledged.}
\section{Preliminaries}\label{sec:prelims}
\subsection{General notation and background}\label{sec:notation}
We briefly comment on the notation used throughout. By $\rsym$ or $\R_{\scew}^{n\times n}$ we denote the symmetric or scew-symmetric $(n\times n)$-matrices with real entries.  All finite dimensional vector spaces are equipped with the euclidean (or, in the matrix case, Frobenius) norm $|\cdot|$, and the inner product on such spaces is denoted $\langle\cdot,\cdot\rangle$. Given $a,b\in\R^{n}$, the symmetric tensor product is given by $a\odot b :=\frac{1}{2}(ab^{\mathsf{T}}+ba^{\mathsf{T}})$. Given $x_{0}\in\R^{n}$ and $r>0$, the open ball of radius $r>0$ centered at $x_{0}\in\R^{n}$ is denoted $\ball(x_{0},r):=\{x\in\R^{n}\colon\;|x-x_{0}|<r\}$. For $0<t<s<\infty$, we denote $\mathcal{A}(x_{0};t,s):=\ball(x_{0},s)\setminus\ball(x_{0},t)$ 
the annulus centered at $x_{0}$ having outer and inner radii $s$ and $t$, respectively.  To distinguish from balls in matrix space, we write $\mathbb{B}(z,r):=\{y\in\rsym\colon\;|y-z|<r\}$ for $z\in\rsym$ and $r>0$. Cubes $Q$ in $\R^{n}$ are tacitly assumed to be non-degenerate, and we denote by $\ell(Q)$ their sidelengths. The $n$-dimensional Lebesgue and $(n-1)$-dimensional Hausdorff measure are denoted $\mathscr{L}^{n}$ and $\mathscr{H}^{n-1}$, respectively. Accordingly, the Hausdorff dimension of a Borel set $A\in\mathscr{B}(\R^{n})$ is denoted $\dim_{\mathscr{H}}(A)$. For $u\in\lebe_{\locc}^{1}(\R^{n};\R^{m})$ and an open set $U\subset\R^{n}$ with $\mathscr{L}^{n}(U)<\infty$, we use the shorthand $(u)_{U}:=\dashint_{U}u\dif x:=\mathscr{L}^{n}(U)^{-1}\int_{U}u\dif x$ whereas, if $U=\ball(x,r)$ is ball, we abbreviate $(u)_{x,r}:=(u)_{\ball(x,r)}$. Moreover, for a given finite dimensional real vector space $V$, we denote $\mathscr{M}_{(\locc)}(\Omega;V)$ the $V$-valued (locally) finite Radon measures on (the open set) $\Omega$. For $\mu\in\mathscr{M}(\Omega;V)$, its Lebesgue-Radon-Nikod\'{y}m decomposition is given by $\mu=\mu^{a}+ \frac{\dif\mu}{\dif|\mu^{s}|}|\mu^{s}|$, where $\mu^{a}\ll\mathscr{L}^{n}$ and $\mu^{s}\bot\mathscr{L}^{n}$.

By $c,C>0$  we denote generic constants whose value might change from line to line, and shall only be specified if their precise value is required. 
\subsection{Function spaces and integral operators}
In this section we give an overview of the requisite function spaces on which the main part is based. This comprises functions of bounded deformation, to be discussed in Section~\ref{sec:prelimsBD}, as well as Orlicz and negative Sobolev spaces to be introduced and discussed in Sections~\ref{sec:Orlicz} and \ref{sec:prelimsauxiliaryFS}. 
\subsubsection{Functions of bounded deformation}\label{sec:prelimsBD}
Let $\Omega\subset\R^{n}$ be open and bounded. We then define $\bd(\Omega)$ as the space of all $u\in\lebe^{1}(\Omega;\R^{n})$ for which the \emph{total deformation} 
\begin{align}
|\!\E u|(\Omega):=\sup\Big\{\int_{\Omega}\langle u,\di(\varphi)\rangle\dif x\colon\;\varphi\in\hold_{c}^{1}(\Omega;\R_{\sym}^{n\times n}),\|\varphi\|_{\lebe^{\infty}(\Omega;\R_{\sym}^{n\times n})}\leq 1\Big\}
\end{align}
is finite; note that by writing $\E u$ we indicate that the symmetric distributional gradient of $u$ is a measure whereas by $\sg(u)$ we tacitly understand that it is representable by an $\lebe^{1}$-map. This space has been introduced in \cite{CMS,Suquet} and studied from various perspectives in \cite{Anze1,ACD,ST,Baba}; unless stated otherwise, all of the following can be traced back to these references. Given $u\in\bd(\Omega)$, the Lebesgue-Radon-Nikod\'{y}m decomposition of $\E u$ reads as $\E u=\E^{a}u+\E^{s}u=\mathscr{E}u\mathscr{L}^{n}\mres\Omega+\frac{\dif\E^{s}u}{\dif|\E^{s}u|}|\E^{s}u|$. Here, $\mathscr{E}u$ takes the r\^{o}le of the symmetric part of the approximate gradient (cf.~\cite{AFP} for this terminology).

Let $u,u_{1},u_{2},...\in\bd(\Omega)$. We say that $u_{k}\stackrel{*}{\rightharpoonup}u$ if and only if $u_{k}\to u$ in $\lebe^{1}(\Omega;\R^{n})$ and $\E u_{k}\stackrel{*}{\rightharpoonup} \E u$ in $\mathscr{M}(\Omega;\rsym)$. If $u_{k}\stackrel{*}{\rightharpoonup} u$ as just defined and $|\!\E u_{k}|(\Omega)\to |\!\E u|(\Omega)$, then we say that $(u_{k})$ converges \emph{(symmetric) strictly} to $u$. If, moreover, $\sqrt{1+|\E u_{k}|^{2}}(\Omega)\to  \sqrt{1+|\E u|^{2}}(\Omega)$ with 
\begin{align*}
\sqrt{1+|\!\E v|^{2}}(\Omega):=\int_{\Omega}\sqrt{1+|\mathscr{E}v|^{2}}\dif x + |\!\E^{s}v|(\Omega),\qquad v\in\bd(\Omega),
\end{align*}
then we say that $(u_{k})$ converges \emph{(symmetric) area-strictly} to $u$. These notions are usually reserved for the $\bv$-context, but as we deal with the $\bd$-situation exclusively we shall often omit the supplementary \emph{symmetric} and simply speak of strict and area-strict convergence. 

Now let $\Omega$ have Lipschitz boundary $\partial\Omega$. Both $\ld(\Omega)$ and $\bd(\Omega)$ then have trace space $\lebe^{1}(\partial\Omega;\R^{n})$; however, note that the trace operator onto $\lebe^{1}(\partial\Omega;\R^{n})$ is not continuous with respect to weak*-convergence on $\bd(\Omega)$. In this case, continuity can only be achieved when $\bd(\Omega)$ is equipped with strict convergence. Moreover, as $\Omega$ has Lipschitz boundary, any $u\in\bd(\Omega)$ can be extended by zero to the entire $\R^{n}$ so that the trivial extension $\overline{u}$ again belongs to $\bd(\R^{n})$ and we have 
\begin{align*}
\E\overline{u} = \E u\mres\Omega + \trace_{\partial\Omega}(u)\odot\nu_{\partial\Omega}\mathscr{H}^{n-1}\mres\partial\Omega, 
\end{align*}
where $\nu_{\partial\Omega}$ is the outward unit normal to $\partial\Omega$. Also, we have the \emph{Gau\ss -Green formula}
\begin{align}\label{eq:GaussGreen}
\int_{\Omega}\langle\varphi,\E u\rangle + \int_{\Omega}\langle\di(\varphi),u\rangle\dif x = \int_{\partial\Omega}\langle \varphi,\trace_{\partial\Omega}(u)\odot\nu_{\partial\Omega}\rangle\dif\mathscr{H}^{n-1}
\end{align}
for all $u\in\bd(\Omega)$ and all $\varphi\in\hold^{1}(\overline{\Omega};\rsym)$; here, $\di$ denotes the row-wise divergence. For latter applications, the following approximation result will turn out particularly useful:
\begin{lemma}\label{lem:smooth}
Let $\Omega\subset\R^{n}$ be an open and bounded Lipschitz domain. Then for any $u\in\bd(\Omega)$ and any $u_{0}\in\ld(\Omega)$ there exists a sequence $(u_{k})\subset u_{0}+\hold_{c}^{\infty}(\Omega;\R^{n})$ such that $u_{k}\to u$ in $\lebe^{1}(\Omega;\R^{n})$ and 
\begin{align*}
\sqrt{1+|\!\E u_{k}|^{2}}(\Omega) \to \sqrt{1+|\!\E u|^{2}}(\Omega) + \int_{\partial\Omega}|\trace_{\partial\Omega}(u_{0}-u)\odot\nu_{\partial\Omega}|\dif\mathscr{H}^{n-1}\qquad\text{as}\;k\to\infty. 
\end{align*}
\end{lemma} 
\subsubsection{Korn- and Poincar\'{e} inequalities in Lebesgue and Orlicz spaces}\label{sec:Orlicz}
To transfer integrability from $\sg(u)$ to the full gradients in a flexible space scale, we recall here Korn-type inequalities in Orlicz spaces; our notation is mainly taken from the recent work of \textsc{Cianchi} \cite{Cianchi}, also see \textsc{Acerbi \& Mingione} \cite{AcerbiMingione} for related results. 

Let $A\colon[0,\infty)\to [0,\infty)$ be a Young function; by this we understand that $A(t)=\int_{0}^{t}a(\tau)\dif\tau$ for $t\geq 0$, where $a\colon [0,\infty)\to[0,\infty]$ is non-decreasing, left-continuous and being neither identical to $0$ nor $\infty$. We then denote $\lebe^{A}(\Omega;\R^{m})$ the linear space of all measurable maps $u\colon\Omega\to\R^{m}$ such that the Luxembourg norm
\begin{align*}
\|u\|_{\lebe^{A}(\Omega;\R^{m})}:=\inf\left\{\lambda>0\colon\; \int_{\Omega}A\Big(\frac{|u|}{\lambda} \Big)\dif x \leq 1\right\}
\end{align*} 
is finite. We then define $E^{1}A(\Omega)$ as the space of all $u\in\lebe^{A}(\Omega;\R^{n})$ such that the distributional symmetric gradient belongs to $\lebe^{A}(\Omega;\rsym)$. As examples, if $A(t)=|t|$, then $E^{1}A(\R^{n})=\ld(\R^{n})$, if $A(t)=|t|^{p}$ for $1<p<\infty$, then $E^{1}A(\R^{n})=\sobo^{1,p}(\R^{n};\R^{n})$. It is worth noting that the Young function $A(t):=t\log(1+t)$ displays a borderline case: For $\alpha\geq 0$, the general conclusion
\begin{align}\label{eq:logsing}
\sg(v)\log^{\alpha}(1+|\sg(v)|)\in\lebe_{\locc}^{1}(\R^{n};\rsym) \Longrightarrow Dv\in\lebe_{\locc}^{1}(\R^{n};\R^{n\times n})
\end{align}
persists if and only if $\alpha\geq 1$; hence, briefly recalling the $L\log L$-setup mentioned in the introduction, variational problems with symmetric gradients belonging to $L\log L$ are essentially dealt with in $\sobo^{1,1}$. Namely, by the \textsc{Smith} representation \cite{Smith} to be used in a different context later on, $u=(u^{1},...,u^{n})\in\hold_{c}^{\infty}(\R^{n};\R^{n})$ can be retrieved from $\sg(u)=(\sg_{ij}(u))_{i,j=1}^{n}$ via 
\begin{align}\label{eq:Smith}
u^{k} = \frac{2}{n\omega_{n}}\sum_{1\leq i\leq j \leq n}\sg_{jk}(u)*\partial_{i}K_{ij}-\sg_{ij}(u)*\partial_{k}K_{ij}+\sg_{ki}*\partial_{j}K_{ij}
\end{align}
for all $k\in\{1,...,n\}$, 
where $K_{ij}(x):=x_{i}x_{j}/|x|^{n}$ for $x\in\R^{n}\setminus\{0\}$. The convolutions here are understood in the Cauchy principal value sense, and so the map $\Phi\colon \sg(u)\mapsto \nabla u$ displays a singular integral of convolution type satisfying the usual H\"{o}rmander condition. Then \eqref{eq:logsing} follows from the theory of singular integrals on Orlicz spaces, cf.~\cite{Cianchi}. For the following, let us remind the reader of the space of \emph{rigid deformations}
\begin{align}
\mathscr{R}(\Omega):=\big\{u\colon\Omega\to\R^{n}\colon\;u(x)=Ax+b,\;A\in\rscew,\,b\in\R^{n} \big\}
\end{align}
which, for open and connected $\Omega$, is precisely the nullspace of $\sg$. Since elements of $\mathscr{R}(\Omega)$ are polynomials, we shall often identify $\mathscr{R}(\Omega)$ with $\mathscr{R}(\R^{n})$.
\begin{lemma}[{\cite[Thm.~3.3, Cor.~3.4, Ex.~3.11]{Cianchi}}]\label{lem:Cianchi}
Let $\Omega\subset\R^{n}$ be an open, bounded and connected Lipschitz domain. Then the following holds: 
\begin{enumerate}
\item\label{item:Cianchi0} For each $1<p<\infty$ there exists $c=c(p,n,\Omega)>0$ such that  
\begin{align*}
\inf_{\pi\in\mathscr{R}(\Omega)}\|\nabla (v-\pi)\|_{\lebe^{p}(\Omega;\R^{n\times n})}\leq c\|\sg(v)\|_{\lebe^{p}(\Omega;\rsym)}\qquad\text{for all}\; v\in \sobo^{1,p}(\Omega;\R^{n}). 
\end{align*}
Moreover, for all $v\in\sobo^{1,p}(\Omega;\R^{n})$ there holds
\begin{align*}
\|\nabla v - (\nabla v)_{\Omega}\|_{\lebe^{p}(\Omega;\R^{n\times n})}\leq c\|\sg(v)-(\sg(v))_{\Omega}\|_{\lebe^{p}(\Omega;\rsym)}.
\end{align*}
\item\label{item:Cianchi1} For each $\beta>0$ there exists $c=c(\beta,n,\Omega)>0$ such that 
\begin{align*}
\inf_{\pi\in\mathscr{R}(\Omega)}\|\nabla(v-\pi)\|_{\exp\lebe^{\frac{\beta}{\beta+1}}(\Omega;\R^{n\times n})}\leq c\|\sg(v)\|_{\exp\lebe^{\beta}(\Omega;\rsym)}\qquad\text{for all}\;v\in E^{1}\exp\lebe^{\beta}(\Omega), 
\end{align*}
where $\exp\lebe^{\beta}(\Omega)$ is the Orlicz space corresponding to $A(t):=\exp(t^{\beta})$. 
\end{enumerate}
\end{lemma}
In the sequel, we gather some instrumental results on certain projection operators and augment \eqref{eq:Smith} by a decomposition result due to \textsc{Reshetnyak} \cite{Reshetnyak2}. Note that, since $\mathscr{R}(\ball(0,1))$ is a finite dimensional vector space, \emph{all norms are equivalent on} $\mathscr{R}(\ball(0,1))$. Thus, by scaling, we find that for each $1\leq q < \infty$ there exists a constant $c(n,q)>0$ such that for all $x_{0}\in\R^{n}$ and $r>0$ there holds 
\begin{align}\label{eq:rigidscaling}
\Big(\dashint_{\ball(x_{0},r)}|\pi|^{q}\dif x\Big)^{\frac{1}{q}} + r\Big(\dashint_{\ball(x_{0},r)}|\nabla \pi|^{q}\dif x\Big)^{\frac{1}{q}} \leq c(n,q) \dashint_{\ball(x_{0},r)}|\pi|\dif x
\end{align}
for all $\pi\in\mathscr{R}(\ball(x_{0},r))$. The same inequality holds true with the obvious modifications for if $q=\infty$ on the left-hand side. Moreover, there exists a bounded linear projection operator $\Pi_{\ball(x_{0},r)}\colon\lebe^{1}(\ball(x_{0},r);\R^{n})\ni u\mapsto \pi_{u}\in\mathscr{R}(\ball(x_{0},r))$ satisfying 
\begin{align}\label{eq:Lpstability}
\Big(\dashint_{\ball(x_{0},r)}|\Pi_{\ball(x_{0},r)}u|^{q}\dif x \Big)^{\frac{1}{q}}\leq c(n,q)\Big(\dashint_{\ball(x_{0},r)}|u|^{q}\dif x\Big)^{\frac{1}{q}}
\end{align}
for all $u\in\lebe^{q}(\ball(x_{0},r);\R^{n})$ and each $1\leq q <\infty$; see the appendix, Section~\ref{sec:stability}, for an elementary proof. A similar result holds for cubes $Q$ instead of balls, and we shall refer to this property as \emph{$\lebe^{q}$-stability} of $\Pi_{\ball(x_{0},r)}$ or $\Pi_{Q}$, respectively. In a routine manner, the foregoing now yields the next lemma which should be well-known, but is hard to be found in the following form:
\begin{lemma}[Projections in Poincar\'{e}- and Korn-type inequalities]\label{rem:stability}
Let $1\leq p < \infty$, $x_{0}\in\R^{n}$ and $r>0$. For each $1\leq q\leq p$ there exists a constant $c=c(n,q)>0$ such that for all $u\in\sobo^{1,p}(\ball(x_{0},r);\R^{n})$ there exists $\pi_{u}\in\mathscr{R}(\ball(x_{0},r))$ such that 
\begin{align*}
\dashint_{\ball(x_{0},r)}|u-\pi_{u}|^{q}\dif x \leq c r^{q} \dashint_{\ball(x_{0},r)}|\sg(u)|^{q}\dif x.
\end{align*} 
In particular, the map $\Pi_{\ball(x_{0},r)} u\mapsto \pi_{u}$ is independent of $q$. The same holds true if we set $q=1$ and replace $\sobo^{1,1}(\ball(x_{0},r);\R^{n})$ by $\ld(\ball(x_{0},r))$. Moreover, if $1<q\leq p$, then there exists a constant $c=c(n,q)>0$ such that for all $u\in\sobo^{1,p}(\ball(x_{0},r);\R^{n})$ there holds (with the same $\pi_{u}$ as above) 
\begin{align*}
\dashint_{\ball(x_{0},r)}|D(u-\pi_{u})|^{q}\dif x \leq c \dashint_{\ball(x_{0},r)}|\sg(u)|^{q}\dif x.
\end{align*}
Moreover, the map $\Pi_{\ball(x_{0},r)}\colon u\mapsto\pi_{u}$ is $\lebe^{q}$-stable for each $1\leq q \leq p$ in the above sense.
\end{lemma}
Clearly, a similar version holds for cubes. 
The foregoing lemma will not be sufficient for all future applications, and so we record the following result due to \textsc{Reshetnyak} \cite{Reshetnyak2}. As we will exclusively use it for cubes, we directly state it in the following form:
\begin{lemma}[{\textsc{Reshetnyak}, \cite{Reshetnyak2}}]\label{lem:Reshrep}
For any open, non-empty cube $Q\subset\R^{n}$, there exists a projection $\widetilde{\Pi}_{Q}\colon \hold^{\infty}(Q;\R^{n})\cap\ld(Q)\to\mathscr{R}(Q)$ and an operator $T_{Q}\colon (\hold^{\infty}\cap\lebe^{1})(Q;\rsym)\to\lebe^{1}(Q;\R^{n})$ such that for any $v\in\ld(Q)$ there holds 
\begin{align}\label{eq:decompResh}
v(x)=(\widetilde{\Pi}_{Q} v)(x) + T_{Q}[\sg(v)](x)\qquad\text{for all}\;\;x\in Q.
\end{align}
Moreover, the operator $T_{Q}$ is of the form 
\begin{align}\label{eq:repkernel}
T_{Q}[\sg(v)](x)= \int_{Q}R_{Q}(x,y)\sg(v)(y)\dif y,
\end{align}
where $R_{Q}\colon Q\times Q\to\mathscr{L}(\rsym;\R^{n})$ satisfies $|R_{Q}(x,y)|\leq c/|x-y|^{n-1}$ for all $x,y\in Q$ with $c=c(n)>0$.
\end{lemma}

\subsubsection{Negative Sobolev spaces}\label{sec:prelimsauxiliaryFS}
The viscosity approximation strategy to be set up in Section~\ref{sec:W11reg} shall require certain negative Sobolev spaces in a crucial manner. As shall become clear later, we have to go beyond the space $\sobo^{-1,1}$ as introduced in \cite{BS1}. Given $k\in\mathbb{N}$, we define the space $\sobo^{-k,1}(\Omega;\R^{n})$ as follows:
\begin{align*}
\sobo^{-k,1}(\Omega;\R^{n}):=\Big\{T\in\mathscr{D}'(\Omega;\R^{n})\colon\;T=\sum_{\substack{\alpha\in\mathbb{N}_{0}^{n} \\ |\alpha|\leq k}}\partial^{\alpha}T_{\alpha},\;T_{\alpha}\in\lebe^{1}(\Omega;\R^{n})\;\text{for all}\;|\alpha|\leq k\Big\}.
\end{align*}
The linear space $\sobo^{-k,1}(\Omega;\R^{n})$ is canonically endowed with the norm 
\begin{align}
\|T\|_{\sobo^{-k,1}(\Omega;\R^{n})}:=\inf \sum_{|\alpha|\leq k}\|T_{\alpha}\|_{\lebe^{1}(\Omega;\R^{n})}, 
\end{align}
the infimum ranging over all representations $T=\sum_{|\alpha|\leq k}\partial^{\alpha}T_{\alpha}$ with $T_{\alpha}\in\lebe^{1}(\Omega;\R^{n})$. Similar as for $\sobo^{-1,1}(\Omega;\R^{n})$ as discussed in \cite{BS1}, $\sobo^{-k,1}$ is not approachable by duality. We collect its most important properties in the following lemma. 
\begin{lemma}\label{lem:negative}
Let $\Omega\subset\R^{n}$ be open and let $k\in\mathbb{N}$ be given. Then the following holds: 
\begin{enumerate}
\item \label{item:negative1} $(\sobo^{-k,1}(\Omega;\R^{n}),\|\cdot\|_{\sobo^{-k,1}(\Omega;\R^{n})})$ is a Banach space. 
\item \label{item:negative2} For every $u\in\lebe^{1}(\Omega;\R^{n})$ and every $\beta\in\mathbb{N}_{0}^{n}$ with $|\beta|\leq k$ there holds 
\begin{align*}
\|\partial^{\beta}u\|_{\sobo^{-k,1}(\Omega;\R^{n})}\leq \|u\|_{\sobo^{|\beta|-k,1}(\Omega;\R^{n})}.
\end{align*} 
\end{enumerate}
\end{lemma}
\begin{proof}
In view of \ref{item:negative1}, we closely follow \cite{BS1} and consider the mapping $\Phi\colon \lebe^{1}(\Omega;\R^{n})^{N} \ni (T_{\alpha})_{|\alpha|\leq k} \mapsto \sum_{|\alpha|\leq k}\partial^{\alpha}T_{\alpha}\in\sobo^{-k,1}(\Omega;\R^{n})$, where $N:=\#\{\alpha\in\mathbb{N}_{0}^{n}\colon\;|\alpha|\leq k\}$. By definition of $\sobo^{-k,1}(\Omega;\R^{n})$, $\Phi$ is a bounded linear operator and thus $\ker(\Phi)$ is a Banach space in itself. By definition of the quotient norm, the canonical quotient map $\Psi\colon \lebe^{1}(\Omega;\R^{n})^{N}/\ker(\Phi)\to \sobo^{-k,1}(\Omega;\R^{n})$ is surjective and isometric. Thus, as $\ker(\Phi)$ is Banach, so is $\lebe^{1}(\Omega;\R^{n})^{N}/\ker(\Phi)$ and eventually, as the isometric image of a Banach space, $(\sobo^{-k,1}(\Omega;\R^{n}),\|\cdot\|_{\sobo^{-k,1}(\Omega;\R^{n})})$. For \ref{item:negative2}, let $\varepsilon>0$ and choose $(T_{\alpha})_{\alpha}\in\lebe^{1}(\Omega;\R^{n})^{N}$ such that $u=\sum_{|\alpha|\leq k-|\beta|}\partial^{\alpha}T_{\alpha}$ and
\begin{align*}
\sum_{|\alpha|\leq k-|\beta|}\|T_{\alpha}\|_{\lebe^{1}(\Omega;\R^{n})} \leq \|u\|_{\sobo^{|\beta|-k,1}(\Omega;\R^{n})}+\varepsilon.
\end{align*}
On the other hand, $\partial^{\beta}u=\sum_{|\alpha|\leq k-|\beta|}\partial^{\alpha+\beta}T_{\alpha}=:\sum_{|\gamma|\leq k}\partial^{\gamma}S_{\gamma}$, where $S_{\gamma}=T_{\alpha}$ if $\gamma=\alpha+\beta$ for some $\alpha$ with $|\alpha|\leq k-|\beta|$ and $S_{\gamma}=0$ otherwise. Therefore, 
\begin{align*}
\|\partial^{\beta}u\|_{\sobo^{-k,1}(\Omega;\R^{n})}\leq \sum_{|\gamma|\leq k}\|S_{\gamma}\|_{\lebe^{1}(\Omega;\R^{n})}\leq \sum_{|\alpha|\leq k-|\beta|}\|T_{\alpha}\|_{\lebe^{1}(\Omega;\R^{n})} \leq \|u\|_{\sobo^{|\beta|-k,1}(\Omega;\R^{n})}+\varepsilon, 
\end{align*}
and we then send $\varepsilon\searrow 0$ to conclude the proof. 
\end{proof}
Next, a lower semicontinuity result in the spirit of \cite[Lem.~3.2]{GK1}, \cite[Lem.~2.6]{BS1}:
\begin{lemma}\label{lem:EkelandLSC}
Let $1<q<\infty$, $k\in\mathbb{N}$ be given and let $\Omega$ be open and bounded with Lipschitz boundary $\partial\Omega$. Suppose that $\mathfrak{f}\colon\rsym\to\R$ is a convex function that satisfies $c^{-1}|z|^{q}-d\leq \mathfrak{f}(z)\leq c(1+|z|^{q})$ for some $c,d>0$ and all $z\in\rsym$. Then, for every $u_{0}\in\sobo^{1,q}(\Omega;\R^{n})$, the functional 
\begin{align*}
\mathcal{F}[u]:=\begin{cases} \displaystyle \int_{\Omega}\mathfrak{f}(\sg(u))\dif x&\;\text{if}\;u\in \mathscr{D}_{u_{0}}:=u_{0}+\sobo_{0}^{1,q}(\Omega;\R^{n}),\\
+\infty&\;\text{if}\;u\in\sobo^{-k,1}(\Omega;\R^{n})\setminus\mathscr{D}_{u_{0}}
\end{cases}
\end{align*}
is lower semicontinuous for the norm topology on $\sobo^{-k,1}(\Omega;\R^{n})$. 
\end{lemma}
\begin{proof}
Let $g,g_{1},g_{2},... \in \sobo^{-k,1}(\Omega;\R^{n})$ be such that $g_{m}\to g$ with respect to the norm topology on $\sobo^{-k,1}(\Omega;\R^{n})$. If $\liminf_{m\to\infty}\mathcal{F}[g_{m}]=+\infty$, there is nothing to prove. Hence assume without loss of generality that $\lim_{j\to\infty}\mathcal{F}[g_{m(j)}]=\liminf_{m\to\infty}\mathcal{F}[g_{m}]<\infty$. Then necessarily $g_{m(j)}\in \mathscr{D}_{u_{0}}$ for all sufficiently large indices $j$ and, since $c^{-1}|z|^{q}-d\leq \mathfrak{f}(z)$ for all $z\in\rsym$, we obtain that $(\sg(g_{m(j)}))$ is bounded in $\lebe^{q}(\Omega;\rsym)$. Since $g_{m(j)}\in\mathscr{D}_{u_{0}}$ and $q>1$, \textsc{Korn}'s inequality in $\sobo_{0}^{1,q}(\Omega;\R^{n})$ implies that $(g_{m(j)})$ is uniformly bounded in $\sobo^{1,q}(\Omega;\R^{n})$. Since $1<q<\infty$, there exists a subsequence $(g_{m(j(i))})\subset (g_{m(j)})$ which converges weakly in $\sobo^{1,q}(\Omega;\R^{n})$ to some $\widetilde{g}\in\mathscr{D}_{u_{0}}$ (note that $\mathscr{D}_{u_{0}}$ is weakly closed in $\sobo^{1,q}(\Omega;\R^{n})$). By the \textsc{Rellich-Kondrachov} theorem, we can moreover assume that $g_{m(j(i))}\to \widetilde{g}$ strongly in $\lebe^{q}(\Omega;\R^{n})$. Then, since $\lebe^{q}(\Omega;\R^{n})\hookrightarrow\sobo^{-k,1}(\Omega;\R^{n})$ by Lemma~\ref{lem:negative}\ref{item:negative2}, 
\begin{align*}
\|g-\widetilde{g}\|_{\sobo^{-k,1}(\Omega;\R^{n})}\leq \|g-g_{m(j(i))}\|_{\sobo^{-k,1}(\Omega;\R^{n})} + \|\widetilde{g}-g_{m(j(i))}\|_{\lebe^{1}(\Omega;\R^{n})}\to 0,\qquad i\to\infty, 
\end{align*}
and thus $g=\widetilde{g}$. By standard results on lower semicontinuity of convex variational integrals of superlinear growth (or, alternatively, \textsc{Reshetnyak}'s lower semicontinuity theorem, Theorem~\ref{lem:reshetnyak} below) $\sg(g_{m(j(i))})\mathscr{L}^{n}\stackrel{*}{\rightharpoonup}\sg(g)\mathscr{L}^{n}$ as $i\to\infty$ thus yields 
\begin{align*}
\mathcal{F}[g] & \leq \liminf_{i\to\infty}\mathcal{F}[g_{m(j(i))}]=\liminf_{m\to\infty}\mathcal{F}[g_{m}].
\end{align*}
The proof is complete. 
\end{proof}
\subsection{The Ekeland variational principle}
In this section we recall a variant of the \textsc{Ekeland} variational principle \cite{Ekeland} that is suitable for our purposes. The version which we state here is a merger of \cite[Thm.~5.6, Rem.~5.5]{Giusti}:
\begin{proposition}\label{prop:Ekeland}
Let $(V,d)$ be a complete metric space and let $\mathcal{F}\colon V\to\R\cup\{\infty\}$ be a lower semicontinuous function (for the metric topology) which is bounded from below and takes a finite value at some point. Suppose that, for some $u\in V$ and some $\varepsilon>0$, there holds $\mathcal{F}[u]\leq \inf \mathcal{F}[V]+\varepsilon$. Then there exists $v\in V$ such that 
\begin{enumerate}
\item $d(u,v)\leq \sqrt{\varepsilon}$,
\item $\mathcal{F}[v]\leq \mathcal{F}[u]$, 
\item for all $w\in V$ there holds $\mathcal{F}[v]\leq \mathcal{F}[w] + \sqrt{\varepsilon}d(v,w)$. 
\end{enumerate}
\end{proposition}
\subsection{Functions of measures and convolutions}\label{sec:funofmeas}
In this section we collect background facts on linear growth integrands and functionals of the form \eqref{eq:varprin}. We begin with 
\begin{lemma}\label{lem:boundbelow}
Suppose that $f\in\hold^{2}(\rsym)$ is convex and satisfies \eqref{eq:lingrowth1} with $c_{1},c_{2},\gamma>0$. Then $f$ is Lipschitz with $\Lip(f)\leq c_{2}$. 
\end{lemma}
The proof of the preceding lemma evolves in the same way as \cite[Lem.~5.2]{Giusti}; the reader might notice that for the conclusion of Lemma~\ref{lem:boundbelow} it is sufficient that $f$ is symmetric rank-one convex -- so convex with respect to directions $a\odot b$, $a,b\in\R^{n}$ -- and satisfies \eqref{eq:lingrowth1}. As to (lower semi)continuity, we shall mostly rely on the following theorem due to \textsc{Reshetnyak} \cite{Reshetnyak1} (see \cite{Reshetnyak1,Anzellotti,AFP} for more information on functions of measures):
\begin{theorem}[\textsc{Reshetnyak} (lower semi-)continuity]\label{lem:reshetnyak}
Let $V$ be a finite dimensional real vector space and let $(\nu_{j})$ be a sequence in $\mathscr{M}(\Omega;V)$ that converges in the weak*-sense to some $\nu\in\mathscr{M}(\Omega;V)$. Moreover, assume that all of $\nu,\nu_{1},\nu_{2},... $ take values in some closed convex cone $K\subset V$. Then the following holds:
\begin{enumerate}
\item If $g \colon K \to \R_{\geq 0}\cup\{+\infty\}$ is lower semicontinuous, convex and $1$-homogeneous, then there holds
\begin{align*}
\int_{\Omega} g\Big(\frac{\dif \nu}{\dif |\nu|}\Big)\dif |\nu| \leq \liminf_{j\to\infty}\int_{\Omega}g\Big(\frac{\dif \nu_{j}}{\dif |\nu_{j}|}\Big)\dif |\nu_{j}|.
\end{align*}
\item  If $g \colon K \to \R_{\geq 0}\cup\{+\infty\}$ is continuous, $1$-homogeneous and if $(\nu_{j})$ converges strictly to~$\nu$ (in the sense that $\nu_{j}\stackrel{*}{\rightharpoonup}\nu$ and $|\nu_{j}|(\Omega)\to|\nu|(\Omega)$), then there holds
\begin{align*}
\int_{\Omega}g\Big(\frac{\dif \nu}{\dif |\nu|}\Big)\dif |\nu| = \lim_{j\to\infty}\int_{\Omega}g\Big(\frac{\dif \nu_{j}}{\dif |\nu_{j}|}\Big)\dif |\nu_{j}|.
\end{align*}
\end{enumerate}
\end{theorem}
Given a lower semicontinuous, convex function $h\colon \rsym\to\R_{\geq 0}$, we put $V:=\R\times \rsym$  and introduce the \emph{linear perspective integrand} $h^{\#}\colon \R_{\geq 0}\times \rsym\to\R\cup\{+\infty\}$ by 
\begin{align}\label{eq:perspectivefunction}
h^{\#}(t,\xi):=\begin{cases} th\Big(\frac{\xi}{t}\Big),&\;t>0,\;\xi \in \rsym,\\
h^{\infty}(\xi)&\;t=0,\;\xi\in \rsym, 
\end{cases}
\end{align}
where $h^{\infty}(\xi)=\lim_{t\searrow 0}th(\frac{\xi}{t})$ so that $h^{\#}$ is positively $1$-homogeneous. Also, if $h$ has linear growth, then $h^{\#}<\infty$. We put $K:=\R_{\geq 0}\times\rsym$.  For $\Omega\subset\R^{n}$ open and $\mu\in\mathscr{M}(\Omega;\rsym)$, we put $\nu:=(\mathscr{L}^{n},\mu)\in\mathscr{M}(\Omega;K)$ and define for $A\in\mathscr{B}(\Omega)$ 
\begin{align*}
h[\mu](A):=\int_{A}h(\mu):= \int_{A}h^{\#}\Big(\frac{\dif\nu}{\dif|\nu|}\Big)\dif|\nu| & = \int_{A}h^{\#}\Big(\frac{\dif\mathscr{L}^{n}}{\dif|\nu|},\frac{\dif\mu}{\dif|\nu|}\Big)\dif|\nu| \\ 
& = \int_{A}h\Big(\frac{\dif\mu}{\dif\mathscr{L}^{n}}\Big)\dif\mathscr{L}^{n} + \int_{A}h^{\infty}\Big(\frac{\dif\mu^{s}}{\dif|\mu^{s}|}\Big)\dif|\mu^{s}|. 
\end{align*}
In particular, if $u,u_{1},u_{2},...\in\bd(\Omega)$ are such that $u_{j}\to u$ symmetric area-strictly in $\bd(\Omega)$ and $f\colon\rsym\to\R_{\geq 0}$ satisfies \eqref{eq:lingrowth1}, then $f[\E u_{j}](\Omega)\to f[\E u](\Omega)$. 

For $\mu\in\mathscr{M}(\Omega;\rsym)$ and $\xi_{0}\in\rsym$, we use the convention 
\begin{align*}
\mu - \xi_{0} := \mu - \xi_{0}\mathscr{L}^{n}. 
\end{align*}
As for $\lebe_{\locc}^{1}$-maps, we define the average of $\mu\in\mathscr{M}(\Omega;\rsym)$ over $\ball(x_{0},r)\subset\Omega$ by
\begin{align}\label{eq:meanvaluemeasures}
(\mu)_{x_{0},r}:=\dashint_{\ball(x_{0},r)}\mu := \frac{\mu(\ball(x_{0},r))}{\mathscr{L}^{n}(\ball(x_{0},r))}.
\end{align}
By the Lebesgue differentiation theorem for Radon measures, $\mathscr{L}^{n}$-a.e. $x_{0}\in\R^{n}$ is a Lebesgue point for $\mu$ in the sense that there exists $\xi_{0}\in\rsym$ such that 
\begin{align}\label{eq:Lebesgue}
\lim_{r\searrow 0}(|\mu-\xi_{0}|)_{x_{0},r}=0.
\end{align}
The Jensen inequality here takes the following form,  cf.~\cite[Lem.~4.12]{Schmidt1}: If $h\colon\rsym\to\R_{\geq 0}$ is convex, then 
\begin{align}\label{eq:Jensenmain}
h\big((\mu)_{x_{0},r}\big)\leq \big(h[\mu] \big)_{x_{0},r}.
\end{align}
For future applications in Section~\ref{sec:poincare} and~\ref{sec:PR}, we call a compactly supported, radial  function $\rho\colon\R^{n}\to [0,1]$ a \emph{standard mollifier} provided $\|\rho\|_{\lebe^{1}(\R^{n})}=1$, $\spt(\rho)\subset\overline{\ball(0,1)}$ and $\rho$ is of class $\hold^{\infty}$ in $\ball(0,1)$. Given $\varepsilon>0$, we then define the $\varepsilon$-rescaled variant by $\rho_{\varepsilon}(x):=\varepsilon^{-n}\rho(\tfrac{x}{\varepsilon})$. As a consequence of \eqref{eq:Jensenmain}, whenever $\mu\in\mathscr{M}_{\locc}(\R^{n};\rsym)$ and $\varepsilon>0$, 
\begin{align*}
h\big((\rho_{\varepsilon}*\mu)\big) \leq (\rho_{\varepsilon}*h[\mu])\qquad\text{in}\;\R^{n}.
\end{align*}
Below, we shall particularly work with the following two choices $\rho^{(1)},\rho^{(2)}\colon\R^{n}\to\R$:
\begin{align*}
\rho^{(1)}:=(\mathscr{L}^{n}(\ball(0,1)))^{-1}\mathbbm{1}_{\ball(0,1)}\;\;\text{and}\;\;\rho^{(2)}:=\gamma_{n}\mathbbm{1}_{\ball(0,1)}\exp\Big(-\frac{1}{1-|\cdot|^{2}}\Big), 
\end{align*}
where $\gamma_{n}$ is adjusted in a way such that $\|\rho^{(2)}\|_{\lebe^{1}(\R^{n})}=1$. Given $u\in\lebe_{\locc}^{1}(\Omega;\R^{n})$ and $\mu\in\mathscr{M}_{\locc}(\Omega;\rsym)$, we put
\begin{align}\label{eq:udeltaepsdef}
\begin{split}
&u_{\varepsilon}:=\rho_{\varepsilon}^{(1)}*u\;\;\;\;\;\;\;\;\;\;\;\;\text{and}\;\;\; u_{\varepsilon,\varepsilon}:=\rho_{\varepsilon}^{(2)}*u_{\varepsilon}, \\
&\mu_{\varepsilon}:=(\rho_{\varepsilon}^{(1)}*\mu) \mathscr{L}^{n}\;\;\;\text{and}\;\;\;\mu_{\varepsilon,\varepsilon}:=(\rho_{\varepsilon}^{(2)}*\mu_{\varepsilon})\mathscr{L}^{n}
\end{split}
\end{align}
for $\varepsilon>0$. Upon straightforward modification, the proof of \cite[Lem.~5.2]{AG} then implies  
\begin{lemma}\label{lem:convex}
Let $\mu\in\mathscr{M}_{\locc}(\Omega;\rsym)$ and let $x_{0}\in\Omega$, $R>0$ be such that $\ball(x_{0},R)\Subset\Omega$. Moreover, let $\varepsilon>0$ satisfy $\varepsilon<\frac{R}{2}$. Then for any convex integrand 
$f\in\hold^{2}(\rsym;\R_{\geq 0})$ with \eqref{eq:lingrowth1} the following holds:
\begin{enumerate}
\item\label{item:AG1} If $0<t_{1}<t_{2}<R-2\varepsilon$, then there exists $t\in (t_{1},t_{2})$ such that 
\begin{align*}
f[\mu_{\varepsilon,\varepsilon}](\ball(x_{0},t))-f[\mu](\ball(x_{0},t))\leq \frac{4\varepsilon}{t_{2}-t_{1}}f[\mu](\ball(x_{0},R)). 
\end{align*}
\item\label{item:AG2} If $R/2<t_{1}<t_{2}<R-2\varepsilon$ and $0<r<R/4$, then there exist $r'\in (r,2r)$ and $t'\in (t_{1},t_{2})$ such that, adopting the annulus notation of Section~\ref{sec:notation},
\begin{align*}
f[\mu_{\varepsilon,\varepsilon}](\mathcal{A}(x_{0},r',t'))-f[\mu](\mathcal{A}(x_{0},r',t'))\leq 4\varepsilon \left(\frac{1}{t_{2}-t_{1}}+\frac{1}{r}\right)f[\mu](\ball(x_{0},R)). 
\end{align*}
\end{enumerate}
\end{lemma}
\subsection{Estimates on $V$-functions and shifted integrands}\label{sec:aux}
We now collect estimates on auxiliary $V$-functions to be dealt with later. To this end, we define for $z\in\R^{m}$ the auxiliary reference integrand 
\begin{align*}
V(z):=\sqrt{1+|z|^{2}}-1,\qquad z\in\R^{m}.
\end{align*}
The functions $V$ will help to define our excess quantity later on, and we record  
\begin{lemma}\label{lem:eest}
For every $m\in\mathbb{N}$, all $z,z'\in\R^{m}$ and $t\geq 0$ the following holds:
\begin{enumerate}
\item\label{item:eest1} $V(tz)\leq 4\max\{t,t^{2}\}V(z)$, 
\item\label{item:eest2} $V(z+z')\leq 2(V(z)+V(z'))$,
\item\label{item:eest3} $(\sqrt{2}-1)\min\{|z|,|z|^{2}\}\leq V(z)\leq \min\{|z|,|z|^{2}\}$,
\item\label{item:eest4} and for every $\ell>0$ there exists a constant $c=c(\ell)>0$ such that if $|z|\leq\ell$, then 
\begin{align*}
\tfrac{1}{c}|z|^{2}\leq V(z) \leq c|z|^{2}.
\end{align*}  
\end{enumerate}
\end{lemma}
All assertions \ref{item:eest2}--\ref{item:eest4} are contained in \cite[Sec.~2.4, Eq.~(2.4)]{GK2}, \cite[Prop.~2.5]{AG}, easily implying \ref{item:eest1}. We conclude this preliminary section with estimates on shifted integrands. To this end, let $f\in\hold^{2}(\rsym)$ be an integrand satisfying \eqref{eq:lingrowth1}. Given $a\in\rsym$, we define the \emph{shifted} or \emph{linearised} integrands $f_{a}\colon\rsym\to\R$ by 
\begin{align}\label{eq:tildef}
f_{a}(\xi):=f(a+\xi)-f(a)-\langle f'(a),\xi\rangle,\;\;\;\xi\in \R_{\sym}^{n\times n}.
\end{align}
We state the next lemma in a form that is directly applicable to our future objectives:
\begin{lemma}\label{lem:shifted}
Let $f\in\hold^{2}(\rsym;\R_{\geq 0})$ be convex and satisfy \eqref{eq:lingrowth1}. Moreover, let $\xi_{0}\in\rsym$ and $0<\varrho_{\xi_{0}}<1$ be such that 
\begin{align}\label{eq:positivedefinitebound}
m_{\xi_{0},\varrho_{\xi_{0}}}:=\min\{\lambda(z)\;\text{smallest eigenvalue of}\;f''(z)\colon\;z\in\overline{\mathbb{B}(\xi_{0},\varrho_{\xi_{0}})}\}>0. 
\end{align}
Then for all matrices $a\in\rsym$ with $\mathbb{B}(a,\tfrac{\varrho_{\xi_{0}}}{2})\subset\mathbb{B}(\xi_{0},\varrho_{\xi_{0}})$ the following holds:
\begin{enumerate}
\item\label{item:shifted1} $f_{a}$ is convex with $f_{a}(0)=0$ and $f'_{a}(0)=0$. Moreover, $f_{a}\geq 0$. 
\item\label{item:shifted2} 
For all $\xi\in\rsym$ we have, with $c(\tfrac{\varrho_{\xi_{0}}}{2})>0$ as in Lemma~\ref{lem:eest}\ref{item:eest4}
\begin{align*}
m_{\xi_{0},\varrho_{\xi_{0}}}\Big(\frac{\varrho_{\xi_{0}}}{2}\Big)^{2}V(\xi) \leq f_{a}(\xi) \leq \Big(c(\tfrac{\varrho_{\xi_{0}}}{2})\sup_{\mathbb{B}(\xi_{0},\varrho_{\xi_{0}})}|f''| +\frac{16\Lip(f)}{(\sqrt{2}-1)\varrho_{\xi_{0}}}\Big)V(\xi).
\end{align*}
\end{enumerate}
\end{lemma}
The elementary proof of the preceding lemma is deferred to the appendix, Section~\ref{sec:proofshifted}.

\section{Examples of integrands and limitations}\label{sec:exampleintegrands}
In this quick intermediate section we present and discuss several sample integrands that underline the applicability and limitations of the main results of the present paper. Here, a scale of integrands $(\Phi_{a})_{1<a<\infty}$ is given by 
\begin{align*}
\Phi_{a}(\xi):=\int_{0}^{|\xi|}\int_{0}^{s}\frac{\dif t}{(1+t^{2})^{\frac{a}{2}}}\dif s,\qquad \xi\in\rsym. 
\end{align*}
Then, essentially by \cite[Ex.~3.9 and 4.17]{Bi2}, $\Phi_{a}$ is $a$-elliptic and \emph{not} $b$-elliptic for any $1<b<a$. Such integrands are covered by Theorem~\ref{thm:W11reg} for if $1<a<1+\frac{2}{n}$, and by Theorem~\ref{thm:PR} for all $1<a<\infty$. The latter theorem particularly includes the example of the area integrand $E(\xi):=\sqrt{1+|\xi|^{2}}$, being $3$-elliptic; recall that we dispense with $3$-elliptic integrands in the framework of Theorem~\ref{thm:W11reg} as we do not have justification of generalised minima belonging to $\lebe_{\locc}^{\infty}$ for such integrands -- a condition which is usually required for Sobolev regularity in the full gradient situation, too. 

An intermediate class of integrands is given by $(\M_{p})_{1\leq p < \infty}$ defined by 
\begin{align*}
\M_{p}(\xi):=\big(1+(1+|\xi|^{2})^{\frac{p}{2}}\big)^{\frac{1}{p}},\qquad\xi\in\rsym. 
\end{align*}
These integrands are $a=3$-elliptic for if $p=1$, and $a=p+1$-elliptic for if $p>1$, cf. \cite[Sec.~3.1]{BS1}. However, integrands that indeed fall \emph{outside} the scope of the paper are the linear growth integrands $(\m_{p})_{1<p<\infty}$ given by $\m_{p}(\xi):=(1+|\xi|^{p})^{\frac{1}{p}}$ unless $p=2$; cf. \textsc{Schmidt} \cite{Schmidt1}, \cite[p.~7]{Schmidt2} for the proof. In fact, if $1<p<2$, then $\m''_{p}(z)$ blows up as $|z|\searrow 0$ and if $2<p<\infty$, then $\m''_{p}(0)=0$. In these situations, Theorem~\ref{thm:PR} applies only if $u\in\gm(F;u_{0})$ satisfies $\inf_{\Omega}|\sg(u)|>0$. Namely, if $|z|\searrow 0$, then $\m_{p}$ exhibits the behaviour of the $p$-Dirichlet energies and, as to partial regularity, forces to employ a $p$-harmonic comparison strategy. Whereas this does work well in the \emph{full gradient case} \cite{Schmidt1} following the works of \textsc{Duzaar \& Mingione} \cite{Du4,Du5}, the requisite comparison estimates in the symmetric gradient context seem to be not available at present.

\section{Local $\sobo^{1,1}$-regularity and the proof of Theorem~\ref{thm:W11reg}}\label{sec:W11reg}
In this section we establish the $\sobo_{\locc}^{1,1}$--regularity result asserted by Theorem~\ref{thm:W11reg}. Here we employ a refined version of a vanishing viscosity approach, to be set up in Section~\ref{sec:viscosity}, with the ultimate objective to obtain suitable second order estimates in Section~\ref{sec:secondorder}. In Section~\ref{sec:W11regproof} we then establish Theorem~\ref{thm:W11reg} and collect selected implications in Section~\ref{sec:cors}, thereby completing the lower three regularity assertions gathered in Figure~\ref{fig:1}. 

\subsection{Strategy and obstructions}\label{sec:obstructions}
We start by clarifying the underlying obstructions first, thereby motivating the particular setup of the proof. For $f$ is convex, the higher Sobolev regularity of Theorem~\ref{thm:W11reg} is usually  accessed through the Euler-Lagrange system satisfied by $u\in\gm(F;u_{0})$. On the other hand, as $\E u$ is a finite $\rsym$-valued Radon measure, the relevant Euler-Lagrange system needs to be understood in the sense of \textsc{Anzellotti} \cite{Anzellotti}, containing the gradient of the positively homogeneous recession function $f^{\infty}$. Note that $f^{\infty}$ essentially ignores the specific ellipticity of $f$ (e.g., with the integrands $\m_{p}$ from the previous section, $\m_{p}^{\infty}=|\cdot|$ for \emph{all} $1 \leq p <\infty$), and hence it is difficult to extract the relevant higher integrability as long as the presence of $\E^{s}u$ is not ruled out per se. Equally, this also explains why directly working on the minima is in fact a useful device for the partial regularity to be addressed in Section~\ref{sec:PR}; we here essentially restrict ourselves to neighbourhoods of points where $\E^{s}u$ is assumed to vanish, cf. Theorem~\ref{thm:PR}.

To overcome this issue in view of higher Sobolev regularity, one is led to consider good minimising sequences, usually obtained by a vanishing viscosity approach, and derive the requisite compactness estimates. As it is common in the case of degenerate $p$-growth functionals with $1<p<2$, the original functionals are stabilised by adding quadratic Dirichlet energies $\frac{1}{j}\|\nabla v\|_{\lebe^{2}}^{2}$. The minima $v_{j}$ of the correspondingly perturbed functionals then are proven to converge (up to a subsequence) to a minimiser $v$ of the original functional, and uniform regularity restimates on the $v_{j}$'s carry over to $v$. When the $p$-growth integrand $f$ (for $p>1$) is strongly convex, hence strictly convex, minima are unique and so this method in particular leads to the requisite regularity estimates for \emph{all} minima. In the linear growth setting, cf.~\eqref{eq:lingrowth1}, the recession function $f^{\infty}$ is positively $1$-homogeneous and thus \emph{never strictly convex} despite possible $a$-ellipticity (and hereafter strict convexity) of $f$. Since, by the representation \eqref{eq:relaxed} of the weak*-relaxed functional $\overline{F}_{u_{0}}[-;\Omega]$, the recession function acts on the singular part $\E^{s}u$ of $u\in\gm(F;u_{0})$ exclusively, even strict convexity of $f$ does \emph{not} imply uniqueness of generalised minima -- which cannot be expected in general anyway, compare the  counterexamples in the more classical $\bv$-case \cite{Finn,Santi}. A vanishing viscosity approach as outlined above thus \emph{is only able to yield Sobolev regularity for at most one generalised minimiser as long as generalised minima are not known to belong to $\ld_{(\locc)}$}.

In particular, based on this approach,  we cannot rule out the existence of other, more irregular generalised minima. A similar issue has been encountered by \textsc{Beck \& Schmidt} \cite{BS1} in the $\bv$-setting and by \textsc{Kristensen} and the author \cite{GK1}. To circumvent this issue, we adapt and extend the modified vanishing viscosity approaches outlined in \cite{BS1,GK1}. Effectively, we start from \emph{an arbitrary} given generalised minimiser $u\in\gm(F;u_{0})$ and construct a suitable minimising sequence $(v_{j})$ that converges to $u$ in the weak*-sense in $\bd(\Omega)$. To do so, we consider an extension of a perturbed version of $F$ to a suitable negative Sobolev space for whose topology the perturbed functional turns out lower semicontinuous. Then \textsc{Ekeland}'s variational principle provides us with an 'almost minimiser' of the perturbed functional, cf. Section~\ref{sec:viscosity}. Such almost minimisers satisfy Euler-Lagrange \emph{differential inequalities} which make elliptic estimates available.  Finally, these almost minimisers are shown to converge in the weak*-sense to the given generalised minimiser, and uniform regularity estimates will eventually inherit to the latter. 

In the setting of functionals on $\bv$ as considered in \cite{BS1}, perturbations in $\sobo^{-1,1}$ are sufficient. This is due to the fact that the full gradients of generalised minima are a priori known to exist as finite Radon measures. As discussed at length in \cite{GK1,G1}, the implementation of the underlying difference quotient approach in the setup of functionals \eqref{eq:varprin} leads to terms of the form 
\begin{align}\label{eq:Tdef}
T=\int_{\Omega}\rho^{2}\frac{|\Delta_{s,h}u_{j}|^{2}}{(1+|\sg(u_{j})|^{2})^{\frac{1}{2}}}\dif x, 
\end{align} 
where $\rho\colon\Omega\to [0,1]$ is a localisation function and, given $v\colon\R^{n}\to\R^{n}$, $h\neq 0$ and $s\in\{1,...,n\}$, 
\begin{align*}
\Delta_{s,h}v(x):=\frac{1}{h}(v(x+he_{s})-v(x))
\end{align*}
denotes the difference quotient of $v$. Here, $(u_{j})$ is a suitable minimising sequence converging to $u$ in the weak*-sense. In the $\bv$-setting (in which case the symmetric gradients in the definition of $T$ are replaced by the full ones), the term $T$ can be controlled by $\|\nabla u_{j}\|_{\lebe^{1}}$. As $T$ is a priori not controllable by \textsc{Ornstein}'s Non-Inequality in the $\bd$-situation, \textsc{Kristensen} and the author \cite{GK1} employ fractional estimates in order to avoid the appearance of $T$, simultaneously perturbing in the space $(\sobo_{0}^{1,\infty})^{*}(\Omega;\R^{n})$. The latter method, being based on the embedding $\bd(\Omega)\hookrightarrow\sobo^{s,n/(n-1+s)}(\Omega;\R^{n})$ for $0<s<1$ then yields weighted Nikolski\u{\i} estimates (and thus $\sobo^{\alpha,1}$-estimates for some suitable $0<\alpha<1$) for the symmetric gradients of generalised minima. However, this only yields the smaller range of ellipticities $1<a<1+\frac{1}{n}$. Still, since $f\in\hold^{2}(\rsym)$, generalised minima should be expected to satisfy a differentiable Euler-Lagrange equation and hence the use of fractional methods does not give the expected optimal ellipticity range $1<a<1+\frac{2}{n}$. In order to obtain the latter, it seems that we are bound to obtain \emph{uniform weighted second order estimates} in the spirit of \textsc{Bildhauer} \cite[Lem.~4.19]{Bi2} or \textsc{Beck \& Schmidt} \cite[Lem.~5.2]{BS1}. Unlike the full gradient case, the requisite second estimates do not come out by a plain difference quotient approach but a fine analysis of the identities provided by suitably weakly perturbed Euler-Lagrange systems, see Theorem~\ref{thm:regdual} below. In view of this aim, it turns out that the suitable perturbation space is $\sobo^{-2,1}(\Omega;\R^{n})$ (cf.~Section~\ref{sec:prelimsauxiliaryFS}), and we can now turn to the precise implementation of the approximation argument.
\vspace{-0.25cm}
\subsection{Viscosity approximations}\label{sec:viscosity}
We now set up the Ekeland-type viscosity approximation scheme, and hereafter suppose that  $f\in\hold(\rsym)$ is convex with \eqref{eq:lingrowth1} and $u_{0}\in\ld(\Omega)$. For ease of notation, we write $\overline{F}:=\overline{F}_{u_{0}}[-;\Omega]$ in the sequel. Let $u\in\gm(F;u_{0})$ be arbitrary. By smooth approximation in the (symmetric) area-strict topology, Lemma~\ref{lem:smooth}, we find a sequence $(u_{j})\subset \mathscr{D}_{u_{0}}:=u_{0}+\ld_{0}(\Omega)$ such that 
\begin{align}\label{eq:approximateconvergencesviscosity}
\begin{split}
& u_{j}\to u\;\;\text{in}\;\lebe^{1}(\Omega;\R^{n}),\\ & \sqrt{1+|\!\E u_{j}|^{2}}(\Omega) \to \sqrt{1+|\!\E u|^{2}}(\Omega) + \int_{\partial\Omega}|\trace_{\partial\Omega}(u_{0}-u)\odot\nu_{\partial\Omega}|\dif\mathscr{H}^{n-1}.
\end{split}
\end{align}
By Theorem~\ref{lem:reshetnyak}~ff. and hereafter continuity of $w\mapsto f[\E w](\Omega)$ for the symmetric area-strict metric, $(u_{j})$ is a minimising sequence for $\overline{F}$, and we have $F[u_{j}]=\overline{F}[u_{j}]\to \overline{F}[u]=\min \overline{F}[\bd(\Omega)]$. Passing to a non-relabeled subsequence, we may thus assume
\begin{align}\label{eq:almostoptimal1}
\min \overline{F}[\bd(\Omega)] \leq F[u_{j}] \leq \min \overline{F}[\bd(\Omega)] + \frac{1}{8j^{2}}\qquad\text{for all $j\in\mathbb{N}$}. 
\end{align}
Since the trace operator $\trace\colon\ld(\R^{n}\setminus\overline{\Omega})\to \lebe^{1}(\partial\Omega;\R^{n})$ is surjective, we find a compactly supported extension $\overline{u}_{0}\in\ld(\R^{n})$ of $u_{0}$. After a routine mollification of $\overline{u}_{0}$, we obtain $u_{j}^{\partial\Omega}\in\sobo^{1,2}(\Omega;\R^{n})$ such that 
\begin{align}\label{eq:boundarybound}
\|u_{j}^{\partial\Omega}-u_{0}\|_{\ld(\Omega)}\leq \frac{1}{8\Lip(f)j^{2}}, 
\end{align}
where $\Lip(f)$ is the Lipschitz constant of $f$ (cf. Lemma~\ref{lem:boundbelow}). We then put $\mathscr{D}_{j}:=u_{j}^{\partial\Omega}+\sobo_{0}^{1,2}(\Omega;\R^{n})\subset\sobo^{1,2}(\Omega;\R^{n})$. Since $u_{j}-u_{0}\in\ld_{0}(\Omega)$, we find $\widetilde{u}_{j}\in \mathscr{D}_{j}$ such that 
\begin{align*}
\|u_{j}-u_{0}-(\widetilde{u}_{j}-u_{j}^{\partial\Omega})\|_{\ld(\Omega)}\leq \frac{1}{8\Lip(f)j^{2}}, 
\end{align*}
from where it follows that 
\begin{align}\label{eq:almostoptimal2}
\|\widetilde{u}_{j}-u_{j}\|_{\ld(\Omega)}\leq \|u_{j}-u_{0}-(\widetilde{u}_{j}-u_{j}^{\partial\Omega})\|_{\ld(\Omega)} + \|u_{0}-u_{j}^{\partial\Omega}\|_{\ld(\Omega)} \leq \frac{1}{4\Lip(f)j^{2}}. 
\end{align}
Since $\sobo_{0}^{1,2}(\Omega;\R^{n})\subset\ld_{0}(\Omega)$, we find for arbitrary $\varphi\in\sobo_{0}^{1,2}(\Omega;\R^{n})$:
\begin{align*}
\inf F[\mathscr{D}_{u_{0}}]  & \leq F[u_{0}+\varphi] \\
& = F[u_{0}+\varphi]-F[u_{j}^{\partial\Omega}+\varphi] + F[u_{j}^{\partial\Omega}+\varphi]\\
& \leq \Lip(f)\|\sg(u_{0}-u_{j}^{\partial\Omega})\|_{\lebe^{1}(\Omega;\rsym)} + F[u_{j}^{\partial\Omega}+\varphi]\\
&\stackrel{\eqref{eq:boundarybound}}{\leq} \frac{1}{8j^{2}} + F[u_{j}^{\partial\Omega}+\varphi]. 
\end{align*}
At this stage, we infimise the previous overall inequality over all $\varphi\in\sobo_{0}^{1,2}(\Omega;\R^{n})$ to obtain 
\begin{align}\label{eq:almostoptimal3}
\inf F[\mathscr{D}_{u_{0}}] \leq \frac{1}{8j^{2}} + \inf F[\mathscr{D}_{j}]. 
\end{align}
Then, since $\min \overline{F}[\bd(\Omega)]=\inf F[\mathscr{D}_{u_{0}}]$, we deduce that
\begin{align}\label{eq:almostoptimal3.0}
\begin{split}
F[\widetilde{u}_{j}] & \;\leq F[\widetilde{u}_{j}]-F[u_{j}] + F[u_{j}] \\ & \;\leq \Lip(f)\|\sg(\widetilde{u}_{j})-\sg(u_{j})\|_{\lebe^{1}(\Omega;\rsym)} + F[u_{j}]\\ &  \stackrel{\eqref{eq:almostoptimal2}}{\leq} \frac{1}{4j^{2}} + F[u_{j}] \stackrel{\eqref{eq:almostoptimal1}}{\leq} \frac{3}{8j^{2}} + \inf F[\mathscr{D}_{u_{0}}]\\
& \stackrel{\eqref{eq:almostoptimal3}}{\leq} \frac{1}{2j^{2}} + \inf F[\mathscr{D}_{j}]. 
\end{split}
\end{align}
We consequently introduce the quantities $A_{j}$ and the integrands $f_{j}\colon\rsym\to\R$ via 
\begin{align}\label{eq:Ajdef}
A_{j}:=1+\int_{\Omega}(1+|\sg(\widetilde{u}_{j})|^{2})\dif x\;\;\;\text{and}\;\;\;f_{j}(\xi):=f(\xi)+\frac{1}{2A_{j}j^{2}}(1+|\xi|^{2})
\end{align}
for $\xi\in\rsym$. In order to employ the Ekeland variational principle with respect to sufficiently weak perturbations, we extend the integral functionals corresponding to $f_{j}$ to $\sobo^{-2,1}(\Omega;\R^{n})$ by 
\begin{align}\label{eq:dualextension}
F_{j}[w]:=\begin{cases} \displaystyle \int_{\Omega}f_{j}(\sg(w))\dif x &\;\text{if}\;w\in\mathscr{D}_{j},\\
+\infty&\;\text{if}\;w\in \sobo^{-2,1}(\Omega;\R^{n})\setminus\mathscr{D}_{j}.
\end{cases}
\end{align}
For each $j\in\mathbb{N}$, the functional $F_{j}$ is not identically $+\infty$ on $\sobo^{-2,1}(\Omega;\R^{n})$. The latter space is Banach by Lemma~\ref{lem:negative}~\ref{item:negative1} and, by Lemma~\ref{lem:EkelandLSC} with $\mathfrak{f}=f_{j}$, $q=2$ and $k=2$, $F_{j}$ is lower semicontinuous with respect to the norm topology on $\sobo^{-2,1}(\Omega;\R^{n})$. Moreover, we record
\begin{align*}
F_{j}[\widetilde{u}_{j}] \leq F[\widetilde{u}_{j}]+\frac{1}{2j^{2}} \stackrel{\eqref{eq:almostoptimal3.0}}{\leq} \frac{1}{j^{2}}+\inf F[\mathscr{D}_{j}] \leq \frac{1}{j^{2}}+\inf F_{j}[\sobo^{-2,1}(\Omega;\R^{n})], 
\end{align*}
having used the very definition of $F_{j}$ in the ultimate step. Therefore, Ekeland's variational principle, Proposition~\ref{prop:Ekeland}, provides us with $v_{j}\in\sobo^{-2,1}(\Omega;\R^{n})$ such that 
\begin{align}\label{eq:almostoptimal4}
\begin{split}
&\|v_{j}-\widetilde{u}_{j}\|_{\sobo^{-2,1}(\Omega;\R^{n})}\leq \frac{1}{j},\\ 
&F_{j}[v_{j}] \leq F_{j}[w] + \frac{1}{j}\|v_{j}-w\|_{\sobo^{-2,1}(\Omega;\R^{n})}\qquad\text{for all}\;w\in \sobo^{-2,1}(\Omega;\R^{n}).
\end{split}
\end{align}
We extract from \eqref{eq:almostoptimal4} some routine information by testing with $w=\widetilde{u}_{j}$: 
\begin{align}\label{eq:usefulfjbounds}
\begin{split}
F_{j}[v_{j}] & \stackrel{\eqref{eq:almostoptimal4}_{2}}{\leq} F_{j}[\widetilde{u}_{j}]+\frac{1}{j}\|v_{j}-\widetilde{u}_{j}\|_{\sobo^{-2,1}(\Omega;\R^{n})}\\ 
& \stackrel{\eqref{eq:almostoptimal4}_{1}}{\leq} F[\widetilde{u}_{j}] + \frac{1}{2A_{j}j^{2}}\int_{\Omega}(1+|\sg(\widetilde{u}_{j})|^{2})\dif x +\frac{1}{j^{2}} \\
& \;\,\stackrel{\eqref{eq:almostoptimal3.0}}{\leq} \inf F[\mathscr{D}_{u_{0}}]+\frac{2}{j^{2}}. 
\end{split}
\end{align}
The latter quantity is finite and so, by the very definition of $F_{j}$, $v_{j}\in\mathscr{D}_{j}\subset\sobo^{1,2}(\Omega;\R^{n})$. Moreover, as $v_{j}-u_{j}^{\partial\Omega}\in\sobo_{0}^{1,2}(\Omega;\R^{n})\subset\ld_{0}(\Omega)$,
\begin{align}\label{eq:infimalestimate}
\begin{split}
\inf F[\mathscr{D}_{u_{0}}] & \leq F[u_{0}+(v_{j}-u_{j}^{\partial\Omega})] - F[v_{j}] + F[v_{j}] \\ 
& \leq \Lip(f)\|u_{0}-u_{j}^{\partial\Omega}\|_{\ld(\Omega)} + F[v_{j}] \\ & \!\!\stackrel{\eqref{eq:boundarybound}}{\leq} \frac{1}{8j^{2}} + F_{j}[v_{j}] \stackrel{\eqref{eq:usefulfjbounds}}{\leq} \frac{3}{j^{2}}+\inf F[\mathscr{D}_{u_{0}}]. 
\end{split}
\end{align}
For latter purposes, we record the \emph{perturbed Euler-Lagrange equation}
\begin{align}\label{eq:EL1}
\left\vert \int_{\Omega}\langle f'_{j}(\sg(v_{j})),\sg(\varphi)\rangle\dif x\right\vert \leq \frac{1}{j}\|\varphi\|_{\sobo^{-2,1}(\Omega;\R^{n})}\qquad\text{for all}\;\varphi\in\sobo_{c}^{1,2}(\Omega;\R^{n}). 
\end{align}
This inequality can be obtained by testing $\eqref{eq:almostoptimal4}_{2}$ with $w=v_{j}\pm\theta\varphi$ for $\theta>0$, $\varphi\in\sobo_{c}^{1,2}(\Omega;\R^{n})$, dividing the resulting inequalities by $\theta$ and then sending $\theta\searrow 0$. Moreover, by the linear growth hypothesis \eqref{eq:lingrowth1} and $\tfrac{1}{c}=\min\{\frac{1}{2},c_{1}\}$, we infer from \eqref{eq:usefulfjbounds} that 
\begin{align}\label{eq:uniformboundL1new}
\int_{\Omega}|\sg(v_{j})|\dif x + \frac{1}{A_{j}j^{2}}\int_{\Omega}(1+|\sg(v_{j})|^{2})\dif x \leq c\Big(\inf F[\mathscr{D}_{u_{0}}] + \gamma\mathscr{L}^{n}(\Omega)+\frac{2}{j^{2}}\Big) 
\end{align}
holds for all $j\in\mathbb{N}$. Finally, we note that due to Poincar\'{e}'s inequality on $\ld_{0}(\Omega)$ we obtain
\begin{align}\label{eq:uniformboundL1new1}
\begin{split}
\sup_{j\in\mathbb{N}}\int_{\Omega}|v_{j}|\dif x & \leq \sup_{j\in\mathbb{N}}\Big[ \int_{\Omega}|v_{j}-u_{j}^{\partial\Omega}|\dif x +\int_{\Omega}|u_{j}^{\partial\Omega}|\dif x\Big]  \\ & \leq C \sup_{j\in\mathbb{N}}\Big[ \int_{\Omega}|\sg(v_{j})|\dif x + \|u_{j}^{\partial\Omega}\|_{\ld(\Omega)}\Big] \stackrel{\eqref{eq:boundarybound},\,\eqref{eq:uniformboundL1new}}{<} \infty, 
\end{split}
\end{align}
where $C>0$ is the constant appearing in the requisite Poincar\'{e} inequality. We finally record 
\begin{lemma}\label{lem:convergeinBD}
The sequence $(v_{j})$ as constructed in \eqref{eq:almostoptimal4} possesses a subsequence $(v_{j(l)})\subset (v_{j})$ such that 
\begin{align*}
v_{j(l)}\stackrel{*}{\rightharpoonup} u\qquad\text{in}\;\bd(\Omega)\;\text{as}\;l\to\infty,  
\end{align*}
where $u\in\gm(F;u_{0})$ is the generalised minimiser fixed in the beginning of the section.
\end{lemma}
\begin{proof}
By \eqref{eq:uniformboundL1new} and \eqref{eq:uniformboundL1new1} we conclude that $(v_{j})$ is uniformly bounded in $\bd(\Omega)$, and thus possesses a subsequence $(v_{j(l)})\subset (v_{j})$ such that $v_{j(l)}\stackrel{*}{\rightharpoonup} v$ in $\bd(\Omega)$ as $l\to\infty$ for some $v\in\bd(\Omega)$. Since $\lebe^{1}(\Omega;\R^{n})\hookrightarrow\sobo^{-2,1}(\Omega;\R^{n})$ by Lemma~\ref{lem:negative}\ref{item:negative2}, $v_{j(l)}\to v$ in $\sobo^{-2,1}(\Omega;\R^{n})$. On the other hand, \eqref{eq:approximateconvergencesviscosity}, \eqref{eq:almostoptimal2} and \eqref{eq:almostoptimal4} imply that $v_{j(l)}\to u$ in $\sobo^{-2,1}(\Omega;\R^{n})$. Hence $u=v$, and the proof is complete.  
\end{proof}
\subsection{Preliminary regularity estimates}
To justify the manipulations on the perturbed Euler-Lagrange equations satisfied by the $v_{j}$'s, we now derive \emph{non-uniform} regularity estimates. Since \eqref{eq:EL1} do not display elliptic differential equations (but differential inequalities), the corresponding higher differentiability assertions need to be approached slightly more carefully than for plain viscosity methods:
\begin{lemma}\label{lem:higherregularityapproximate}
Let $f\in\hold^{2}(\rsym)$ satisfy \eqref{eq:lingrowth1} and, for some $\Lambda\in (0,\infty)$, the bound 
\begin{align}\label{eq:generalisedaelliptic}
0<\langle f''(z)\xi,\xi\rangle \leq \Lambda\frac{|\xi|^{2}}{(1+|z|^{2})^{\frac{1}{2}}}\qquad\text{for all}\;z,\xi\in\rsym. 
\end{align}
Define $v_{j}$ for $j\in\mathbb{N}$ by \eqref{eq:almostoptimal4}. Then there holds $v_{j}\in \sobo_{\locc}^{2,2}(\Omega;\R^{n})$. 
\end{lemma}

\begin{proof}
Let $x_{0}\in\Omega$ and $0<r<R<\dista(x_{0},\partial\Omega)$. Also, let $s\in\{1,...,n\}$,  $0<h<\frac{1}{2}(\dista(x_{0},\partial\Omega)-R)$ and pick $\rho\in\hold_{c}^{\infty}(\Omega;[0,1])$ be such that $\mathbbm{1}_{\ball(x_{0},r)}\leq \rho\leq \mathbbm{1}_{\ball(x_{0},R)}$. 
We test the perturbed Euler-Lagrange equation \eqref{eq:EL1} with $\varphi:=\Delta_{s,-h}(\rho^{2}\Delta_{s,h}v_{j})\in \sobo_{c}^{1,2}(\Omega;\R^{n})$. In consequence, integration by parts for difference quotients yields 
\begin{align}\label{eq:approximateelliptic}
\left\vert\int_{\Omega}\langle \Delta_{s,h}f'_{j}(\sg(v_{j})),\sg(\rho^{2}\Delta_{s,h}v_{j})\rangle\dif x \right\vert \leq \frac{1}{j}\|\Delta_{s,-h}(\rho^{2}\Delta_{s,h}v_{j})\|_{\sobo^{-2,1}(\Omega;\R^{n})}. 
\end{align}
We define for $\mathscr{L}^{n}$-a.e. $x\in\ball(x_{0},R)$ bilinear forms $\mathscr{B}_{j,s,h}(x)\colon\rsym\times\rsym\to\R$ by
\begin{align*}
\mathscr{B}_{j,s,h}(x)[\eta,\xi] & := \int_{0}^{1}\langle f''_{j}(\sg(v_{j})(x)+th\Delta_{s,h}\sg(v_{j})(x))\eta,\xi\rangle\dif t,\qquad \eta,\xi\in\rsym.
\end{align*}
Then we note that, because of \eqref{eq:generalisedaelliptic} and the definition of $f_{j}$, 
\begin{align}\label{eq:ellipticityboundapproximate}
(j^{2}A_{j})^{-1}|\xi|^{2}\leq \mathscr{B}_{j,s,h}(x)[\xi,\xi]\leq (\Lambda + (j^{2}A_{j})^{-1})|\xi|^{2} =: C_{j}|\xi|^{2}
\end{align}
for all $\xi\in\rsym$, independently from $s,h$ and $x$. Thus each $\mathscr{B}_{j,s,h}(x)$ is an elliptic bilinear form itself and a suitable version of Young's inequality is available. With this notation, we infer from \eqref{eq:approximateelliptic} by expanding the terms on the left and regrouping
\begin{align*}
\mathbf{I}:=\int_{\Omega}\mathscr{B}_{j,s,h}(x)[\rho\sg(\Delta_{s,h}v_{j}),\rho&\sg(\Delta_{s,h}v_{j})]\dif x  \leq \int_{\Omega}\mathscr{B}_{j,s,h}(x)[\rho\sg(\Delta_{s,h}v_{j}),2\nabla\rho\odot \Delta_{s,h}v_{j}]\dif x \\
& + \frac{1}{j}\|\Delta_{s,-h}(\rho^{2}\Delta_{s,h}v_{j})\|_{\sobo^{-2,1}(\Omega;\R^{n})}\\
& \leq \frac{1}{2}\int_{\Omega}\mathscr{B}_{j,s,h}(x)[\rho\sg(\Delta_{s,h}v_{j}),\rho\sg(\Delta_{s,h}v_{j})]\dif x \\ 
& + \frac{1}{2}\int_{\Omega}\mathscr{B}_{j,s,h}(x)[2\nabla\rho\odot\Delta_{s,h}v_{j},2\nabla\rho\odot\Delta_{s,h}v_{j}]\dif x\\
& + \frac{1}{j}\|\Delta_{s,-h}(\rho^{2}\Delta_{s,h}v_{j})\|_{\sobo^{-2,1}(\Omega;\R^{n})} =: \mathbf{II}+\mathbf{III}+\mathbf{IV}. 
\end{align*}
Absorbing term $\mathbf{II}$ into $\mathbf{I}$, we obtain 
\begin{align}\label{eq:absorbconveniently}
\frac{1}{2j^{2}A_{j}}\int_{\Omega}|\rho\sg(\Delta_{s,h}v_{j})|^{2}\dif x \stackrel{\eqref{eq:ellipticityboundapproximate}}{\leq}\frac{1}{2}\mathbf{I}=\mathbf{I}-\mathbf{II} \leq \mathbf{III}+\mathbf{IV}
\end{align}
and thus need to give bounds on $\mathbf{III}$ and $\mathbf{IV}$. As a consequence of \eqref{eq:ellipticityboundapproximate}, we immediately obtain 
\begin{align*}
\mathbf{III} & \leq 4C_{j}\sup_{\Omega}|\nabla\rho|^{2}\int_{\ball(x_{0},R)}|\Delta_{s,h}v_{j}|^{2}\dif x \leq 4C_{j}(\sup_{\Omega}|\nabla\rho|^{2})\|v_{j}\|_{\sobo^{1,2}(\Omega;\R^{n})}^{2}
\end{align*} 
which is finite due to $v_{j}\in\sobo^{1,2}(\Omega;\R^{n})$. As to term $\mathbf{IV}$, we use Lemma~\ref{lem:negative}\ref{item:negative2} to find by $\lebe^{q}(\Omega;\R^{n})\hookrightarrow\sobo^{-2,1}(\Omega;\R^{n})$ for some $1<q<2$:
\begin{align}\label{eq:IIIestimate}
\begin{split}
\mathbf{IV} & \leq \frac{c(\Omega,q)}{j}\|\partial_{s}(\rho^{2}\Delta_{s,h}v_{j})\|_{\lebe^{q}(\Omega;\R^{n})} \leq \frac{c(\Omega,q,n)}{j}\|\sg(\rho^{2}\Delta_{s,h}v_{j})\|_{\lebe^{q}(\Omega;\rsym)} \\ & \leq \frac{c(\Omega,q,n)}{j}\|\nabla\rho\odot\Delta_{s,h}v_{j}\|_{\lebe^{q}(\Omega;\rsym)} + \frac{c(\Omega,q,n)}{j}\|\rho^{2}\sg(\Delta_{s,h}v_{j})\|_{\lebe^{q}(\Omega;\rsym)}\\
& \leq \frac{c(\Omega,q,n)}{j}(\sup_{\Omega}|\nabla\rho|)\|\nabla v_{j}\|_{\lebe^{2}(\Omega;\R^{n})} + \Big(\frac{1}{4A_{j}j^{2}}\int_{\Omega}|\rho\sg(\Delta_{s,h}v_{j})|^{2}\dif x + c(\Omega,n,j,q) \Big)^{\frac{1}{q}}\\
& \leq \frac{c(\Omega,q,n)}{j}(\sup_{\Omega}|\nabla\rho|)\|\nabla v_{j}\|_{\lebe^{2}(\Omega;\R^{n})} + \frac{1}{4A_{j}j^{2}}\int_{\Omega}|\rho\sg(\Delta_{s,h}v_{j})|^{2}\dif x + c(\Omega,n,j,q)
\end{split}
\end{align}
where $c(\Omega,n,j,q)\geq 1$ and $c(\Omega,q)>0$ are constants. Here we used Korn's inequality in $\sobo_{0}^{1,q}(\Omega;\R^{n})$ in the second and Young's inequality in the penultimate step. The second term on the very right hand side of inequality \eqref{eq:IIIestimate} consequently is absorbed into the very left hand side of \eqref{eq:absorbconveniently}, and then we obtain $\sup_{|h|<\frac{1}{2}(\dista(x_{0},\partial\Omega)-R)}\mathbf{I}<\infty$. Thus, $(\Delta_{s,h}\sg(v_{j}))_{h}$ is uniformly bounded in $\lebe^{2}(\ball(x_{0},r);\rsym)$ and hence $\partial_{s}\sg(v_{j})$ exists in $\lebe^{2}(\ball(x_{0},r);\rsym)$ for each $s\in\{1,...,n\}$. As a consequence, $\partial_{s}v_{j}\in\sobo^{1,2}(\ball(x_{0},r);\R^{n})$ by Korn's inequality. By arbitrariness of $s\in\{1,...,n\}$, $x_{0}\in\Omega$ and $R>0$ sufficiently small, we thus obtain $v_{j}\in\sobo_{\locc}^{2,2}(\Omega;\R^{n})$. The proof is complete. 
\end{proof}

\subsection{Uniform second order estimates}\label{sec:secondorder}
We now turn to uniform estimates (in $j\in\mathbb{N}$) for the viscosity approximating sequence $(v_{j})$. The following result is a key ingredient in the proof of Theorem~\ref{thm:W11reg}, and we single it out as a theorem on its own right:
\begin{theorem}\label{thm:regdual}
Let $f\in\hold^{2}(\rsym)$ satisfy \eqref{eq:lingrowth1}. Moreover, suppose that for some $\Lambda>0$ there holds 
\begin{align}\label{eq:generalisedelliptic}
0<\langle f''(z)\xi,\xi\rangle \leq \Lambda \frac{|\xi|^{2}}{(1+|z|^{2})^{\frac{1}{2}}}\qquad\text{for all}\;z,\xi\in\rsym.
\end{align}
Moreover, let $(v_{j})$ be the viscosity approximation sequence constructed in the previous subsection. Then there exists a constant $c=c(\Lambda,c_{1},c_{2},\gamma,n)>0$ such that for every $x_{0}\in\Omega$, $0<r<1$ with $\ball(x_{0},2r)\Subset\Omega$ and all $j\in\mathbb{N}$ there holds 
\begin{align}\label{eq:dualauxest}
\begin{split}
\sum_{k=1}^{n}&\int_{\ball(x_{0},r)}\langle f''_{j}(\sg(v_{j}))\partial_{k}\sg(v_{j}),\partial_{k}\sg(v_{j})\rangle\dif x  \leq \frac{c}{r^{2}}\int_{\ball(x_{0},2r)}|\sg(v_{j})|\dif x  \\ & +  \frac{c}{A_{j}j^{2}r^{3}} \int_{\ball(x_{0},2r)}(1+|\sg(v_{j})|^{2})\dif x + \frac{c}{jr} \Big(\inf F[\mathscr{D}_{u_{0}}] + \gamma\mathscr{L}^{n}(\Omega)+\frac{2}{j^{2}}\Big).
\end{split}
\end{align}
\end{theorem}
In the following, it is customary to introduce the shorthand notation 
\begin{align*}
\sigma_{j}:=f'_{j}(\sg(v_{j}))\;\;\;\text{and}\;\;\;\mathcal{A}_{j}[\nu;\xi,\eta]:=\langle f''_{j}(\nu)\xi,\eta\rangle,\qquad\nu,\xi,\eta\in\rsym. 
\end{align*}
We begin by collecting the properties of $\sigma_{j}$: 
\begin{lemma}\label{lem:ELbetter}
Let the integrand $f\in\hold^{2}(\rsym)$ satisfy \eqref{eq:lingrowth1} and \eqref{eq:generalisedelliptic} and define $v_{j}$ by \eqref{eq:almostoptimal4}. Then for all $\ell\in\{1,...,n\}$ and $\varphi\in\sobo_{c}^{1,2}(\Omega;\R^{n})$ there holds
\begin{align}\label{eq:betterEL1}
\left\vert\int_{\Omega}\langle\partial_{\ell}\sigma_{j},\sg(\varphi)\rangle\dif x \right\vert \leq \frac{1}{j}\|\varphi\|_{\sobo^{-1,1}(\Omega;\R^{n})}. 
\end{align}
\end{lemma}
\begin{proof}
By Lemma~\ref{lem:higherregularityapproximate}, $v_{j}\in\sobo_{\locc}^{2,2}(\Omega;\R^{n})$. We note that $\partial_{\ell}\sigma_{j} = f''_{j}(\sg(v_{j}))\partial_{\ell}\sg(v_{j})$, and since $\sup_{z\in\rsym}|f''_{j}(z)|<\infty$, $\sigma_{j}\in\sobo_{\locc}^{1,2}(\Omega;\rsym)$. Let $\varphi\in\hold_{c}^{\infty}(\Omega;\R^{n})$. Then $\partial_{\ell}\varphi$ is an admissible competitor in \eqref{eq:EL1} and so, since $\sigma_{j}\in\sobo_{\locc}^{1,2}(\Omega;\rsym)$, 
\begin{align*}
\left\vert\int_{\Omega}\langle \partial_{\ell}\sigma_{j},\sg(\varphi)\rangle\dif x \right\vert = \left\vert\int_{\Omega}\langle \sigma_{j},\sg(\partial_{\ell}\varphi)\rangle\dif x \right\vert \stackrel{\eqref{eq:EL1}}{\leq} \frac{1}{j}\|\partial_{\ell}\varphi\|_{\sobo^{-2,1}(\Omega;\R^{n})} \leq \frac{1}{j}\|\varphi\|_{\sobo^{-1,1}(\Omega;\R^{n})}. 
\end{align*}
Here, the last estimate is valid by Lemma~\ref{lem:negative}~\ref{item:negative2}. Then the case of general $\sobo_{c}^{1,2}(\Omega;\R^{n})$-maps $\varphi$ follows by routine smooth approximation and $\sobo^{1,2}(\Omega;\R^{n})\hookrightarrow\sobo^{-1,1}(\Omega;\R^{n})$. 
\end{proof}
We now come to the 
\begin{proof}[Proof of Theorem~\ref{thm:regdual}] We divide the proof into three steps, and fix $j\in\mathbb{N}$ throughout. 

\emph{Step $1$. Modified perturbed Euler-Lagrange equations.} To establish \eqref{eq:dualauxest}, we shall use the weak Euler--Lagrange equation \eqref{eq:betterEL1} from Lemma~\ref{lem:ELbetter} satisfied by $\sigma_{j}$. Let $k\in\{1,...,n\}$ and let $x_{0}\in\Omega,\,0<r<1$ be such that $\ball(x_{0},2r)\Subset\Omega$. We choose a cut-off function $\rho\in\hold_{c}^{\infty}(\Omega;[0,1])$ such that $\mathbbm{1}_{\ball(x_{0},r)}\leq\rho\leq\mathbbm{1}_{\ball(x_{0},2r)}$ and $|\nabla^{k}\rho|\leq \big(\frac{2}{r}\big)^{k}$ for $k\in\{1,2,3\}$.  Without loss of generality, the interior $\ball'$ of $\spt(\rho)$ is a ball, too. 

Then, since $v_{j}\in\sobo_{\locc}^{2,2}(\Omega;\R^{n})$ by Lemma~\ref{lem:higherregularityapproximate}, we obtain that $\varphi:=\rho^{2}\partial_{k}(v_{j}-a_{j})=:\rho^{2}\partial_{k}w_{j}$ belongs to $\sobo_{c}^{1,2}(\Omega;\R^{n})$ and hence qualifies as a competitor map in \eqref{eq:betterEL1}. Here, $a_{j}\in\mathscr{R}(\Omega)$ is a rigid deformation to be specified later on, and $w_{j}$ is defined in the obvious manner. We write $A=(A^{im})_{i,m=1}^{n}$ for an $(n\times n)$--matrix $A$ and denote the $l$--th component of a vector $u\in\R^{n}$ by $u^{(l)}$. Then applying \eqref{eq:betterEL1} to $\ell=k$ and summing over $k\in\{1,...,n\}$ yields by virtue of Lemma~\ref{lem:negative}~\ref{item:negative2} 
\begin{align}\label{eq:crucialidentity}
\begin{split}
\sum_{k,i,m}\int_{\Omega}(\partial_{k}\sigma_{j}^{im})\sg^{im}(\rho^{2}\partial_{k}w_{j})\dif x & \leq \frac{1}{j}\sum_{k=1}^{n}\|\rho^{2}\partial_{k}w_{j}\|_{\sobo^{-1,1}(\Omega;\R^{n})} \\
&\leq \frac{1}{j}\sum_{k=1}^{n}\|\partial_{k}(\rho^{2}w_{j})-2(\rho\partial_{k}\rho) w_{j}\|_{\sobo^{-1,1}(\Omega;\R^{n})} \\
& \leq \frac{c(n)}{jr}\|w_{j}\|_{\lebe^{1}(\ball';\R^{n})},
\end{split}
\end{align}
where the left-hand sum is taken over all indices $k,i,m\in\{1,...,n\}$. Towards~\eqref{eq:dualauxest}, we note that 
\begin{align}\label{eq:mainrewrite}
\begin{split}
\sum_{k=1}^{n}\int_{\ball(x_{0},r)}\langle f''_{j}(\sg(v_{j}))\partial_{k}\sg(v_{j}),\partial_{k}\sg(v_{j})\rangle\dif x & \leq \sum_{k=1}^{n}\int_{\Omega}\mathcal{A}_{j}[\sg(v_{j});\rho\partial_{k}\sg(v_{j}),\rho\partial_{k}\sg(v_{j})]\dif x \\ & = \sum_{k,i,m}\int_{\Omega}(\partial_{k}\sigma_{j}^{im})\rho^{2}\partial_{k}\sg^{im}(v_{j})\dif x, 
\end{split}
\end{align}
whereby it suffices to estimate the right hand side in view of \eqref{eq:crucialidentity}. 
From \eqref{eq:crucialidentity} we deduce
\begin{align}\label{eq:thescrutinizer}
\begin{split}
2\sum_{k,i,m}\int_{\Omega}(\partial_{k}\sigma_{j}^{im})(\rho^{2}\sg^{im}(\partial_{k}v_{j}))\dif x &  \leq - \sum_{k,i,m} \int_{\Omega} (\partial_{k}\sigma_{j}^{im})((\partial_{i}\rho^{2})\partial_{k}w_{j}^{(m)} +(\partial_{m}\rho^{2})\partial_{k}w_{j}^{(i)})\dif x \\
& + \frac{c(n)}{jr}\|w_{j}\|_{\lebe^{1}(\ball';\R^{n})}\\ 
& = -\sum_{k,i,m} \int_{\Omega} (\partial_{k}\sigma_{j}^{im})((\partial_{i}\rho^{2})\partial_{k}w_{j}^{(m)} + (\partial_{i}\rho^{2})\partial_{m}w_{j}^{(k)})\dif x \\
& +  \sum_{k,i,m}\int_{\Omega}(\partial_{k}\sigma_{j}^{im})((\partial_{i}\rho^{2})\partial_{m}w_{j}^{(k)}+ (\partial_{m}\rho^{2})\partial_{i}w_{j}^{(k)})\dif x\\
& - \sum_{k,i,m}\int_{\Omega}(\partial_{k}\sigma_{j}^{im})( (\partial_{m}\rho^{2})\partial_{k}w_{j}^{(i)}+(\partial_{m}\rho^{2})\partial_{i}w_{j}^{(k)})\dif x \\
& +  \frac{c(n)}{jr}\|w_{j}\|_{\lebe^{1}(\ball';\R^{n})} =: \mathbf{I}+\mathbf{II}+\mathbf{III} +  \frac{c(n)}{jr}\|w_{j}\|_{\lebe^{1}(\ball';\R^{n})}.
\end{split}
\end{align}
\emph{Step $2$. Estimating the terms $\mathbf{I},\mathbf{II}$ and $\mathbf{III}$.} Ad~$\mathbf{I}$ and $\mathbf{III}$. Let us note that, since the indices $i,m$ run over all numbers $1,...,n$ and $\sigma_{j}(x)\in\rsym$ for $\mathscr{L}^{n}$-a.e. $x\in\Omega$, we have $\mathbf{I}=\mathbf{III}$. Moreover, we note that the artificial terms leading to the appearance of $\mathbf{II}$ are just introduced to have the symmetric gradient appearing, that is, terms which are conveniently controllable. In consequence, defining ${}_{j}\Theta_{k} := ({}_{j}\Theta_{k}^{im})_{i,m=1}^{n}$ and ${}_{j}\widetilde{\Theta}_{k} := ({}_{j}\widetilde{\Theta}_{k}^{im})_{i,m=1}^{n}$ with 
\begin{align*}
& {}_{j}\Theta_{k}^{im}:=(\partial_{m}\rho^{2})\sg^{ik}(w_{j}),\\
& {}_{j}\widetilde{\Theta}_{k}^{im}:=2(\partial_{m}\rho)\sg^{ik}(w_{j}),\qquad k,i,m\in\{1,...,n\},
\end{align*}
we find by $\sigma_{j}(x)\in\rsym$ for $\mathscr{L}^{n}$-a.e. $x\in\Omega$ and the definition of the Frobenius inner product on $\R^{n\times n}$ 
\begin{align*}
|\mathbf{I}+\mathbf{III}|  \leq 2|\mathbf{III}| &  \leq 4\left\vert\sum_{k,i,m}\int_{\Omega}(\partial_{k}\sigma_{j}^{im})(\partial_{m}\rho^{2})\sg^{ik}(w_{j})\dif x\right\vert \\ & = 4\left\vert \sum_{k=1}^{n}\int_{\Omega}\langle \partial_{k}\sigma_{j},{}_{j}\Theta_{k}^{\sym}\rangle\dif x \right\vert  =: \mathbf{IV}. 
\end{align*}
We now employ the definition of $\sigma_{j}$ and $\mathcal{A}_{j}[\sg(v_{j});\cdot,\cdot]$. Then we obtain, applying the Cauchy-Schwarz inequality to the bilinear forms $\mathcal{A}_{j}[\sg(v_{j});\cdot,\cdot]$:
\begin{align*}
\mathbf{IV} & \leq \frac{1}{2}\sum_{k=1}^{n}\int_{\Omega}\mathcal{A}_{j}[\sg(v_{j});\rho\partial_{k}\sg(v_{j}),\rho\partial_{k}\sg(v_{j})]\dif x + 8\sum_{k=1}^{n}\int_{\Omega}\mathcal{A}_{j}[\sg(v_{j});{}_{j}\widetilde{\Theta}_{k}^{\sym},{}_{j}\widetilde{\Theta}_{k}^{\sym}]\dif x =: \mathbf{IV}'.
\end{align*}
Appealing to \eqref{eq:generalisedelliptic} and recalling $|\nabla\rho|\leq\frac{2}{r}$, we then further estimate 
\begin{align}\label{eq:IVfinalestimate}
\begin{split}
\mathbf{IV}'& \leq \frac{1}{2}\sum_{k=1}^{n}\int_{\Omega}\mathcal{A}_{j}[\sg(v_{j});\rho\partial_{k}\sg(v_{j}),\rho\partial_{k}\sg(v_{j})]\dif x \\ & + \frac{128\Lambda n^{2}}{r^{2}}\int_{\ball(x_{0},2r)}|\sg(v_{j})|\dif x + \frac{128n^{2}}{r^{2}A_{j}j^{2}}\int_{\ball(x_{0},2r)}|\sg(v_{j})|^{2}\dif x\\
& =: \mathbf{V}_{1}+\mathbf{V}_{2}+\mathbf{V}_{3}. 
\end{split}
\end{align}
Ad $\mathbf{II}$. By symmetry of $\sigma_{j}$, i.e., $\sigma_{j}^{im}(x)=\sigma_{j}^{mi}(x)$ for all $i,m\in\{1,...,n\}$, $j\in\mathbb{N}$ and for $\mathscr{L}^{n}$-a.e. $x\in\Omega$, and a permutation of indices, it suffices to estimate the term 
\begin{align}\label{eq:IIrewriteest}
2|\mathbf{VI}|:= 2\left\vert\sum_{k,i,m}\int_{\Omega}(\partial_{k}\sigma_{j}^{im})(\partial_{i}\rho^{2})(\partial_{m}w_{j}^{(k)})\dif x\right\vert
\end{align}
with an obvious definition of $\mathbf{VI}$.  Integrating by parts twice yields
\begin{align}\label{eq:VIrewriteest}
\begin{split}
\mathbf{VI} = \sum_{k,i,m}\int_{\Omega}&\partial_{k}\sigma_{j}^{im}(\partial_{i}\rho^{2})\partial_{m}w_{j}^{(k)}\dif x  = -\sum_{k,i,m}\int_{\Omega}\sigma_{j}^{im}\partial_{k}((\partial_{i}\rho^{2})\partial_{m}w_{j}^{(k)})\dif x\\
& = -\sum_{k,i,m}\int_{\Omega}\sigma_{j}^{im}((\partial_{ik}\rho^{2})\partial_{m}w_{j}^{(k)}+(\partial_{i}\rho^{2})\partial_{mk}w_{j}^{(k)})\dif x\\
& = \sum_{k,i,m}\int_{\Omega}\partial_{m}(\sigma_{j}^{im}\partial_{ik}\rho^{2})w_{j}^{(k)}+\partial_{m}(\sigma_{j}^{im}\partial_{i}\rho^{2})\partial_{k}w_{j}^{(k)}\dif x \\
& = \sum_{k,i,m}\int_{\Omega}(\partial_{m}\sigma_{j}^{im})(\partial_{ik}\rho^{2})w_{j}^{(k)}+\sigma_{j}^{im}(\partial_{ikm}\rho^{2})w_{j}^{(k)}\dif x \\
& + \sum_{k,i,m}\int_{\Omega}(\partial_{m}\sigma_{j}^{im})(\partial_{i}\rho^{2})\partial_{k}w_{j}^{(k)} + \sigma_{j}^{im}(\partial_{im}\rho^{2})\partial_{k}w_{j}^{(k)}\dif x \\ & =: \mathbf{VI}_{1}+...+\mathbf{VI}_{4}, 
\end{split}
\end{align}
where $\mathbf{VI}_{1},...,\mathbf{VI}_{4}$ are defined in the obvious manner. Note that, by the $\sobo_{\locc}^{2,2}$--regularity of $v_{j}$ and the $\sobo_{\locc}^{1,2}$-regularity of $\sigma_{j}$, this is a valid computation. The crucial point in this calculation is that the only derivatives that apply to $w_{j}$ appear in the form $\partial_{k}w_{j}^{(k)}$ (and are decoupled from the $(i,m)$--components), and summation over $k\in\{1,...,n\}$ corresponds to taking the divergence of $w_{j}$. We define $\psi_{j,k}:=(\psi_{j,k}^{(i)})_{i=1}^{n}:=((\partial_{ik}\rho^{2})w_{j}^{(k)})_{i=1}^{n}\in\sobo_{c}^{1,2}(\Omega;\R^{n})$. Then, with $\di(\sigma_{j})$ denoting the row-wise divergence, we obtain 
\begin{align}\label{eq:VI1est}
\begin{split}
|\mathbf{VI}_{1}| & = \left\vert\sum_{k,i}\int_{\Omega}\di(\sigma_{j}^{(i)})(\partial_{ik}\rho^{2})w_{j}^{(k)}\dif x \right\vert = \left\vert \sum_{k}\int_{\Omega}\langle\di(\sigma_{j}), \psi_{j,k}\rangle\dif x\right\vert \\ 
& = \left\vert \sum_{k=1}^{n}\int_{\Omega}\langle\sigma_{j},\sg(\psi_{j,k})\rangle\dif x\right\vert \stackrel{\eqref{eq:EL1}}{\leq } \frac{1}{j}\sum_{k=1}^{n} \|\psi_{j,k}\|_{\sobo^{-2,1}(\Omega;\R^{n})}\\
&  \leq \frac{n}{j}\|\,|\nabla^{2}\rho^{2}|\,w_{j}\|_{\lebe^{1}(\Omega;\R^{n})} \leq \frac{c(n)}{jr^{2}}\|w_{j}\|_{\lebe^{1}(\ball';\R^{n})} =: \mathbf{VII}.
\end{split}
\end{align}
Here we used Lemma~\ref{lem:negative}\ref{item:negative2} in the penultimate inequality. The term $\mathbf{VI}_{3}$ is treated similarly, now defining $\widetilde{\psi}_{j,k} := ((\partial_{i}\rho^{2})\partial_{k}w_{j}^{(k)})_{i=1}^{n}\in\sobo_{c}^{1,2}(\Omega;\R^{n})$ as $w_{j}\in\sobo_{\locc}^{2,2}(\Omega;\R^{n})$ by Lemma~\ref{lem:higherregularityapproximate}. Then we estimate analogously
\begin{align*}
|\mathbf{VI}_{3}| & \leq \left\vert \sum_{k}\int_{\Omega}\langle\di(\sigma_{j}),\widetilde{\psi}_{j,k}\rangle\dif x \right\vert  \\ &  \leq \frac{1}{j}\sum_{k=1}^{n}\|((\partial_{i}\rho^{2})\partial_{k}w_{j}^{(k)})_{i=1}^{n}\|_{\sobo^{-2,1}(\Omega;\R^{n})} \\
& \leq \frac{1}{j}\sum_{k,i=1}^{n}\|(\partial_{i}\rho^{2})\partial_{k}w_{j}^{(k)}\|_{\sobo^{-2,1}(\Omega)} = \frac{1}{j}\sum_{k,i=1}^{n}\|\partial_{k}((\partial_{i}\rho^{2})w_{j}^{(k)}) - (\partial_{ik}\rho^{2})w_{j}^{(k)}\|_{\sobo^{-2,1}(\Omega)}. 
\end{align*}
At this stage, note that by repeated use of Lemma~\ref{lem:negative}\ref{item:negative2}, 
\begin{align*}
\|\partial_{k}((\partial_{i}\rho^{2})w_{j}^{(k)}) - (\partial_{ik}\rho^{2})w_{j}^{(k)}\|_{\sobo^{-2,1}(\Omega;\R)} & \leq \|(\partial_{i}\rho^{2})w_{j}^{(k)}\|_{\lebe^{1}(\Omega)} + \|(\partial_{ik}\rho^{2})w_{j}^{(k)}\|_{\lebe^{1}(\Omega)}. 
\end{align*}
Hence, we obtain (by possibly enlarging the constant $c(n)>0$ from the estimation of $|\mathbf{VI}_{1}|$)  
\begin{align}\label{eq:VI2est}
|\mathbf{VI}_{3}| \leq \frac{c(n)}{jr^{2}}\|w_{j}\|_{\lebe^{1}(\ball';\R^{n})} = \mathbf{VII}.
\end{align}
We turn to the estimation of $\mathbf{VI}_{2}$ and $\mathbf{VI}_{4}$. We recall that we still have the freedom to choose the rigid deformations $a_{j}$ as they appear in the definition of $w_{j}$. As $\spt(\rho)=\ball'$ is ball\footnote{In view of Poincar\'{e}'s inequality, it would be sufficient to assume that $\spt(\rho)$ is a connected Lipschitz domain.}, we find a constant $C(\ball')>0$ such that for every $v\in\sobo^{1,2}(\ball';\R^{n})$  there exists $a\in\mathscr{R}(\R^{n})$ such that 
\begin{align}\label{eq:poincareproof}
\int_{\ball'}|v-a|\dif x \leq c_{n}r\int_{\ball'}|\sg(v)|\dif x\;\;\text{and}\;\;\int_{\ball'}|v-a|^{2}\dif x \leq c_{n}r^{2}\int_{\ball'}|\sg(v)|^{2}\dif x. 
\end{align}
It is important that for each such $v$ we can choose \emph{one} rigid deformation $a$ to make both inequalities work, and by Lemma~\ref{rem:stability}, this is in fact possible. Accordingly, we choose for each $j\in\mathbb{N}$ some $a_{j}\in\mathscr{R}(\R^{n})$ such that inequality \eqref{eq:poincareproof} holds with $v$ being replaced by $v_{j}$ and with $a$ being replaced by $a_{j}$. Turning to $\mathbf{VI}_{2}$, we go back to the definition of $\sigma_{j}$ and thereby obtain by virtue of Young's inequality and the above Poincar\'{e} inequalities \eqref{eq:poincareproof} that 
\begin{align}\label{eq:VI3est}
\begin{split}
|\mathbf{VI}_{2}| & \;\;\;\leq \sum_{k,i,m}\int_{\Omega}(|f'(\sg(v_{j}))|+\frac{1}{A_{j}j^{2}}|\sg(v_{j})|)\,|(\partial_{ikm}\rho^{2})|\,|w_{j}^{(k)}|\dif x \\
& \;\;\; \leq \frac{c(n)}{r^{3}}\Big(\Lip(f) \int_{\ball'}|w_{j}|\dif x + \frac{1}{2A_{j}j^{2}}\Big(\int_{\ball'}|w_{j}|^{2}\dif x + \int_{\ball'}|\sg(w_{j})|^{2}\dif x\Big)\Big)\\
& \stackrel{0<r<1}{\leq} \frac{c(n)\max\{\Lip(f),1\}}{r^{3}}\Big( r\int_{\ball'}|\sg(v_{j})|\dif x + \frac{1+r^{2}}{2A_{j}j^{2}}\int_{\ball'}|\sg(v_{j})|^{2}\dif x\Big).
\end{split}
\end{align}
As to $\mathbf{VI}_{4}$, we note that since $(\sg(v_{j}))$ is uniformly bounded in $\lebe^{1}(\Omega;\rsym)$ by \eqref{eq:uniformboundL1new}, so is $(\di(v_{j}))$ in $\lebe^{1}(\Omega)$. We then estimate, using the pointwise bound\footnote{Note that rigid deformations have zero divergence.} $|\di(w_{j})|\leq |\sg(v_{j})|$ and \eqref{eq:uniformboundL1new},  
\begin{align}\label{eq:VI4est}
\begin{split}
\mathbf{VI}_{4} & \leq \frac{c(n)}{r^{2}}\int_{\Omega}|\sigma_{j}|\,|\di(w_{j})|\dif x \\ & \leq \frac{c(n)\max\{\Lip(f),1\}}{r^{2}}\Big(\int_{\ball'}|\sg(v_{j})|+\frac{1}{A_{j}j^{2}}|\sg(v_{j})|^{2}\dif x\Big).
\end{split}
\end{align}
By our choice of $a_{j}$ and \eqref{eq:uniformboundL1new}, $\mathbf{VII}$ can now be estimated by 
\begin{align}\label{eq:VIIest}
\mathbf{VII} \leq \frac{c(n)}{jr^{2}}\|w_{j}\|_{\lebe^{1}(\ball';\R^{n})} \leq \frac{c(n)}{jr} \Big(\inf F[\mathscr{D}_{u_{0}}] + \gamma\mathscr{L}^{n}(\Omega)+\frac{2}{j^{2}}\Big). 
\end{align}
\emph{Step $3$. Conclusion.} We now gather estimates and start from \eqref{eq:mainrewrite} to find 
\begin{align*}
\sum_{k=1}^{n}\int_{\Omega}&\mathcal{A}_{j}[\sg(v_{j});\rho\partial_{k}\sg(v_{j}),\rho\partial_{k}\sg(v_{j})]\dif x  \stackrel{\eqref{eq:mainrewrite}}{\leq} \mathbf{I} + \mathbf{II} + \mathbf{III} + \frac{c(n)}{jr}\|w_{j}\|_{\lebe^{1}(\ball';\R^{n})} \\
& \stackrel{\eqref{eq:IVfinalestimate},\,\eqref{eq:VIrewriteest}}{\leq} (\mathbf{V}_{1}+\mathbf{V}_{2}+\mathbf{V}_{3}) + \mathbf{VI}_{1} + ... + \mathbf{VI}_{4} + \frac{c(n)}{j}\|\sg(w_{j})\|_{\lebe^{1}(\ball';\R^{n})}, 
\end{align*}
whence we absorb $\mathbf{V}_{1}$ into the left side of the previous inequality. We then succesively combine \eqref{eq:IVfinalestimate}, \eqref{eq:VI1est}--\eqref{eq:VIIest} with \eqref{eq:uniformboundL1new} to obtain via $0<r<1$
\begin{align*}
\sum_{k=1}^{n}\int_{\Omega}& \mathcal{A}_{j}[\sg(v_{j});\rho\partial_{k}\sg(v_{j}),\rho\partial_{k}\sg(v_{j})]\dif x \leq \frac{c}{r^{2}}\int_{\ball(x_{0},2r)}|\sg(v_{j})|\dif x  +  \frac{c}{A_{j}j^{2}r^{3}} \times \\ & \times \int_{\ball(x_{0},2r)}(1+|\sg(v_{j})|^{2})\dif x + \frac{c}{jr} \Big(\inf F[\mathscr{D}_{u_{0}}] + \gamma\mathscr{L}^{n}(\Omega)+\frac{2}{j^{2}}\Big), 
\end{align*}
where we track constants to find that $c=c(\Lip(f),\Lambda,n,\gamma,c_{1})>0$. Since $\Lip(f)$ only depends on $c_{2},\gamma$ by Lemma~\ref{lem:boundbelow}, this immediately gives  \eqref{eq:dualauxest} by \eqref{eq:mainrewrite}, and the proof is hereby complete. 
\end{proof}

\subsection{Proof of Theorem~\ref{thm:W11reg}}\label{sec:W11regproof}
Based on Theorem~\ref{thm:regdual}, we can proceed to the proof of Theorem~\ref{thm:W11reg}. It needs to be noted that the \emph{second} order estimate given in \eqref{eq:dualauxest} is the decisive ingredient which we lacked in \cite{GK1}, and in the following we demonstrate how \eqref{eq:dualauxest} leads to a Sobolev regularity improvement. Here, we are led by the ideas exposed in \cite{Bi1,BS1} for the gradient case.

\begin{proof}[Proof of Theorem~\ref{thm:W11reg}]
Let $u\in\gm(F;u_{0})$ be given and let $\ball(x_{0},5r)\subset\Omega$ be an open ball. In this situation, $u$ is a local generalised minimiser, which in particular implies that $\overline{F}_{u}[u;\ball(x_{0},5r)]\leq\overline{F}_{u}[v;\ball(x_{0},5r)]$ for all $v\in\bd(\ball(x_{0},5r))$. We now denote $(v_{j})$ the specific Ekeland viscosity approximation sequence as constructed in \eqref{eq:almostoptimal4}ff., with $\Omega$ being replaced by $\ball(x_{0},5r)$ and $u_{0}$  being replaced by $u|_{\ball(x_{0},5r)}$. Lemma~\ref{lem:convergeinBD} then implies that there exists a subsequence $(v_{j(l)})\subset (v_{j})$ such that $v_{j(l)}\stackrel{*}{\rightharpoonup} u$ in $\bd(\ball(x_{0},5r))$ as $l\to\infty$. 

We begin with $n\geq 3$. Since in particular $1<a<2$ in the present situation, we introduce the auxiliary convex function $
V_{a}(\xi):=(1+|\xi|^{2})^{\frac{2-a}{4}}$, $\xi\in\rsym$. Recalling $(v_{j(l)})\subset\sobo_{\locc}^{2,2}(\ball(x_{0},5r);\R^{n})$ from Lemma~\ref{lem:higherregularityapproximate} and differentiating $V_{a}(\sg(v_{j(l)}))$, we obtain for all $k\in\{1,...,n\}$
\begin{align*}
|\partial_{k}V_{a}(\sg(v_{j(l)}))|^{2} & \leq \Big(\frac{2-a}{2}\Big)^{2}|\partial_{k}\sg(v_{j(l)})|^{2}\,|\sg(v_{j(l)})|^{2}(1+|\sg(v_{j(l)})|^{2})^{\frac{-2-a}{2}} \\ & \leq c(a)\frac{|\partial_{k}\sg(v_{j(l)})|^{2}}{(1+|\sg(v_{j(l)})|^{2})^{\frac{a}{2}}}. 
\end{align*}
Therefore, we find by the previous inequality, the lower bound in \eqref{eq:ellipticity} and Theorem~\ref{thm:regdual}:
\begin{align}\label{eq:elscrutinho}
\begin{split}
\|V_{a}(\sg(v_{j(l)}))\|_{\lebe^{\frac{2n}{n-2}}(\ball(x_{0},r))}^{2} & \leq c(n)\Big(\|\nabla (V_{a}(\sg(v_{j(l)})))\|_{\lebe^{2}(\ball(x_{0},r))}^{2}\Big. \\ & \Big. \;\;\;\;\;\;\;\;\;\;\;\;\;\;\;\;\;\;\;\;\;\;\;\;+ \frac{1}{r^{2}}\int_{\ball(x_{0},r)}|V_{a}(\sg(v_{j(l)}))|^{2}\dif x\Big) \\ 
& \leq c(n,a)\Big(\int_{\ball(x_{0},r)}\frac{|\nabla(\sg(v_{j(l)}))|^{2}}{(1+|\sg(v_{j(l)})|^{2})^{\frac{a}{2}}}\dif x \Big. \\ & \;\;\;\; \;\;\;\; \;\;\;\; \;\;\;\; \;\;\;\; \Big.+\frac{1}{r^{2}} \int_{\ball(x_{0},r)}(1+|\sg(v_{j(l)})|^{2})^{\frac{1}{2}}\dif x\Big) \\
& \!\!\!\!\!\!\!\!\!\!\!\stackrel{\eqref{eq:dualauxest},\,\eqref{eq:uniformboundL1new}}{\leq} \frac{c}{r^{2}}\int_{\ball(x_{0},2r)}|\sg(v_{j(l)})|\dif x \\ & + \frac{c}{A_{j(l)}j(l)^{2}r^{3}}\times \int_{\ball(x_{0},2r)}(1+|\sg(v_{j(l)})|^{2})\dif x \\ &+\frac{c}{j(l)r} \Big(\overline{F}_{u}[u;\ball(x_{0},5r)] + \gamma\mathscr{L}^{n}(\ball(x_{0},5r))+\frac{2}{j(l)^{2}}\Big) \\ & + cr^{n-2} =: \mathbf{I}_{l} + ... + \mathbf{IV}_{l},
\end{split}
\end{align}
where $c=c(n,a,\lambda,\Lambda,\gamma,c_{1},c_{2})>0$. Here, the first estimate is valid by the scaled Sobolev inequality, whereas we have used $\sqrt{1+t^{2}}\leq 1+t$ for the ultimate one. As a consequence of the definition of $V_{a}$, we find 
\begin{align}\label{eq:trivialVaest}
|z|^{\frac{2-a}{2}}\leq V_{a}(z)\qquad\text{for all}\;z\in\rsym. 
\end{align}
This yields local uniform boundedness of $(\sg(v_{j(i)}))$ in $\lebe^{q}(\ball(x_{0},r);\rsym)$ for $q=\frac{2-a}{n-2}n$, and the latter number satisfies $q>1$ if and only if $1<a<1+\frac{2}{n}$, which is the standing assumption of Theorem~\ref{thm:W11reg}. Let us note in advance that \eqref{eq:infimalestimate} implies that $\lim_{l\to\infty}\mathbf{II}_{l}=0$, whereas $\lim_{l\to\infty}\mathbf{III}_{l}=0$ holds trivially. Now consider the function $\Psi_{q}(t):=|t|^{q}$ for $q>1$. Then $\Psi_{q}^{\infty}(t)=\infty$ for if $|t|>0$.  Since $\sg(v_{j(l)})\mathscr{L}^{n}\mres\ball(x_{0},r)\stackrel{*}{\rightharpoonup}\E u\mres\ball(x_{0},r)$, we obtain as a consequence of \textsc{Reshetnyak}'s theorem, Lemma \ref{lem:reshetnyak}, and the notation adopted in \eqref{eq:perspectivefunction} afterwards with $\nu=(\mathscr{L}^{n},\E u)$, 
\begin{align}\label{eq:qestimate}
\begin{split}
\int_{\ball(x_{0},r)}\Psi_{q}(\mathscr{E}u)\dif x + \int_{\ball(x_{0},r)}\Psi_{q}^{\infty}&\Big(\frac{\dif\E^{s}u}{\dif |\!\E^{s} u|}\Big)\dif |\!\E^{s} u| = \int_{\ball(x_{0},r)}\Psi_{q}^{\#}\Big(\frac{\dif\nu}{\dif|\nu|}\Big)\dif|\nu| \\ & \leq \liminf_{l\to\infty}\int_{\ball(x_{0},r)}\Psi_{q}^{\#}(1,\sg(v_{j(l)}))\dif x \\ & \!\!\!\!
\stackrel{\eqref{eq:elscrutinho}}{\leq} c\Big(\frac{|\!\E u|(\ball(x_{0},5r))}{r^{2}}+r^{n-2} \Big)^{\frac{n}{n-2}}.
\end{split}
\end{align}
Since the very right hand side is finite, we conclude that $\E^{s}u$ vanishes on $\ball(x_{0},r)$. By arbitrariness of $\ball(x_{0},r)$, we moreover infer that $\E^{s}u\equiv 0$ in $\Omega$ and so $u\in\ld(\Omega)$ together with $\sg(u)=\mathscr{E}u$. Moreover, by Korn's inequality, $\nabla u\in \lebe^{q}(\ball(x_{0},r);\R^{n\times n})$. To obtain the precise form of  \eqref{eq:Kornmain}, we choose a rigid deformation $\pi_{u}\in\mathscr{R}(\ball(x_{0},r))$ such that 
\begin{align*}
\|\nabla u\|_{\lebe^{q}(\ball(x_{0},r);\R^{n\times n})}& \;\,\leq \|\nabla (u-\pi_{u})\|_{\lebe^{q}(\ball(x_{0},r);\R^{n\times n})} + \|\nabla \pi_{u}\|_{\lebe^{q}(\ball(x_{0},r);\R^{n\times n})}\\ 
& \stackrel{\eqref{eq:rigidscaling}}{\leq} c\Big( \|\sg(u)\|_{\lebe^{q}(\ball(x_{0},r);\rsym)} + r^{\frac{n}{q}-1}\dashint_{\ball(x_{0},r)}|\pi_{u}|\dif x \Big)\\
& \stackrel{\eqref{eq:Lpstability}}{\leq} c\Big( \|\sg(u)\|_{\lebe^{q}(\ball(x_{0},r);\rsym)} + r^{\frac{n}{q}-1}\dashint_{\ball(x_{0},r)}|u|\dif x\Big)\\
& \!\stackrel{\eqref{eq:qestimate}}{\leq} c\Big(\Big(\frac{|\!\E u|(\ball(x_{0},5r))}{r^{2}}+r^{n-2}\Big)^{\frac{1}{2-a}} + r^{\frac{n}{q}-1} \dashint_{\ball(x_{0},r)}|u|\dif x \Big).
\end{align*}
Dividing the previous inequality by $r^{\frac{n}{q}}=r^{\frac{n-2}{2-a}}$, we obtain 
\begin{align*}
\Big(\dashint_{\ball(x_{0},r)}|\nabla u|^{q}\dif x \Big)^{\frac{1}{q}} \leq c\Big(\Big(1+\dashint_{\ball(x_{0},5r)}|\!\E u|\Big)^{\frac{1}{2-a}} + \frac{1}{r} \dashint_{\ball(x_{0},r)}|u|\dif x \Big).
\end{align*}
This is \eqref{eq:Kornmain} and the proof is complete for if $n\geq 3$. Now let $n=2$. As above, $(V_{a}(\sg(v_{j(l)})))$ is locally uniformly bounded in $\sobo^{1,2}(\ball(x_{0},5r);\R^{n})$ and thus, using \textsc{Trudinger}'s embedding $\sobo^{1,n}(\Omega)\hookrightarrow\exp\lebe^{\frac{n}{n-1}}(\Omega)$, \eqref{eq:elscrutinho} equally yields 
\begin{align*}
\|V_{a}(\sg(v_{j(l)}))\|_{\exp\lebe^{\frac{n}{n-1}}(\ball(x_{0},r))} \leq \sqrt{\mathbf{I}_{l}+...+\mathbf{IV}_{l}}.
\end{align*}
Working with $\Psi(t)=\exp(t^{\frac{n}{n-1}\frac{2-a}{2}})=\exp(t^{2-a})$ instead of $\Psi_{q}$ from above, we similarly conclude that $u\in E^{1}\exp\lebe^{2-a}(\ball(x_{0},r))$. We then employ \textsc{Cianchi}'s inequality from Lemma~\ref{lem:Cianchi}\ref{item:Cianchi1} with $\beta=2-a(>0)$ and hereafter $\frac{\beta}{\beta+1}=\frac{2-a}{3-a}$. In consequence, 
\begin{align*}
\|\nabla u\|_{\exp\lebe^{\frac{2-a}{3-a}}(\ball(x_{0},r))} & \leq \|\nabla (u-\pi_{u})\|_{\exp\lebe^{\frac{2-a}{3-a}}(\ball(x_{0},r))} + \|\nabla\pi_{u}\|_{\exp\lebe^{\frac{2-a}{3-a}}(\ball(x_{0},r);\R^{n\times n})}\\
& \leq c\Big(\|\sg(u)\|_{\exp\lebe^{2-a}(\ball(x_{0},r);\rsym)}+ \frac{1}{r}\dashint_{\ball(x_{0},r)}|u|\dif x\Big) \\ 
& \leq c\Big(\Big(1+\dashint_{\ball(x_{0},5r)}|\!\E u|\Big)^{\frac{1}{2-a}} + \frac{1}{r}\dashint_{\ball(x_{0},r)}|u|\dif x\Big),
\end{align*}
and the proof is complete. 
\end{proof}

\subsection{Selected implications}\label{sec:cors}
We now collect some consequences of the results established above and particularly improve the results from \cite{GK1}. We begin by strengthening \cite[Cor.~3.8]{GK1}, justifying the second arrow from below in Figure~\ref{fig:1}.
\begin{corollary}[Existence of second derivatives]\label{cor:2ndorder} Let $n\geq 2$ and suppose that $f\in\hold^{2}(\rsym)$ satisfies \eqref{eq:lingrowth1} and \eqref{eq:ellipticity} for some $1<a<\frac{n}{n-1}$. Then there holds $\gm_{\locc}(F)\subset\sobo_{\locc}^{2,q}(\Omega;\R^{n})$ for any $1<q<n\frac{2-a}{n-a}$. 
\end{corollary}
\begin{proof}
Let $u\in\gm_{\locc}(F)$ and let $\ball(x_{0},r)\Subset\Omega$ be an open ball. By Theorem~\ref{thm:W11reg} and its proof, we have $\gm_{\locc}(F)\subset\sobo_{\locc}^{1,p}(\Omega;\R^{n})$ for any $1<p<\infty$ if $n=2$ and any $1<p<\frac{2-a}{n-2}n$ if $n\geq 3$. Let $(v_{j})$ be the Ekeland viscosity approximation sequence constructed in \eqref{eq:almostoptimal4} with $\Omega$ being replaced by $\ball(x_{0},r)$ and $u_{0}$ being replaced by $u|_{\ball(x_{0},r)}$. Then we record, using Young's inequality with exponents $\frac{2}{q}$ and $\frac{2}{2-q}$ for some $1\leq q<2$,
\begin{align*}
\int_{\ball(x_{0},r)}|\nabla\sg(v_{j})|^{q}\dif x & \leq \frac{q}{2}\int_{\ball(x_{0},r)}\frac{|\nabla\sg(v_{j})|^{2}}{(1+|\sg(v_{j})|^{2})^{\frac{a}{2}}}\dif x +\frac{2-q}{2}\int_{\ball(x_{0},r)}(1+|\sg(v_{j})|^{2})^{\frac{aq}{4}\frac{2}{2-q}}\dif x. 
\end{align*}
The first term is uniformly controlled by Theorem~\ref{thm:regdual}. If $n=2$, then the second term is uniformly bounded in $j\in\mathbb{N}$ regardless of $1\leq q <2$ as $\sup_{j\in\mathbb{N}}\|\sg(v_{j})\|_{\lebe^{p}(\ball(x_{0},r);\rsym)}<\infty$ for all $1<p<\infty$ (see the proof of Theorem~\ref{thm:W11reg}). If $n\geq 3$ and $1<a<\frac{n}{n-1}$, then $1<a<1+\frac{2}{n}$, and again by the proof of Theorem~\ref{thm:W11reg}, the second term is uniformly bounded in $j\in\mathbb{N}$ if  
\begin{align}\label{eq:adetermine}
a\frac{q}{2-q} < \frac{2-a}{n-2}n\;\;\;\text{that is,}\;\;\;q<n\frac{2-a}{n-a}=:\widetilde{q}(n).
\end{align}
Note that $\widetilde{q}(n)> 1$ if and only if $1< a < \frac{n}{n-1}$. Hence, $(v_{j})$ is locally uniformly bounded in $\sobo^{2,q}$ for $1<q<\frac{n(2-a)}{n-a}$. From here the result follows in the same way as in the proof of Theorem~\ref{thm:W11reg}, again using Korn's inequality. 
\end{proof}
Compared with \cite{GK1}, we have now established that for the ellipticity regime $1<a<\frac{n}{n-1}$, \emph{all} generalised minima possess second derivatives in some $\lebe_{\locc}^{q}$, $q>1$. An easy application of the measure density lemma \cite[Prop.~2.7]{Giusti} yields the following
\begin{corollary}[Singular set bounds]\label{cor:dimbound}
Let $f\in\hold^{2}(\rsym)$ satisfy \eqref{eq:lingrowth1} and \eqref{eq:ellipticity} for some $1<a<1+\frac{2}{n}$. For a given map $v\in\bd_{\locc}(\Omega)$, put 
\begin{align*}
\Sigma_{v}:=\left\{x\in\Omega\colon \limsup_{R\searrow 0}\left[\dashint_{\ball(x,R)}|\mathscr{E}v-z|\dif\mathscr{L}^{n}+\frac{|\E^{s}v|(\ball(x,R))}{\mathscr{L}^{n}(\ball(x,R))}\right]>0\;\text{for all}\;z\in\rsym \right\}.
\end{align*} 
Then the following holds:
\begin{enumerate}
\item If $n=2$ and $1<a<2$, then \emph{any} $u\in\gm_{\locc}(F)$ satisfies $\dim_{\mathscr{H}}(\Sigma_{u})=0$. 
\item If $n\geq 3$ and $1<a<\frac{n}{n-1}$, then \emph{any} $u\in\gm_{\locc}(F)$ satisfies $\dim_{\mathscr{H}}(\Sigma_{u})\leq n\frac{n-2}{n-a}$.
\end{enumerate}
\end{corollary}
We conclude this section by describing the structure of $\gm(F;u_{0})$ and begin with 
\begin{corollary}[Uniqueness modulo elements of $\mathscr{R}(\Omega)$]\label{cor:uniqueness0}
Let $\Omega\subset\R^{n}$ be an open, bounded and connected set with Lipschitz boundary and $u_{0}\in\ld(\Omega)$. In the situation of Theorem~\ref{thm:W11reg}, generalised minimisers are unique up to rigid deformations, that is, 
\begin{align*}
u,v\in\gm(F;u_{0})\Longrightarrow \exists\pi\in\mathscr{R}(\Omega)\colon\;u=v+\pi. 
\end{align*}
\end{corollary}
\begin{proof}
By Theorem~\ref{thm:W11reg}, $\gm(F;u_{0})\subset\ld(\Omega)$. Now suppose that $u,v\in\gm(F;u_{0})$ are two generalised minima such that $\sg(u)\neq\sg(v)$ $\mathscr{L}^{n}$-a.e.. Then, by strict convexity of $f$ and convexity of $f^{\infty}$, 
\begin{align*}
\overline{F}_{u_{0}}\Big[\frac{u+v}{2}\Big] & < \sum_{w\in\{u,v\}}\frac{1}{2}\int_{\Omega}f(\sg(w))\dif x + \frac{1}{2}\int_{\partial\Omega}f^{\infty}(\trace(u_{0}-w)\odot\nu_{\partial\Omega})\dif\mathscr{H}^{n-1}. 
\end{align*}
For $u,v\in\gm(F;u_{0})$, the right-hand side equals $\min \overline{F}_{u_{0}}[\bd(\Omega)]$ which consequently yields a contradiction. Hence, $\sg(u-v)=0$ $\mathscr{L}^{n}$-a.e., and since $\Omega$ is connected, this implies that $u=v+\pi$ for some $\pi\in\mathscr{R}(\Omega)$.
\end{proof}
As in the $\bv$-case, Corollary~\ref{cor:uniqueness0} cannot be improved to yield full uniqueness. 
To this end, it is important to require a suitable variant of strict convexity on the recession function $f^{\infty}$; note that $f^{\infty}$ is positively $1$-homogeneous and hence not strictly convex. In this respect, the relevant concept is as follows (also see \cite[Sec.~4.5]{Schmidt2}): We say that a function $g\colon \rsym\to\R$ has \emph{strictly convex sublevel sets} provided for each $t\in\R$ the set $\Gamma_{t}(g):=\{z\in\rsym\colon\;g(z)<t\}$ is bounded, convex and if $z_{1},z_{2}\in \partial\Gamma_{t}(g)$, then $\lambda z_{1}+(1-\lambda)z_{2}\notin\partial\Gamma_{t}(g)$ for any $0<\lambda<1$. 
\begin{corollary}[Uniqueness and structure of $\gm(F;u_{0})$]\label{cor:uniqueness}
Let $\Omega\subset\R^{n}$ be an open, bounded, connected set with Lipschitz boundary such that for any fixed $a\in\R$ there holds 
\begin{align*}
\mathscr{H}^{n-1}(\{x\in\partial\Omega\colon\;x_{i}=a\})= 0\qquad\text{for all}\;i\in\{1,...,n\}.
\end{align*}
In the situation of Corollary~\ref{cor:uniqueness0}, suppose that the map $f_{\nu}^{\infty}\colon\R^{n}\ni z\mapsto f^{\infty}(z\odot\nu)$ has strictly convex sublevel sets for all $\nu\in\R^{n}\setminus\{0\}$. Then there exists a generalised minimiser $u\in\gm(F;u_{0})$ and a rigid deformation $\pi\in\mathscr{R}(\Omega)$ such that 
\begin{align}\label{eq:repGM}
\gm(F;u_{0})=\big\{ u + \lambda\pi\colon\;\lambda\in[-1,1]\}. 
\end{align}
Finally, if there exists a generalised minimiser $u$  which attains the boundary values $\trace_{\partial\Omega}(u_{0})$ $\mathscr{H}^{n-1}$-a.e. on $\partial\Omega$, then $\gm(F;u_{0})=\{u\}$.  
\end{corollary}
The condition on $f_{\nu}^{\infty}$ to have strictly convex sublevel sets is satisfied if, e.g., $f$ is spherically symmetric, ruling out that $(f^{\infty})^{-1}(\{1\})$ contains any line segments of positive length. Corollary~\ref{cor:uniqueness} follows from Corollary~\ref{cor:uniqueness0} similarly as in the $\bv$-case, cf.~\cite[Thm.~1.16]{BS1}, but is technically more demanding; for the reader's convenience, the appendix A, Section~\ref{sec:uniquenessappendix}, includes the precise reasoning with emphasis on the two-dimensional case. 

\section{A family of convolution-type Poincar\'{e} inequalities}\label{sec:poincare} 
Approaching Theorem~\ref{thm:PR}, we pause to provide a family of convolution inequalities to instrumentally enter the partial regularity proof below. We believe that the result might be of independent interest, and thus state selected versions thereof in the end of the section.
\begin{proposition}\label{prop:convest}
Let $\lambda>1$ and let $V(z):=\sqrt{1+|z|^{2}}-1$ be the auxiliary reference integrand as usual. Then there exists a constant $c=c(n,\lambda)>0$ such that the following holds: For every open and bounded Lipschitz domain $\Omega\subset\R^{n}$, $u\in\bd_{\locc}(\R^{n})$ and numbers $\varepsilon,L>0$ there holds 
\begin{align}\label{eq:Poincaremain}
\int_{\Omega}V(L(u-\rho_{\varepsilon}*u))\dif x \leq c\max\{(L\varepsilon),(L\varepsilon)^{2}\} \int_{\overline{\Omega+\ball(0,\lambda\sqrt{n}\varepsilon)}}V(\E u), 
\end{align}
where $\rho\colon \R^{n}\to \R_{\geq 0}$ is an arbitrary standard mollifier in the sense of Section~\ref{sec:funofmeas}.
\end{proposition}
Before passing to the proof of the preceding proposition, let us remark that \eqref{eq:Poincaremain} cannot be established as in the full gradient case, cf.~\cite[Lemma 5.3]{AG}. Namely, if we wish to obtain \eqref{eq:Poincaremain} for $u\in\bv(\R^{n};\R^{N})$ with the symmetric gradient on the right-hand side being replaced by the full gradient, it suffices to invoke the fundamental theorem of calculus in conjunction with Jensen's inequality. In view of \eqref{eq:Poincaremain}, \textsc{Ornstein}'s Non-Inequality forces us to avoid the appearance of the full gradient on the right-hand side. Upon localisation, a slightly weaker result can be readily obtained as follows: Invoking the \textsc{Smith} representation formula \eqref{eq:Smith} and then arguing as in the full gradient case, we may conclude\footnote{Namely, express the difference $u-\rho_{\varepsilon}*u$ by the convolution integrals emerging from \eqref{eq:Smith} and then use the embedding of $\bd(\R^{n})$ into $\sobo_{\locc}^{\alpha,1}(\R^{n};\R^{n})$ or $\sobo^{s,(n-1+s)}(\R^{n};\R^{n})$ for $0<s<1$ (cf.~\cite{GK2}). Since suitable fractional potentials of $\sobo^{\alpha,1}$-maps can be controlled conveniently, this allows to arrive at the claimed estimate for all $0<\beta<1$.} that for any $0<\beta<1$ (but \emph{not} for $\beta=1$) there exists $C=C(\beta,\text{diam}(\Omega))>0$ with
\begin{align}\label{eq:betaestimate}
\|V(L(u-\rho_{\varepsilon}))\|_{\lebe^{1}(\Omega)}\leq C\min\{L\varepsilon^{\beta},L^{2}\varepsilon^{2\beta}\}|V(\E u)|(\overline{\Omega+\ball(0,\varepsilon)})
\end{align}
for all $u\in\bd(\Omega)$ and $L>0$. However, this is neither optimal nor good enough for deriving the requisite decay estimate in Section~\ref{sec:PR}; see the proof of Proposition~\ref{prop:main1} and Remark~\ref{rem:discussexponent} afterwards. 
\begin{proof}[Proof of Proposition~\ref{prop:convest}] The proof consists in four main steps. After giving the geometric setup in a first step, we establish a preliminary Poincar\'{e}-type inequality involving the reference integrand $V$ in the second step. Then we globalise by a covering argument with respect to cubes having edgepoints contained in a certain lattice, depending on the parameters $\varepsilon$ and $\lambda$. Lastly, we smoothly approximate to conclude the full claim.
\begin{figure}
\centering
\begin{tikzpicture}
\draw [<->] (1.4, -2.5) -- (2.1,-2.5);
\node at (1.75,-2.9) {$\varepsilon_{\lambda}$};
\node at (-2.0,-2.6) {$\Gamma_{\varepsilon_{\lambda}}:=\varepsilon_{\lambda}\mathbb{Z}^{n}$};
\draw [-] (-0.7, -0.0) -- (-0.7,-0.7);
\draw [-] (-0.7, -0.7) -- (-0.0,-0.7);
\draw [-] (-0.0, -0.7) -- (-0.0,-0.0);
\draw [-] (-0.7, -0.0) -- (-0.0,-0.0);
\draw [-] (-2.1, -2.1) -- (-2.1,1.4);
\draw [-] (-2.1, 1.4) -- (1.4,1.4);
\draw [-] (1.4,1.4) -- (1.4,-2.1);
\draw [-] (-2.1,0.7) -- (-1.4,0.7);
\draw [-] (-1.4,0.7) -- (-1.4,1.4);
\draw [-] (-1.4,0.7) -- (-0.7,0.7);
\draw [-] (-0.7,0.7) -- (-0.7,1.4);
\draw [-] (-0.7,0.7) -- (-0.0,0.7);
\draw [-] (-0.0,0.7) -- (-0.0,1.4);
\draw [-] (1.4,-2.1) -- (-2.1,-2.1);
\draw [-] (-1.4, 0.7) -- (-1.4,-0.7);
\draw [-] (-2.1, -0.7) -- (-1.4,-0.7);
\draw [-] (-2.1, -0.0) -- (-1.4,-0.0);
\node [color=gray] at (0.3,1.0) {\textbullet};
\node [color=gray] at (0.6,1.0) {\textbullet};
\node [color=gray] at (0.9,1.0) {\textbullet};
\node [color=gray] at (-1.75,-1.0) {\textbullet};
\node [color=gray] at (-1.75,-1.3) {\textbullet};
\node [color=gray] at (-1.75,-1.6) {\textbullet};
\node [black] at (-1.75,2.0) {$Q^{(1)}$};
\node [black] at (-1.05,2.0) {$Q^{(2)}$};
\node [black] at (-0.35,2.0) {$Q^{(3)}$};
\draw [->,dotted] (-2.6,0.4) -- (-0.35,-0.35);
\draw [->,dotted] (-1.75,1.8) -- (-1.75,1.2);
\draw [->,dotted] (-1.05,1.8) -- (-1.05,1.2);
\draw [->,dotted] (-0.35,1.8) -- (-0.35,1.2);

\draw[step=.7cm,gray,dotted] (-2.5,-2.5) grid (2.5,2.5);
\node [black] at (-2.9,0.5) {$Q$};
\node [black] at (-2.9,-1.7) {$\widetilde{Q}$};
\node [green!50!black] at (2.4,0.9) {$N_{\varepsilon}(Q)$};
\draw [->,dotted] (1.8,0.85) -- (0.4,0.2);
\draw [<->] (0.7,-0.7) -- (0.7,-2.1);
\node at (0.9,-1.4) {$\varepsilon$};
\draw [<->] (3,1.4) -- (3,-2.1);
\node at (4,-0.35) {$(2\ell+1)\varepsilon_{\lambda}$};
\filldraw [fill=green!20!white, draw=green!50!black, opacity=0.3](1.4,0,0) arc (0:90:1.4) -- (0,0,0) -- cycle;
\filldraw [fill=green!20!white, draw=green!50!black, opacity=0.3](-0.7,1.4,0) arc (90:180:1.4) -- (-0.7,0,0) -- cycle;
\filldraw [fill=green!20!white, draw=green!50!black, opacity=0.3](-2.1,-0.7,0) arc (180:270:1.4) -- (-0.7,-0.7) -- cycle;
\filldraw [fill=green!20!white, draw=green!50!black, opacity=0.3](0,-2.1,0) arc (270:360:1.4) -- (0,-0.7,0) -- cycle;
\filldraw [fill=green!20!white, draw=green!50!black, opacity=0.3](-0.7,-2.1) rectangle (0,1.4) ;
\filldraw [fill=green!20!white, draw=green!50!black, opacity=0.3](-2.1,-0.7) rectangle (1.4,0) ;

\end{tikzpicture}
\caption{Neighbouring cube notation.}
\label{fig:cubegeometry}
\end{figure}
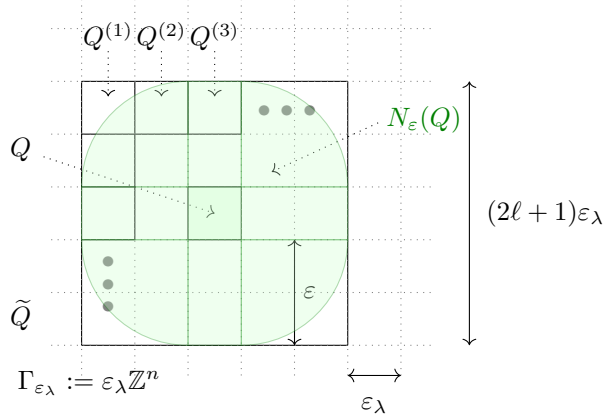

\emph{Step 1. Preliminaries.} Let $\lambda>1$ be given. Let $\Omega$ be as in the proposition and denote, for $t>0$, $N_{t}(\Omega):=\{x\in \R^{n}\colon\;\dista(x,\Omega)<t\}$ the $t$-neighbourhood of $\Omega$. We put $\ell:=\lceil \frac{1}{\lambda-1}\rceil+1\in\mathbb{N}$ so that $\frac{1}{\lambda-1}<\ell$, and define $\varepsilon_{\lambda}:=\frac{\varepsilon}{\ell}$.

We now consider the lattice $\Gamma_{\varepsilon_{\lambda}}:=\varepsilon_{\lambda}\mathbb{Z}^{n}$ and denote $\mathcal{Q}_{\varepsilon_{\lambda}}$ the collection of all open cubes of sidelength $\varepsilon_{\lambda}$ and edge points contained in $\Gamma_{\varepsilon_{\lambda}}$. Given $Q\in\mathcal{Q}_{\varepsilon_{\lambda}}$, we denote $\widetilde{Q}$ the cube which has the same center as $Q$ and sides parallel to those of $Q$ but sidelength $(2\ell+1)\varepsilon_{\lambda}$. Then $\widetilde{Q}$ has all its edge points equally contained in $\Gamma_{\varepsilon_{\lambda}}$,  $N_{\varepsilon}(Q)=Q+\ball(0,\varepsilon)\subset\widetilde{Q}$, and can be written as the union of $\mathscr{N}=\mathscr{N}(\lambda,n)\in\mathbb{N}$ cubes from $\mathcal{Q}_{\varepsilon_{\lambda}}$; for notational convenience, we denote these cubes $Q^{(i)}$, $i=1,...,\mathscr{N}$, and arrange that for all $Q\in\mathcal{Q}_{\varepsilon_{\lambda}}$, the relative positioning of $Q^{(i)}$ to $Q$ is the same -- see Figure~\ref{fig:cubegeometry} for this setup. Moreover, if $Q\in\mathcal{Q}_{\varepsilon_{\lambda}}$ satisfies $Q\cap \Omega\neq\emptyset$, then we have $\widetilde{Q}\subset N_{\lambda\sqrt{n}\varepsilon}(\Omega)$. In fact, in this case there exists $x_{0}\in Q\cap \Omega$ and thus for any $z\in \widetilde{Q}$ we have $\dista(z,\Omega) \leq |x_{0}-z|$. By the geometry of $\widetilde{Q}$ (see Figure~\ref{fig:cubegeometry}), it is clear that $|x_{0}-z|$ does not exceed 
\begin{align*}
\sqrt{n}\varepsilon_{\lambda}+\sqrt{n}\varepsilon = \sqrt{n}\varepsilon\big(\frac{1}{\ell}+1\big)<\lambda \sqrt{n}\varepsilon
\end{align*}
and hence $\dist(z,\Omega)<\lambda\sqrt{n}\varepsilon$ so that $z\in N_{\lambda\sqrt{n}\varepsilon}(\Omega)$. Summarising, for every $Q\in\mathcal{Q}_{\varepsilon_{\lambda}}$ with $Q\cap \Omega\neq\emptyset$, we have $\widetilde{Q}=\bigcup_{i=1}^{\mathscr{N}}Q^{(j)}\subset N_{\lambda\sqrt{n}\varepsilon}(\Omega)$.

\emph{Step 2. A Poincar\'{e}-type inequality for the reference integrand $V$.}
In a second step, we claim that there exists a constant $c=c(n)>0$ such that for every open cube $Q\subset\R^{n}$, every $L>0$ and every $u\in\hold^{\infty}(\R^{n};\R^{n})$ there holds
\begin{align}\label{eq:nonlinearPoincare}
\int_{Q}V(L(u-\widetilde{\Pi}_{Q}u))\dif x \leq C\max\{L\ell(Q),(L\ell(Q))^{2}\}\int_{Q}V(\sg(u))\dif x.
\end{align}
Here, $\widetilde{\Pi}_{Q}u$ denotes the rigid deformation determined by Proposition~\ref{lem:Reshrep}. It is crucial for this inequality to be available in this very form, and so we provide the details. Thus let $u\in \hold^{\infty}(\R^{n};\R^{n})$ and employ the representation from Lemma~\ref{lem:Reshrep}: There exists $\widetilde{\Pi}_{Q}u\in\mathscr{R}(Q)$ such that for all $x\in Q$ there holds 
\begin{align*}
u(x)=\widetilde{\Pi}_{Q}u(x)+T_{Q}[\sg(u)](x)= \widetilde{\Pi}_{Q}u(x) + \int_{Q}R_{Q}(x,y)\sg(u)(y)\dif y, 
\end{align*}
where $|R_{Q}(x,y)|\leq C_{R}|x-y|^{1-n}$ for all $x,y\in Q$, $x\neq y$, with a constant $C_{R}=C_{R}(n)>0$. 
\begin{figure}
\centering
\begin{tikzpicture}
\draw [<->] (3.6, -4.0) -- (4.05,-4.0);
\node at (3.9,-4.3) {$\varepsilon_{\lambda}$};
\draw [<->] (4.85, -2.7) -- (4.85,-3.15);
\node at (5.3,-3.0) {$\varepsilon_{\lambda}$};
\node at (-3.8,-4.0) {$\Gamma_{\varepsilon_{\lambda}}:=\varepsilon_{\lambda}\mathbb{Z}^{n}$};
\node at (0.6,-0.3) {$\Omega$};
\node at (-1.13,-0.68) {$Q$};
\node at (-0.5,-0) {$\widetilde{Q}$};
\draw [-] (-1.359, -0.44) -- (-1.359,-0.9);
\draw [-] (-1.359,-0.9) -- (-0.899,-0.9);
\draw [-] (-0.899,-0.9) -- (-0.899,-0.44);
\draw [-] (-0.899,-0.44) -- (-1.359, -0.44);
\draw [-] (-2.279,-1.82) -- (-2.279, 0.48);
\draw [-] (-2.279, 0.48) -- (0.021, 0.48);
\draw [-] (0.021, 0.48) -- (0.021, -1.82);
\draw [-] (-2.279, -1.82) -- (0.021,-1.82);
\draw [<->] (-1.6, -2.6) -- (-1.91,-3.65);
\node [ColorPink] at (-1,1.1) {\textbullet};
\node [ColorPink] at (-0.5,1.25) {\textbullet};
\node [ColorPink] at (-0.0,1.30) {\textbullet};
\node [ColorPink] at (0.5,1.25) {\textbullet};
\node [ColorPink] at (1,1.1) {\textbullet};
\node [ColorPink] at (1.4,0.9) {\textbullet};
\node [ColorPink] at (-1.2,-2.4) {\textbullet};
\node [ColorPink] at (-0.93,-2.45) {\textbullet};
\node [ColorPink] at (-0.7,-2.4) {\textbullet};
\node [ColorPink] at (1.4,-2.4) {\textbullet};
\node [ColorPink] at (1.7,-2.33) {\textbullet};
\node [ColorPink] at (2.0,-2.26) {\textbullet};
\node [ColorPink] at (2.3,-2.19) {\textbullet};
\node [ColorPink] at (2.6,-2.05) {\textbullet};
\node [ColorPink] at (2.8,-1.92) {\textbullet};
\draw[step=.45cm,gray,dotted] (-5.0,-4.0) grid (5.0,3.0);
\shadedraw[color=blue, opacity=0.3] plot[smooth cycle,thick] coordinates {
    (0:2.0)
    (25:1.9)
    (50:2.0)
    (75:2.1)
    (90:2.0)
    (100:1.9)
    (110:1.7)
    (120:1.6) 
    (130:1.7)
    (140:2.0)  
    (150:2.2)
    (160:2.4)
    (170:2.5)
    (180:3.0)
    (190:3.2)
    (200:3.7)
    (210:3.5)
    (220:3.5)
    (240:3.0)
    (245:2.9)
    (250:2.8)
    (260:2.7) 
    (270:2.5)
    (275:2.4)
    (280:2.0)
    (290:2.4)
    (300:2.5)
    (305:3.1)
    (310:3.2)
    (320:3.0)    
    (330:3.3)
    (340:3.8)
    (350:4.0)    
  };
\node [color=black] at (-2.3,-2.9) {$\lambda\sqrt{n}\varepsilon$};
\node [black] at (2.15,1.2) {$S_{\lambda\sqrt{n}\varepsilon}$};
\draw[color=black, opacity=0.5] plot[smooth cycle,thick] coordinates {
    (0:2.0)
    (25:1.9)
    (50:2.0)
    (75:2.1)
    (90:2.0)
    (100:1.9)
    (110:1.7)
    (120:1.6) 
    (130:1.7)
    (140:2.0)  
    (150:2.2)
    (160:2.4)
    (170:2.5)
    (180:3.0)
    (190:3.2)
    (200:3.7)
    (210:3.5)
    (220:3.5)
    (240:3.0)
    (245:2.9)
    (250:2.8)
    (260:2.7) 
    (270:2.5)
    (275:2.4)
    (280:2.0)
    (290:2.4)
    (300:2.5)
    (305:3.1)
    (310:3.2)
    (320:3.0)    
    (330:3.3)
    (340:3.8)
    (350:4.0)    
  };
\draw[line width=65pt, color=ColorPink, opacity=0.6] plot [smooth, tension=0.2] coordinates{
    (135:1.8)
    (140:2.0)
    (150:2.2)
    (160:2.4)
    (170:2.5)
    (180:3.0)
    (190:3.2)
    (200:3.7)
    (210:3.5)
    (220:3.5)
    (240:3.0)};
\draw[line width=65pt, color=ColorPink, opacity=0.6] plot [smooth, tension=0.2] coordinates{
    (330:3.3)
    (340:3.8)
    (350:4.0)
    (0:2.0)
    (25:1.9)};
\draw[line width=65pt, color=ColorPink, opacity=0.6] plot [smooth, tension=0.2] coordinates{
   (260:2.6)
   (265:2.5)
   (270:2.5)
   (275:2.4)
   (280:2.0)
   (290:2.4)
   (295:2.45)};
\end{tikzpicture}
\caption{Not-to-scale construction in the proof of Proposition \ref{prop:convest}. In step 1, the lattice parameter $\varepsilon_{\lambda}$ must be adjusted in a way such that for any $Q\in\mathcal{Q}_{\varepsilon_{\lambda}}$, $\widetilde{Q}\subset N_{\lambda\sqrt{n}\varepsilon}(\Omega)=\Omega\cup S_{\lambda\sqrt{n}\varepsilon}$. Note that the correcting rigid deformations required for the nonlinear Poincar\'{e} inequality of step 2 are taken over the enlarged cubes $\widetilde{Q}$.}
\end{figure}
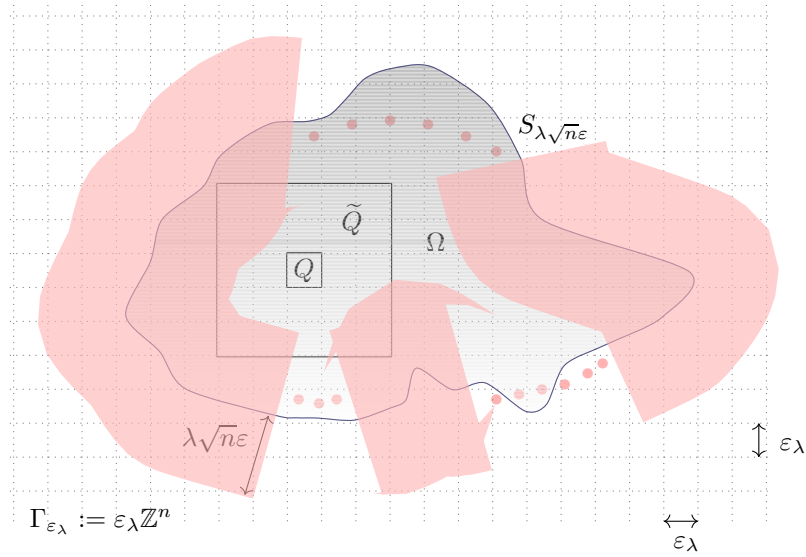
Let $x\in Q$. We define a measure $\mu_{x}\colon \mathscr{B}(Q)\to\R_{\geq 0}$ by putting $\mu_{x}(A):=\int_{A}C_{R}|x-y|^{1-n}\dif y$ for $A\in\mathscr{B}(Q)$. Since $|x-y|<\sqrt{n}\ell(Q)$ for all $x,y\in Q$, 
\begin{align*}
\mu_{x}(Q) = C_{R}\int_{Q}\frac{\dif y}{|x-y|^{n-1}}  \geq \frac{C_{R}}{\sqrt{n}^{n-1}}\ell(Q).
\end{align*}
We also need a remark on the upper bound. Namely, if $x\in Q$, then $Q\subset \ball(x,\sqrt{n}\ell(Q))$ independently of $x$. Thus, with $\omega_{n}=\mathscr{L}^{n}(\ball(0,1))$,
\begin{align*}
\mu_{x}(Q) & \leq C_{R} \int_{\ball(x,\sqrt{n}\ell(Q))}\frac{\dif y}{|x-y|^{n-1}} = C_{R}\int_{\ball(0,\sqrt{n}\ell(Q))}\frac{\dif y}{|y|^{n-1}}\leq C_{R}\omega_{n}n\sqrt{n}\ell(Q). 
\end{align*}
In conclusion, there exists $c=c(n)>0$ such that 
\begin{align}\label{eq:muxbound}
\frac{1}{c}\ell(Q)\leq \mu_{x}(Q) \leq c\ell(Q)
\end{align}
for all cubes $Q$ and $x\in Q$. Now, $\mu_{x}/\mu_{x}(Q)$ is a probability measure on $\mathscr{B}(Q)$ for every $x\in Q$. In consequence, as $|u-\widetilde{\Pi}_{Q}u|\leq |T_{Q}[\sg(u)]|$ pointwisely in $Q$ and $V\colon\R_{\geq 0}\to\R_{\geq 0}$ is monotone, we estimate by Jensen's inequality
\vspace{-0.5cm}
\begin{align*}
\int_{Q}V(L(u-\widetilde{\Pi}_{Q}u))\dif x & \leq \int_{Q}V\left( L\int_{Q}R_{Q}(x,y)\sg(u)(y)\dif y\right)\dif x \\
& \leq \int_{Q}V\left(LC_{R}\mu_{x}(Q)\int_{Q}\frac{|\sg(u)(y)|}{|x-y|^{n-1}}\frac{\dif y}{\mu_{x}(Q)} \right)\dif x \\
& \!\!\!\!\!\!\!\!\!\!\!\!\!\!\!\!\!\!\!\stackrel{\text{Lemma~\ref{lem:eest}\ref{item:eest1},\,\eqref{eq:muxbound}}}{\leq} c\max\{(L\ell(Q)), (L\ell(Q))^{2}\} \times \\ 
& \;\;\;\;\;\;\;\;\;\;\;\;\;\;\;\;\;\;\;\; \times \int_{Q}V\left(\int_{Q}\frac{|\sg(u)(y)|}{|x-y|^{n-1}}\frac{C_{R}\dif y}{\mu_{x}(Q)} \right)\dif x \\
&  \!\!\!\!\stackrel{\text{Jensen}}{\leq} c\max\{(L\ell(Q)), (L\ell(Q))^{2}\}\times \int_{Q}\int_{Q}V(\sg(u)(y))\frac{\dif \mu_{x}(y)}{\mu_{x}(Q)} \dif x \\
& \!\!\!\stackrel{\eqref{eq:muxbound}}{\leq} c\max\{(L\ell(Q)), (L\ell(Q))^{2}\}\frac{1}{\ell(Q)}\int_{Q}\int_{Q}\frac{V(\sg(u)(y))}{|x-y|^{n-1}}\dif y\dif x \\
&\!\leq c\max\{(L\ell(Q)), (L\ell(Q))^{2}\}\frac{1}{\ell(Q)}\int_{Q}V(\sg(u)(y))\mu_{y}(Q)\dif y  \\
&  \!\!\!\stackrel{\eqref{eq:muxbound}}{\leq} c\max\{(L\ell(Q)), (L\ell(Q))^{2}\}\int_{Q}V(\sg(u)(y))\dif y.
\end{align*}
Tracking the dependencies of constants, $c=c(n)>0$, thereby establishing \eqref{eq:nonlinearPoincare}.

\emph{Step 3. Inequality~\eqref{eq:Poincaremain} for $\hold^{\infty}$-maps.} As a main feature of the symmetric gradient operator, let us note that as first order polynomials, all elements $\pi\in \mathscr{R}(\R^{n})$ of its nullspace are \emph{harmonic}. Thus they satisfy the mean value property and, as a consequence, convolution with standard mollifiers locally turns out to be the identity on the rigid deformations, cf. \cite[Chpt.~2.2.3, Thm.~6]{Evans}. For any $Q\in\mathcal{Q}_{\varepsilon_{\lambda}}$, we recall the definition of the cube $\widetilde{Q}$ from step 1. Then \eqref{eq:nonlinearPoincare} holds true with $Q$ and $\widetilde{\Pi}_{Q}u$ being replaced by $\widetilde{Q}$ and $\widetilde{\Pi}_{\widetilde{Q}}u$, respectively. We then obtain, using Lemma~\ref{lem:eest}(i) in the third step
\begin{align*}
\int_{\Omega}V(L(u-\rho_{\varepsilon}*u))\dif x & \leq  \sum_{\substack{Q\in\mathcal{Q}_{\varepsilon_{\lambda}}\\ Q\cap \Omega \neq \emptyset}}\int_{Q}V(L(u-\rho_{\varepsilon}*u))\dif x \\ 
& \leq  \sum_{\substack{Q\in\mathcal{Q}_{\varepsilon_{\lambda}}\\ Q\cap \Omega \neq \emptyset}}\int_{Q}V(L(u-\widetilde{\Pi}_{\widetilde{Q}}u) - \rho_{\varepsilon}*L(u-\widetilde{\Pi}_{\widetilde{Q}}u))\dif x\;\;\\
& \!\!\!\!\!\!\!\!\!\!\!\!\!\!\stackrel{\text{Lemma~\ref{lem:eest}\ref{item:eest2}}}{\leq} 2\sum_{\substack{Q\in\mathcal{Q}_{\varepsilon_{\lambda}}\\ Q\cap \Omega \neq \emptyset}}\int_{Q}V(L(u-\widetilde{\Pi}_{\widetilde{Q}}u)) + V(\rho_{\varepsilon}*L(u-\widetilde{\Pi}_{\widetilde{Q}}u))\dif x. 
\end{align*}
At this stage, we use Jensen's and Young's inequalities to conclude that for any $Q\in\mathcal{Q}_{\varepsilon_{\lambda}}$ there holds
\begin{align}\label{eq:helpineq}
\begin{split}
\int_{Q}V(\rho_{\varepsilon}*L(u-\widetilde{\Pi}_{\widetilde{Q}}u))\dif x & \leq \int_{Q}\rho_{\varepsilon}*V(|u-\widetilde{\Pi}_{\widetilde{Q}}u|)\dif x \\  & \leq \int_{N_{\varepsilon}(Q)}V(|L(u-\widetilde{\Pi}_{\widetilde{Q}}u)|)\dif x \\ & \!\!\!\!\!\!\!\!\!\stackrel{N_{\varepsilon}(Q)\subset\widetilde{Q}}{\leq} \int_{\widetilde{Q}}V(|L(u-\widetilde{\Pi}_{\widetilde{Q}}u)|)\dif x \\ & \!\!\!\! \stackrel{\eqref{eq:nonlinearPoincare}}{\leq} c\max\{L\ell(\widetilde{Q}),(L\ell(\widetilde{Q}))^{2}\}\int_{\widetilde{Q}}V(\sg(u))\dif x \\ & \!\!\!\!\!\!\!\!\!\!\!\!\stackrel{\ell(\widetilde{Q})=\frac{2\ell+1}{\ell}\varepsilon}{\leq} c\max\{L\varepsilon,(L\varepsilon)^{2}\}\sum_{j=1}^{\mathscr{N}}\int_{Q^{(j)}}V(\sg(u))\dif x, 
\end{split}
\end{align}
where $c=c(\lambda,n)$ (note that $\ell$ only depends on $\lambda$). Note that, for any fixed $j\in\{1,...,\mathscr{N}\}$ and all $Q,Q'\in\mathcal{Q}_{\varepsilon_{\lambda}}$ with $Q\neq Q'$, $Q^{(j)}\cap Q'^{(j)}=\emptyset$. On the other hand, by step 1, for any fixed $j\in\{1,...,n\}$, $\bigcup_{Q\in\mathcal{Q}_{\varepsilon_{\lambda}},\;Q\cap\Omega\neq\emptyset}Q^{(j)}\subset \bigcup_{Q\in\mathcal{Q}_{\varepsilon_{\lambda}},\,Q\cap\Omega\neq\emptyset}\widetilde{Q}\subset N_{\lambda\sqrt{n}\varepsilon}(\Omega)$. Therefore, 
\begin{align}\label{eq:helpineq1}
\sum_{\substack{Q\in\mathcal{Q}_{\varepsilon_{\lambda}}\\ Q\cap\Omega\neq\emptyset}}\int_{Q^{(j)}}V(\sg(u))\dif x\leq \int_{N_{\lambda\sqrt{n}\varepsilon}(\Omega)}V(\sg(u))\dif x.
\end{align}
Consequently, we obtain by $Q\subset \widetilde{Q}$ and \eqref{eq:helpineq} in the first step: 
\begin{align*}
\Big(\sum_{\substack{Q\in\mathcal{Q}_{\varepsilon_{\lambda}}\\ Q\cap \Omega \neq \emptyset}}\int_{Q}V(L(u-\widetilde{\Pi}_{\widetilde{Q}}u)))\Big) & + \Big( \sum_{\substack{Q\in\mathcal{Q}_{\varepsilon_{\lambda}}\\ Q\cap \Omega \neq \emptyset}} \int_{Q}V(\rho_{\varepsilon}*L(u-\widetilde{\Pi}_{\widetilde{Q}}))\dif x\Big) \\
& \!\!\!\stackrel{\eqref{eq:helpineq}}{\leq} c\max\{(L\varepsilon),(L\varepsilon)^{2}\}\sum_{\substack{Q\in\mathcal{Q}_{\varepsilon_{\lambda}}\\ Q\cap \Omega \neq \emptyset}}\sum_{j=1}^{\mathscr{N}}\int_{Q^{(j)}}V(\sg(u))\dif x\\
& \!=c\max\{(L\varepsilon),(L\varepsilon)^{2}\}\sum_{j=1}^{\mathscr{N}}\sum_{\substack{Q\in\mathcal{Q}_{\varepsilon_{\lambda}}\\ Q\cap \Omega \neq \emptyset}}\int_{Q^{(j)}}V(\sg(u))\dif x \\
& \!\!\!\stackrel{\eqref{eq:helpineq1}}{\leq} c\max\{(L\varepsilon),(L\varepsilon)^{2}\}\int_{N_{\lambda\sqrt{n}\varepsilon}(\Omega)}V(\sg(u))\dif x. 
\end{align*}
Since $\mathscr{N}=\mathscr{N}(\lambda,n)$, $c=c(\lambda,n)$ in the previous estimation, and \eqref{eq:Poincaremain} follows for $u\in\hold^{\infty}(\R^{n};\R^{n})$. 

\emph{Step 4. Passage to the general case.} Let $u\in\bd_{\locc}(\R^{n})$. By localisation, it is no loss of generality to assume $u\in\bd(\R^{n})$. Let $\eta\in\hold_{c}^{\infty}(\ball(0,1);[0,1])$ be a standard mollifier. We put $u_{k}:=\eta_{1/k}*u$, so that, by passing to a non-relabeled subsequence, $u_{k}\to u$ $\mathscr{L}^{n}$-a.e. in $\R^{n}$.  This yields by Fatou's lemma for all $\varepsilon>0$
\begin{align*}
\int_{\Omega}V(L(u-\rho_{\varepsilon}*u))\dif x & \leq \liminf_{k\to\infty} \int_{\Omega}V(L(u_{k}-\rho_{\varepsilon}*u_{k}))\dif x \\ & \leq c\max\{(L\varepsilon), (L\varepsilon)^{2}\} \liminf_{k\to\infty}\int_{N_{\lambda\sqrt{n}\varepsilon}(\Omega)}V(\sg(u_{k}))\dif y\\
& \leq c\max\{(L\varepsilon), (L\varepsilon)^{2}\} \liminf_{k\to\infty}\int_{N_{\lambda\sqrt{n}\varepsilon+\frac{1}{k}}(\Omega)}V(\E u)\\ 
& \leq c\max\{(L\varepsilon), (L\varepsilon)^{2}\} \int_{\overline{N_{\lambda\sqrt{n}\varepsilon}(\Omega)}}V(\E u), 
\end{align*}
where we used inequality \eqref{eq:Poincaremain} for smooth maps in the second and Jensen's inequality in the third step. This is the inequality claimed in the proposition and the proof of \eqref{eq:Poincaremain} for $u\in\bd_{\locc}(\R^{n})$ is hereby complete. 
\end{proof}
For consistency, let us note that if the right hand side of \eqref{eq:Poincaremain} is zero, then $u$ must coincide with a rigid deformation on each of the connected components of $U+\ball(0,\lambda\sqrt{n}\varepsilon)$ and so on those of $U$; in consequence, it must coincide with its mollification on each of these connected components and hence the left hand side is zero indeed. 

\begin{corollary}\label{cor:convolutioncontrol}
Let $\lambda>1$ and let $V(z):=\sqrt{1+|z|^{2}}-1$ be the auxiliary reference integrand as usual. Then there exists a constant $c=c(n,\lambda)>0$ such that the following holds: For every open and bounded Lipschitz domain $\Omega\subset\R^{n}$, $u\in\bd_{\locc}(\R^{n})$ and numbers $\varepsilon,L>0$ there holds 
\begin{align}\label{eq:Poincaremain1}
\int_{\Omega}V(L(u-\eta_{\varepsilon}*(\rho_{\varepsilon}*u)))\dif x \leq c\max\{(L\varepsilon),(L\varepsilon)^{2}\} \int_{\overline{\Omega+\ball(0,2\lambda\sqrt{n}\varepsilon)}}V(\E u), 
\end{align}
where $\rho,\eta\colon \R^{n}\to \R_{\geq 0}$ are arbitrary standard mollifiers in the sense of Section~\ref{sec:funofmeas}.
\end{corollary}
\begin{proof}
Denote the left-hand side of \eqref{eq:Poincaremain1} by $\mathbf{I}$. We use Lemma~\ref{lem:eest}\ref{item:eest2} and Jensen's inequality to obtain
\begin{align*}
\mathbf{I} & \leq 2\int_{\Omega}V(L(u-(\rho_{\varepsilon}*u)))\dif x + 2\int_{\Omega}V(L(\rho_{\varepsilon}*u-\eta_{\varepsilon}*(\rho_{\varepsilon}*u)))\dif x \\ 
& \!\!\!\!\!\!\!\stackrel{\text{Prop.}~\ref{prop:convest}}{\leq} c\max\{(L\varepsilon),(L\varepsilon)^{2}\}\Big(\int_{\overline{\Omega+\ball(0,\lambda\sqrt{n}\varepsilon)}}V(\E u) + \int_{\overline{\Omega+\ball(0,\lambda\sqrt{n}\varepsilon)}}V(\E\, (\rho_{\varepsilon}*u))\Big)\\
& \!\!\!\!\!\stackrel{\lambda\sqrt{n}>1}{\leq} c\max\{(L\varepsilon),(L\varepsilon)^{2}\}\int_{\overline{\Omega+\ball(0,2\lambda\sqrt{n}\varepsilon)}}V(\E u), 
\end{align*}
where again $c=c(\lambda,n)$. The proof is complete. 
\end{proof}
We conclude this section by discussing a particular borderline case in the spirit of \eqref{eq:Poincaremain}, for simplicity stated on the entire $\R^{n}$:
\begin{corollary}[Sobolev-Poincar\'{e} inequality in convolution form]
For any $1\leq p \leq \frac{n}{n-1}$ there exists a constant $c=c(n,p)>0$ with the following property: For every $u\in\bd(\R^{n})$ and $\varepsilon>0$ there holds 
\begin{align}\label{eq:Poincaremain0000}
\Big(\int_{\R^{n}}|u-\rho_{\varepsilon}*u|^{p}\dif x\Big)^{\frac{1}{p}} \leq c\,\varepsilon^{1-n+\frac{n}{p}} \int_{\R^{n}}|\!\E u|, 
\end{align}
where $\rho\colon \R^{n}\to \R_{\geq 0}$ is an arbitrary standard mollifier in the sense of Section~\ref{sec:funofmeas}.
\end{corollary}
\begin{proof}
Let $u\in\ld(\R^{n})$. By the \textsc{Strauss} inequality \cite{Strauss} and Poincar\'{e}'s Inequality, Lemma~\ref{rem:stability}, there exists a constant $c=c(n,p)>0$ such that $\|u-\widetilde{\Pi}_{Q}u\|_{\lebe^{p}(Q;\R^{n})}\leq c\ell(Q)^{\frac{p-pn+n}{p}}\|\sg(u)\|_{\lebe^{1}(Q;\rsym)}$ for all $u\in\ld(\R^{n})$ and cubes $Q\subset\R^{n}$. We argue  as in the proof of Proposition~\ref{prop:convest}, but now work with the lattice $\Gamma_{\varepsilon}=\varepsilon\mathbb{Z}^{n}$ and, for $Q\in\mathcal{Q}_{\varepsilon}$, define $\widetilde{Q}$ to be the cube with the same center as $Q$ but $(2n+1)$-times its sidelength. This yields
\begin{align*}
\int_{\R^{n}}|u-\rho_{\varepsilon}*u|^{p}\dif x & \leq c\sum_{j=1}^{\mathscr{N}(n)}\sum_{Q\in\mathcal{Q}_{\varepsilon}}\ell(\widetilde{Q})^{p-pn+n}\Big(\int_{Q^{(j)}}|\sg(u)|\dif x\Big)^{p}\\ 
& \!\!\!\!\!\!\!\!\!\!\!\!\!\stackrel{\ell^{1}(\mathbb{Z}^{n})\hookrightarrow \ell^{p}(\mathbb{Z}^{n})}{\leq} c\varepsilon^{p-pn+n}\sum_{j=1}^{\mathscr{N}(n)}\Big(\sum_{Q\in\mathcal{Q}_{\varepsilon}}\int_{Q^{(j)}}|\sg(u)|\dif x \Big)^{p} \\ 
& \,\leq c \varepsilon^{p-pn+n}\Big(\int_{\R^{n}}|\sg(u)|\dif x\Big)^{p}, 
\end{align*}
and from here the conclusion follows by smooth approximation as above.  
\end{proof}
Following the scheme of proof, other inequalities can equally be obtained, so, e.g., by replacing the $\lebe^{p}$-norm on the left-hand side of \eqref{eq:Poincaremain0000} by Sobolev-Slobodeckji\u{\i} (use $\bd(\R^{n})\hookrightarrow\sobo^{s,\frac{n}{n-1+s}}(\R^{n};\R^{n})$, $0<s<1$, cf.~\cite{GK1}) or Triebel-Lizorkin seminorms. 
\section{Partial $\hold^{1,\alpha}$-regularity and the proof of Theorem~\ref{thm:PR}}\label{sec:PR}
In this section we provide the proof of the second main result of this paper, Theorem~\ref{thm:PR}, allowing for possibly very degenerate ellipticities.
\subsection{Outline of the proof and setup}\label{sec:PRintro}
In order to reach the full degenerate elliptic regime which Theorem~\ref{thm:PR} applies to, we employ a direct comparison strategy that uses mollifications of generalised minima as comparison maps. A direct strategy here is suggested by both the very weak compactness properties of $\bd$ and the general lack of higher integrability of generalised minima in the very degenerate ellipticity regime (e.g., if $1+\frac{2}{n}\leq a <\infty$). Comparison methods of this type, originally employed in \cite{AG} for the full gradient case, consequently require to control $V$-function-type distances of generalised minima to their mollifications. This is where the convolution-type Poincar\'{e} inequalities of the previous section enter crucially. More precisely, we proceed as follows: 
\begin{itemize}
\item[(i)]\label{item:step1} Section~\ref{sec:estcomparison}: \emph{Estimates for comparison maps.} By linearisation, Proposition~\ref{prop:1} establishes that if a $\hold^{1,\alpha}$-H\"{o}lder continuous function satisfies a certain smallness condition and has symmetric gradient close to some carefully chosen reference point, then it almost enjoys the typical decay for linear systems. For the linearised integrands, full gradient estimates are available by \textsc{Korn}'s inequality in $\lebe^{2}$. 
\item[(ii)]\label{item:step2} Section~\ref{sec:radiiselect}: \emph{Smoothing and selection of good radii.} To construct the requisite $\hold^{1,\alpha}$-comparison maps for step (i), we carefully mollify the given generalised minimiser and demonstrate that, under suitable smallness assumptions, the mollification parameters can be chosen such that the comparison estimates from (i) become available, cf. Lemma~\ref{lem:adjust} and Corollary~\ref{cor:adjust}. 
\item[(iii)]\label{item:step3} Section~\ref{sec:comparison}: \emph{Comparison estimates and decay.} Here we give the aforementioned comparison argument and employ minimality to deduce a preliminary decay estimate for generalised minima, cf.~Proposition~\ref{prop:main1}. To control the emerging terms, the comparison will be essentially reduced to \emph{good} annuli where the relevant differences can be dealt with conveniently. The construction of such annuli hinges on Lemma~\ref{lem:convex}, giving control over the symmetric gradients, whereas Corollary~\ref{cor:convolutioncontrol} allows to suitably bound lower order terms.
\end{itemize}
These steps lead to an $\varepsilon$-regularity result, Corollary~\ref{cor:epsreg}, finally implying Theorem~\ref{thm:PR}; cf.~Section~\ref{sec:proofPR}. 
We now introduce the requisite terminology for the proof below: Given $x_{0}\in\Omega$ and $R>0$ such that $\ball(x_{0},R)\Subset\Omega$, we define for $u\in\bd_{\locc}(\Omega)$ two excess quantities by 
\begin{align}\label{eq:definitionexcessquantities}
\excenew(u;x_{0},R):=\int_{\ball(x_{0},R)}V(\E u-(\E u)_{x_{0},R})\;\;\;\text{and}\;\;\;\excesso(u;x_{0},R):=\frac{\excenew(u;x_{0},R)}{\mathscr{L}^{n}(\ball(x_{0},R))},
\end{align}
where the mean values in the definition of $\excenew,\widetilde{\Phi}$ are taken with respect to $\mathscr{L}^{n}$, cf.~\eqref{eq:meanvaluemeasures}. 
\subsection{Preliminary decay estimates}\label{sec:prelimdecay}
After the preparations of the previous section, we now carry out the steps (i), (ii) and (iii) as outlined in Section~\ref{sec:PRintro} above.
\subsubsection{Estimates for comparison maps}\label{sec:estcomparison}
Let $f\in\hold^{2}(\rsym)$ satisfy \eqref{eq:lingrowth1} and let $0<\alpha<1$. Throughout this paragraph, we fix $\xi_{0}\in\R_{\sym}^{n\times n}$, a radius $0<\varrho_{\xi_{0}}<1$ and assume that $f\in \hold^{2}(\R_{\sym}^{n\times n})$ satisfies 
\begin{align}\label{eq:ellipticcomparison}
\lambda|\xi|^{2}\leq \langle f''(\xi_{0})\xi,\xi\rangle \leq\Lambda|\xi|^{2}\qquad\text{for all}\;\xi\in\rsym
\end{align}
for some $0<\lambda\leq\Lambda<\infty$. Moreover, we suppose that there exists a bounded and non-decreasing function $\omega_{\xi_{0},\varrho_{\xi_{0}}}\colon \R_{\geq 0}\to\R_{\geq 0}$ with $\lim_{t\searrow 0}\omega_{\xi_{0},\varrho_{\xi_{0}}}(t)=0$ such that  
\begin{align}\label{eq:modcon}
|f''(\xi)-f''(\xi_{0})|\leq\omega_{\xi_{0},\varrho_{\xi_{0}}}(|\xi-\xi_{0}|)\qquad\text{for all $\xi\in\mathbb{B}(\xi_{0},\varrho_{\xi_{0}})$}. 
\end{align}
Finally, for $0<r<R$ and $x_{0}\in\Omega$ with $\ball(x_{0},R)\Subset\Omega$ and $v\in \hold^{1,\alpha}(\overline{\ball(x_{0},r)};\R^{n})$ we put 
\begin{align*}
&\devi(v;x_{0},r):=\int_{\ball(x_{0},r)}f(\sg(v))\dif x-\inf\left\{\int_{\ball(x_{0},r)}f(\sg(w))\dif x\colon \!\!\!\begin{array}{c}
w\in \hold^{1,\alpha}(\overline{\ball(x_{0},r)};\R^{n}) \\ w = v\;\text{on}\;\partial\!\ball(x_{0},r)
\end{array}\!\!\!\right\},\\
&\mathbf{t}_{\alpha,\xi_{0}}(v;x_{0},r):=\sup_{\ball(x_{0},r)}|\sg(v)-\xi_{0}|+2^{\alpha}r^{\alpha}[\sg(v)]_{\hold^{0,\alpha}(\overline{\ball(x_{0},r)};\rsym)}
\end{align*}
The deviation $\devi$ captures how far $v$ is away from minimising $F$ on $\hold^{1,\alpha}(\overline{\ball(x_{0},r)};\R^{n})$ for its own boundary values. Conversely, $\mathbf{t}_{\alpha,\xi_{0}}$ will prove instrumental to find the mentioned smallness condition which is necessary to infer the decay estimate of the H\"{o}lder continuous comparison maps. We have
\begin{proposition}\label{prop:1}
Let $f,\alpha,\xi_{0},\varrho_{\xi_{0}},\lambda,\Lambda$ and $\omega_{\xi_{0},\varrho_{\xi_{0}}}$ be as above. Then there exists $1\leq c_{\comp}=c_{\comp}(n,\lambda,\Lambda)<\infty$ such that the following holds: If $v\in \hold^{1,\alpha}(\overline{\ball(x_{0},R/2)};\R^{n})$ satisfies $\mathbf{t}_{\alpha,\xi_{0}}(v;x_{0},R/2)<\varrho_{\xi_{0}}/c_{\comp}$, then there exists a bounded, non-decreasing function $\vartheta\colon\R_{\geq 0}\to\R_{\geq 0}$ with $\lim_{t\searrow 0}\vartheta(t)=0$, only depending on $n,\lambda,\Lambda$ and $\omega_{\xi_{0},\varrho_{\xi_{0}}}$, and a constant $c=c(n,\lambda,\Lambda)>0$ such that for all $0<r<R/2$ we have 
\begin{align}\label{eq:smoothes}
\begin{split}
\int_{\ball(x_{0},r)}|\sg(v)-(\sg(v))_{x_{0},r}|^{2}  \dif x & \leq c \Big(\left(\frac{r}{R}\right)^{n+2}\int_{\ball(x_{0},R/2)}|\sg(v)-(\sg(v))_{x_{0},R/2}|^{2}\dif x\Big. \\
& \Big.+\vartheta(\mathbf{t}_{\alpha,\xi_{0}}(v;x_{0},R/2))\int_{\ball(x_{0},R/2)}|\sg(v)-\xi_{0}|^{2}\dif x\Big.  \\ & \Big.+ \devi(v;x_{0},R/2)\Big).
\end{split}
\end{align}
\end{proposition}
The preceding proposition essentially follows by reduction to the full gradient case as a consequence of \textsc{Korn}'s inequality. For the reader's convenience, it is established in the Appendix, Section~\ref{sec:LCE}, together with the requisite estimates for linear systems.
\subsubsection{Smoothing and selection of good radii}\label{sec:radiiselect}
In this section we concentrate on step (ii) and establish the required adjusting of the smoothing parameters. The following lemma and its corollary closely follow \cite[Lem.~4.2]{AG} but with a slight change in the relevant constants. Here and in all of what follows, we choose and fix a constant $\lambda_{\con}>1$ for latter application of the convolution inequality from Proposition~\ref{prop:convest}; for instance, $\lambda_{\con}:=1+\frac{1}{1000}$ will do. 
\begin{lemma}\label{lem:adjust}
Let $u\in \bd_{\locc}(\R^{n})$, $x_{0}\in\R^{n}$, $r>0$ and put $\xi_{0}:=(\E u)_{x_{0},r}$. Moreover, 
suppose that $\excesso(u;x_{0},r)<1$, where $\excesso$ is defined by \eqref{eq:definitionexcessquantities}. Then for each $0<\alpha<1$ there exists $c=c(n,\alpha)>0$ such that if 
\begin{align}\label{eq:deltaeps}
\varepsilon = \frac{1}{48\sqrt{n}\lambda_{\con}}r\excesso(u;x_{0},r)^{\frac{1}{n+4\alpha}}, 
\end{align}
then the mollification $u_{\varepsilon,\varepsilon}$ of $u$ (cf.~\eqref{eq:udeltaepsdef}) satisfies
\begin{align}\label{eq:adjustclaimequation}
\mathbf{t}_{\alpha,\xi_{0}}(u_{\varepsilon,\varepsilon};x_{0},\tfrac{r}{2})\leq c(n,\alpha) \excesso(u;x_{0},r)^{\frac{\alpha}{n+4\alpha}}.
\end{align}
\end{lemma}
\begin{proof}
First observe that, as a consequence of elementary estimates for convolutions, we obtain with a constant $c=c(n)>0$
\begin{align}\label{eq:step}
\mathbf{t}_{\alpha,\xi_{0}}\Big(u_{\varepsilon,\varepsilon};x_{0},\frac{r}{2}\Big) & \leq c\,\left(1+\left(\frac{r}{\varepsilon}\right)^{\alpha}\right)\sup_{\ball(x_{0},\frac{r}{2}+\varepsilon)}|\sg(u_{\varepsilon})-\xi_{0}|.
\end{align}
In fact, for $x\in\ball(x_{0},\frac{r}{2})$ we have $|\sg(u_{\varepsilon,\varepsilon})(x)-\xi_{0}| = |\rho_{\varepsilon}^{(2)}*(\sg(u_{\varepsilon})-\xi_{0})(x)|$ and thus 
\begin{align}\label{eq:step1}
\sup_{x\in\ball(x_{0},r/2)}|\sg(u_{\varepsilon,\varepsilon})(x)-\xi_{0}|\leq \sup_{x\in\ball(x_{0},r/2+\varepsilon)}|\sg(u_{\varepsilon})(x)-\xi_{0}|.
\end{align}
On the other hand, for any radially symmetric standard mollifier $\eta\colon\ball(0,1)\to[0,1]$ there exists a constant $c_{\eta}>0$ such that for all $g\in\lebe^{1}(\R^{n};\R_{\sym}^{n\times n})$ and $\delta>0$ there holds 
\begin{align}\label{eq:step3}
[\eta_{\delta}*g]_{\hold^{0,\alpha}(\overline{\ball(x_{0},r/2)};\rsym)}\leq \frac{c_{\eta}}{\delta^{\alpha}}\sup_{\ball(x_{0},r/2+\delta)}|g-\xi|\qquad\text{for all}\;\xi\in\R_{\sym}^{n\times n},
\end{align}
which can be established by straightforward computation. Therefore, with $c=c(n)>0$,
\begin{align}\label{eq:step2}
r^{\alpha}[\sg(u_{\varepsilon,\varepsilon})]_{\hold^{0,\alpha}(\ball(x_{0},r/2);\rsym)}\leq c \left(\frac{r}{\varepsilon}\right)^{\alpha}\sup_{\ball(x_{0},r/2+\varepsilon)}|\sg(u_{\varepsilon})-\xi_{0}|. 
\end{align}
In consequence, adding \eqref{eq:step1} and \eqref{eq:step2} yields \eqref{eq:step}, and in order to arrive at the claimed estimate, we must give an estimate for $\sup_{\ball(x_{0},r/2+\varepsilon)}|\sg(u_{\varepsilon})-\xi_{0}|$. As $\varepsilon$ is adjusted by \eqref{eq:deltaeps} and thus $\ball(x,\varepsilon)\subset\ball(x_{0},r)$ for all $x\in\ball(x_{0},\frac{r}{2}+\varepsilon)$, we obtain by Jensen's inequality
\begin{align}\label{eq:Kepschoose}
\begin{split}
V(\sg(u_{\varepsilon})(x)-\xi_{0}) & \leq \dashint_{\ball(x,\varepsilon)}V(\E u-\xi_{0}) \\ & \leq \left(\frac{r}{\varepsilon}\right)^{n}\excesso(u;x_{0},r)\leq \widetilde{\ell} \excesso(u;x_{0},r)^{\frac{4\alpha}{n+4\alpha}} \leq \widetilde{\ell}
\end{split}
\end{align}
for all such $x$, where $\widetilde{\ell}=(48\sqrt{n}\lambda_{\con})^{n}$. Here, the ultimate estimate is due to our assumption $\excesso(u;x_{0},r)<1$. By Lemma~\ref{lem:eest}\ref{item:eest4} with $\ell=\sqrt{\widetilde{\ell}^{2}+2\widetilde{\ell}}$ and using \eqref{eq:Kepschoose}, we obtain for all $x\in\ball(x_{0},\frac{r}{2}+\varepsilon)$ with a constant $c(n)>0$ (as our choice of $\ell$ only depends on $n$)
\begin{align}\label{eq:intermediateadjust}
|\sg(u_{\varepsilon})(x)-\xi_{0}|^{2} \leq c(n)V(\sg(u_{\varepsilon})(x)-\xi_{0})\leq c(n) \excesso(u;x_{0},r)^{\frac{4\alpha}{n+4\alpha}}. 
\end{align}
Now, by \eqref{eq:step}, the specific choice of $\varepsilon$ by \eqref{eq:deltaeps}, \eqref{eq:intermediateadjust} and since $\excesso(u;x_{0},r)<1$, 
\begin{align*}
\mathbf{t}_{\alpha,\xi_{0}}(u_{\varepsilon,\varepsilon};x_{0},r/2) & \leq c(n,\alpha)\left(1+\left(\excesso(u;x_{0},r)\right)^{-\frac{\alpha}{n+4\alpha}}\right)\excesso(u;x_{0},r)^{\frac{2\alpha}{n+4\alpha}} \\ 
& \leq c(n,\alpha) \excesso(u;x_{0},r)^{\frac{\alpha}{n+4\alpha}}.
\end{align*}
This is \eqref{eq:adjustclaimequation}, and the proof is complete. 
\end{proof}
Working from \eqref{eq:adjustclaimequation}, Jensen's inequality in conjunction with Lemma~\ref{lem:eest}\ref{item:eest4} then yields
\begin{corollary}\label{cor:adjust}
In the situation and adopting the terminology of Lemma \ref{lem:adjust}, we have 
\begin{align}\label{eq:adjustcoreq}
\begin{split}
&\int_{\ball(x_{0},r/2)}|\sg(u_{\varepsilon,\varepsilon})-\xi_{0}|^{2}\dif x \leq c(n,\alpha) \excenew(u;x_{0},r),\;\;\;\text{and}\\&\int_{\ball(x_{0},r/2)}|\sg(u_{\varepsilon,\varepsilon})-(\sg(u_{\varepsilon,\varepsilon}))_{x_{0},r/2}|^{2}\dif x \leq c(n,\alpha) \excenew(u;x_{0},r).
\end{split}
\end{align}
\end{corollary}

\subsubsection{Comparison estimates and decay}\label{sec:comparison}
In this section, we let $u\in\gm(F;u_{0})$ be a generalised minimiser, where $f$ satisfies the requirements of Theorem~\ref{thm:PR}. Throughout, let $x_{0}\in\Omega$ and $R>0$ with $\ball(x_{0},R)\Subset\Omega$ be given. We put $\xi_{0}:=(\E u)_{x_{0},R}$ and let $\varrho_{\xi_{0}}>0$. For $a\in\mathbb{B}(\xi_{0},\varrho_{\xi_{0}})$ we recall from \eqref{eq:tildef} the shifted integrand $f_{a}\colon\rsym\to\R$ defined by 
\begin{align*}
f_{a}(\xi):=f(a+\xi)-f(a)-\langle f'(a),\xi\rangle,\;\;\;\xi\in \R_{\sym}^{n\times n}.
\end{align*}
Given a map $w\colon \ball(x_{0},R)\to \R^{n}$, we then define $\widetilde{w}_{a}\colon\ball(x_{0},R)\to \R^{n}$ by 
\begin{align}\label{eq:tildew}
\widetilde{w}_{a}(x):=w(x)-A_{x_{0}}(x):= w(x)-a(x-x_{0}).
\end{align} 
\begin{proposition}[Preliminary decay estimate]\label{prop:main1}
Let $f\in\hold^{2}(\rsym)$ be a convex function with \eqref{eq:lingrowth1}. Also, suppose that $\xi_{0}\in\rsym$, $0<\varrho_{\xi_{0}}<1$ are such that the following hold:
\begin{enumerate}
\item There exists a bounded and non-decreasing function $\omega_{\xi_{0},\varrho_{\xi_{0}}}\colon\R_{\geq 0}\to\R_{\geq 0}$ with
\begin{align}\label{eq:mainlemmacontinuity}
\begin{split}
& \lim_{t\searrow 0}\omega_{\xi_{0},\varrho_{\xi_{0}}}(t)=0,\\
& |f''(\xi)-f''(\xi_{0})|\leq\omega_{\xi_{0},\varrho_{\xi_{0}}}(|\xi-\xi_{0}|)\;\;\;\text{for all}\;\xi\in\mathbb{B}(\xi_{0},\varrho_{\xi_{0}}).
\end{split}
\end{align}
\item $m_{\xi_{0},\varrho_{\xi_{0}}}:=\min\{\lambda(z)\;\text{smallest eigenvalue of}\;f''(z)\colon\;z\in\overline{\mathbb{B}(\xi_{0},\varrho_{\xi_{0}})}\}>0.$
\end{enumerate}
Then there exist constants $\Theta=\Theta(\varrho_{\xi_{0}},n,\alpha)\in (0,1)$ and 
\begin{align*}
c=c(n,\lambda_{\con},\varrho_{\xi_{0}},m_{\xi_{0},\varrho_{\xi_{0}}},\Lip(f),\sup_{\mathbb{B}(\xi_{0},\varrho_{\xi_{0}})}|f''|)>0
\end{align*}
such that 
\begin{align}\label{eq:preliminarysmallnessTheta}
(\E u)_{x_{0},R}=\xi_{0}\;\;\;\text{and}\;\;\;\excesso(u;x_{0},R)<\Theta
\end{align}
imply that
\begin{align}\label{eq:partial}
\excenew(u;x_{0},r) \leq c\Big(\excenew(v;x_{0},2r)+ \Big(1+\left(\frac{R}{r}\right)^{n+1}\Big)(\excesso(u;x_{0},R))^{\frac{1}{2n+8\alpha}}\excenew(u;x_{0},R)\Big)
\end{align}
holds for all $0<r<R/4$. Here we have set $v:=u_{\varepsilon,\varepsilon}$ (cf.~\eqref{eq:udeltaepsdef}) where 
\begin{align}\label{eq:epsilonchoosemain}
\varepsilon:=\frac{1}{48\sqrt{n}\lambda_{\con}}R\excesso(u;x_{0},R)^{\frac{1}{n+4\alpha}}.
\end{align}
\end{proposition}
\begin{proof}
The comparison argument underlying the proof consists of three ingredients: Lemma~\ref{lem:convex} and Proposition~\ref{prop:convest}, both expressing properties of (the symmetric gradients of) generic $\bd_{\locc}$-maps, and generalised local minimality of $u$. 

\emph{Step 1. Preliminaries.} Let $0<r<R/4$ and put $a:=(\sg(v))_{x_{0},r}$. By \eqref{eq:epsilonchoosemain}   and since $t\mapsto \mathbf{t}_{\alpha,\xi_{0}}(v;x_{0},t)$ is non-decreasing, we have by Lemma~\ref{lem:adjust} and $\xi_{0}=(\E u)_{x_{0},R}$
\begin{align*}
\mathbf{t}_{\alpha,\xi_{0}}(v;x_{0},r) & \leq \mathbf{t}_{\alpha,\xi_{0}}(v;x_{0},\tfrac{R}{2})\leq c(n,\alpha)\excesso(u;x_{0},R)^{\frac{n}{n+4\alpha}},
\end{align*}
where we can assume without loss of generality that $c(n,\alpha)>1$. From here we deduce
\begin{align}\label{eq:distaxi0}
|\xi_{0}-a| & \leq \dashint_{\ball(x_{0},r)}|\sg(v)-\xi_{0}|\dif x \leq \mathbf{t}_{\alpha,\xi_{0}}(v;x_{0},r) \leq c(n,\alpha)\excesso(u;x_{0},R)^{\frac{n}{n+4\alpha}}, 
\end{align}
and put, with $c(n,\alpha)>0$ as in the preceding inequality,
\begin{align}\label{eq:achoose}
\Theta:=\left(\frac{\varrho_{\xi_{0}}}{4c(n,\alpha)}\right)^{1+\frac{4\alpha}{n}}\left(\frac{1}{10}\right)^{2(n+4\alpha)}.
\end{align}
With this choice of $\Theta$, $\excesso(u;x_{0},R)<\Theta$ implies $|\xi_{0}-a|<\varrho_{\xi_{0}}/2$ by virtue of \eqref{eq:distaxi0}. Hence $\mathbb{B}(a,\frac{\varrho_{\xi_{0}}}{2})\subset\mathbb{B}(\xi_{0},\varrho_{\xi_{0}})$, and so   $\eqref{eq:mainlemmacontinuity}_{2}$ continues to hold in $\mathbb{B}(a,\frac{\varrho_{\xi_{0}}}{2})$. Lemma~\ref{lem:shifted}\ref{item:shifted2} moreover implies that there exists $c_{0}=c_{0}(\xi_{0},\varrho_{\xi_{0}},m_{\xi_{0},\varrho_{\xi_{0}}},\Lip(f),\sup_{\mathbb{B}(\xi_{0},\varrho_{\xi_{0}})}|f''|))>1$ such that 
\begin{align}\label{eq:loweruppermain}
\tfrac{1}{c_{0}}V(\xi) \leq f_{a}(\xi)\leq c_{0}V(\xi)\qquad\text{for all}\;\xi\in\rsym.
\end{align}
\emph{Step 2. Selection of good radii.} For this proof, we put for $w\in\bd_{\locc}(\Omega)$ with slight abuse of notation
\begin{align*}
\overline{F}_{a}[w;\omega]:=\int_{\omega}f_{a}(\E w)
\end{align*}
whenever $\omega\Subset\Omega$ has Lipschitz boundary $\partial\omega$. By Lemma~\ref{lem:shifted}\ref{item:shifted1}, $f_{a}\geq 0$ and so $\overline{F}_{a}\geq 0$. To employ the comparison argument in step 3 from below, we require a suitable bound on the difference $\overline{F}_{a}[\widetilde{v}_{a};\mathcal{A}]-\overline{F}_{a}[\widetilde{u}_{a};\mathcal{A}]$ in terms of the excess $\excesso(u;x_{0},R)$, $\mathcal{A}\subset\ball(x_{0},R)$ denoting an annulus. This task can, in general, only be achieved on certain annuli $\mathcal{A}$, and we proceed by constructing the latter. We define an exit index 
\begin{align}\label{eq:exit}
N:=\left\lfloor \frac{125}{8(\excesso(u;x_{0},R))^{\frac{1}{2n+8\alpha}}}\right\rfloor. 
\end{align}
Then, by \eqref{eq:preliminarysmallnessTheta} and \eqref{eq:achoose}, $N\geq 15/(\excesso(u;x_{0},R))^{\frac{1}{2n+8\alpha}}$. We then put, for $k\in\{1,...,8N\}$,
\begin{align}\label{eq:akdef}
\kappa_{k}:=\frac{5}{8}R+k\frac{R}{500}(\excesso(u;x_{0},R))^{\frac{1}{2n+8\alpha}}
\end{align}
so that $\kappa_{k}\in [\frac{5}{8}R,\frac{7}{8}R]$. By our choice \eqref{eq:epsilonchoosemain} of $\varepsilon$, we have $R-2\varepsilon>\tfrac{7}{8}R$. Also, by Lemma~\ref{lem:shifted}\ref{item:shifted2}, $f_{a}\in\hold^{2}(\rsym)$ is of linear growth. Since moreover $r<\tfrac{R}{4}$, Lemma \ref{lem:convex}\ref{item:AG2} is applicable and yields that for any $k\in\{1,...,N\}$ there exist $t_{k}\in(\kappa_{8k-1},\kappa_{8k})$ and $r_{k}\in(r,2r)$ with 
\begin{align}\label{eq:rtestimate}
\begin{split}
\int_{\mathcal{A}(x_{0};r_{k},t_{k})}f_{a}(\E\widetilde{v}_{a})-f_{a}(\E\widetilde{u}_{a}) & \stackrel{\text{Lemma~\ref{lem:convex}}}{\leq}  4\varepsilon\Big(\frac{1}{\kappa_{8k}-\kappa_{8k-1}}+\frac{1}{r} \Big)\int_{\ball(x_{0},R)}f_{a}(\E\widetilde{u}_{a}) \\ 
& \stackrel{\eqref{eq:akdef},\,\eqref{eq:epsilonchoosemain}}{\leq} 
\frac{50}{\lambda_{\con}\sqrt{n}}\Big((\excesso(u;x_{0},R))^{\frac{1}{2n+8\alpha}}\Big. \\ & \Big. \;\;\;\;\;\;\;\;\;\;\;\;\;\;+(\excesso(u;x_{0},R))^{\frac{1}{n+4\alpha}}\Big(\frac{R}{r}\Big) \Big)\int_{\ball(x_{0},R)}f_{a}(\E\widetilde{u}_{a}). 
\end{split}
\end{align}
Now, recalling the choice \eqref{eq:epsilonchoosemain} of $\varepsilon$,  for $k=1,...,N$ the annuli 
\begin{align}\label{eq:annulidefine}
\begin{split}
\mathcal{A}_{k} & :=\mathcal{A}(x_{0};t_{k}-2\lambda_{\con}\sqrt{n}\varepsilon,s_{k}+2\lambda_{\con}\sqrt{n}\varepsilon),\;s_{k}:=t_{k}+\tfrac{R}{500}(\excesso(u;x_{0},R))^{\frac{1}{2n+8\alpha}}
\end{split}
\end{align}
are pairwise disjoint and contained in $\ball(x_{0},R)$. Let us address this point in detail: By our choice of $\varepsilon$, disjointness of $\mathcal{A}_{k}$ and $\mathcal{A}_{k+1}$ is equivalent to 
\begin{align}\label{eq:radiuscompute}
\begin{split}
\mathcal{A}_{k+1}\cap\mathcal{A}_{k}=\emptyset & \stackrel{\eqref{eq:annulidefine}}{\Leftrightarrow} t_{k+1}-2\lambda_{\con}\sqrt{n}\varepsilon > s_{k}+2\lambda_{\con}\sqrt{n}\varepsilon \\ 
& \!\!\!\!\!\!\!\stackrel{\eqref{eq:annulidefine},\eqref{eq:epsilonchoosemain}}{\Leftrightarrow} t_{k+1}-t_{k} > \frac{R}{500}(\excesso(u;x_{0},R))^{\frac{1}{2n+8\alpha}}+ \frac{R}{12}(\excesso(u;x_{0},R))^{\frac{1}{n+4\alpha}}. 
\end{split}
\end{align}
Now note that by construction, $t_{k+1}-t_{k}>\kappa_{8k+7}-\kappa_{8k}=\frac{7}{500}R(\excesso(u;x_{0},R))^{\frac{1}{2n+8\alpha}}$, and so the last inequality of \eqref{eq:radiuscompute} is certainly satisfied provided $\frac{72}{500}>(\excesso(u;x_{0},R))^{\frac{1}{2n+8\alpha}}$, which in turn follows from \eqref{eq:achoose}. Now, succesively employing \eqref{eq:annulidefine}, $t_{N}\leq \kappa_{8N}\leq \tfrac{7}{8}R$, \eqref{eq:epsilonchoosemain} and \eqref{eq:achoose}, we similarly arrive at $s_{N}+2\lambda_{\con}\sqrt{n}\varepsilon<R$. Thus, $\mathcal{A}_{k}\subset\ball(x_{0},R)$ for all $k\in\{1,...,N\}$. 

By pairwise disjointness of the $\mathcal{A}_{k}$'s and $\mathcal{A}_{k}\subset\ball(x_{0},R)$, we can therefore conclude that there exists $k'\in\{1,...,N\}$ such that 
\begin{align*}
N\int_{\mathcal{A}_{k'}}f_{a}(\E\widetilde{u}_{a})\leq \int_{\mathcal{A}_{1}}f_{a}(\E\widetilde{u}_{a}) +...+\int_{\mathcal{A}_{N}}f_{a}(\E\widetilde{u}_{a}) \leq \int_{\ball(x_{0},R)}f_{a}(\E\widetilde{u}_{a}).
\end{align*}
To extract information from this estimate, we employ the lower bound on $N$, cf. \eqref{eq:exit}ff., to obtain
\begin{align}\label{eq:annulusest2}
\int_{\mathcal{A}_{k'}}f_{a}(\E\widetilde{u}_{a})\leq  \frac{1}{15} (\excesso(u;x_{0},R))^{\frac{1}{2n+8\alpha}}\int_{\ball(x_{0},R)}f_{a}(\E\widetilde{u}_{a}),
\end{align}
For future purposes, let us particularly remark that 
\begin{align*}
\frac{\varepsilon}{s_{k'}-t_{k'}} = \frac{500}{R(\excesso(u;x_{0},R))^{\frac{1}{2n+8\alpha}}}\frac{R(\excesso(u;x_{0},R))^{\frac{1}{n+4\alpha}}}{48\lambda_{\con}\sqrt{n}} \stackrel{\eqref{eq:achoose}}{\leq} \frac{1}{\lambda_{\con}\sqrt{n}}\frac{11}{10} < 1 
\end{align*}
because of $\sqrt{n}\geq\sqrt{2}$ and hence
\begin{align}\label{eq:smallerthanone}
\max\left\{\left(\frac{\varepsilon}{s_{k'}-t_{k'}}\right),\left(\frac{\varepsilon}{s_{k'}-t_{k'}}\right)^{2}\right\} \leq 10(\excesso(u;x_{0},R))^{\frac{1}{2n+8\alpha}}.
\end{align}

\emph{Step 3. Comparison estimates.} Let now $r_{k'},t_{k'},s_{k'}$ be defined as in step 2 so that $r<r_{k'}<\tfrac{R}{2}<t_{k'}<s_{k'}<R$. We define a Lipschitz function $\rho\colon\ball(x_{0},R)\to[0,1]$ by 
\begin{align}\label{eq:cutoff}
\rho(x):=\frac{2}{s_{k'}-t_{k'}}(|x|-t_{k'})\mathbbm{1}_{\{t_{k'}\leq |x|\leq (s_{k'}+t_{k'})/2\}}(x)+\mathbbm{1}_{\{|x|>(s_{k'}+t_{k'})/2\}}(x)
\end{align}
for $x\in\ball(x_{0},R)$. Then we have $\psi:=\widetilde{v}_{a}+\rho(\widetilde{u}_{a}-\widetilde{v}_{a})\in \bd(\ball(x_{0},R))$ and, in particular, $\psi|_{\partial\!\ball(x_{0},s_{k'})}=\widetilde{u}_{a}|_{\partial\!\ball(x_{0},s_{k'})}$. Since thus $u|_{\partial\!\ball(x_{0},s_{k'})}=(\psi+A_{x_{0}})|_{\partial\!\ball(x_{0},s_{k'})}$ $\mathscr{H}^{n-1}$-a.e. on $\partial\!\ball(x_{0},s_{k'})$, generalised local minimality of $u$ for $F$ implies by virtue of the integration by parts formula \eqref{eq:GaussGreen}
\begin{align*}
\int_{\ball(x_{0},s_{k'})}f_{a}(\E\widetilde{u}_{a}) & =\int_{\ball(x_{0},s_{k'})}f(\E u) - f(a) - \langle f'(a),\E\widetilde{u}_{a}\rangle \\
& \leq \int_{\ball(x_{0},s_{k'})}f(\E\,(\psi+A_{x_{0}})) - f(a) - \langle f'(a),\E\psi\rangle = \int_{\ball(x_{0},s_{k'})}f_{a}(\E\psi). 
\end{align*}
Splitting $\ball(x_{0},s_{k'})$ according to the definition of $\rho$, we consequently arrive at
\begin{align*}
\overline{F}_{a}[\widetilde{u}_{a};\ball(x_{0},r_{k'})] & + \overline{F}_{a}[\widetilde{u}_{a};\ball(x_{0},t_{k'})\setminus\ball(x_{0},r_{k'})]  + \overline{F}_{a}[\widetilde{u}_{a};\ball(x_{0},s_{k'})\setminus\ball(x_{0},t_{k'})] \\ & \leq \overline{F}_{a}[\widetilde{v}_{a};\ball(x_{0},r_{k'})]+\overline{F}_{a}[\widetilde{v}_{a};\ball(x_{0},t_{k'})\setminus\ball(x_{0},r_{k'})] \\ & +\overline{F}_{a}[(\widetilde{v}_{a}+\rho(\widetilde{u}_{a}-\widetilde{v}_{a}));\ball(x_{0},s_{k'})\setminus\ball(x_{0},t_{k'})]. 
\end{align*}
Regrouping terms and employing \eqref{eq:usmall}, we consequently arrive at
\begin{align*}
\mathbf{I}:=\overline{F}_{a}&[\widetilde{u}_{a};\ball(x_{0},r_{k'})] \leq \Big[\overline{F}_{a}[\widetilde{v}_{a};\ball(x_{0},r_{k'})] \Big. \\ &\Big.+ \Big(\overline{F}_{a}[\widetilde{v}_{a};\ball(x_{0},t_{k'})\setminus\ball(x_{0},r_{k'})]-\overline{F}_{a}[\widetilde{u}_{a};\ball(x_{0},t_{k'})\setminus\ball(x_{0},r_{k'})]\Big) \Big. \\ & + \Big(\overline{F}_{a}[(\widetilde{v}_{a}+\rho(\widetilde{u}_{a}-\widetilde{v}_{a}));\ball(x_{0},s_{k'})\setminus\ball(x_{0},t_{k'})]-\overline{F}_{a}[\widetilde{u}_{a};\ball(x_{0},s_{k'})\setminus\ball(x_{0},t_{k'})\Big) \Big]\\ 
& =:\mathbf{II}+\mathbf{III}+\mathbf{IV}.
\end{align*}
\emph{Ad $\mathbf{I}$}. By Jensen's inequality and Lemma~\ref{lem:eest}\ref{item:eest2} in the first and \eqref{eq:loweruppermain}, $r\leq r_{k'}$ in the second step, we find
\begin{align}\label{eq:usmall}
\begin{split}
\excenew(u;x_{0},r) & \leq 4 \int_{\ball(x_{0},r)}V(\E u-a) \leq 4c_{0} \overline{F}_{a}[\widetilde{u}_{a};\ball(x_{0},r_{k'})]= 4c_{0}\mathbf{I}. 
\end{split}
\end{align}
\emph{Ad $\mathbf{II}$}. In a similar vein as in the estimation of \eqref{eq:usmall}, we recall $a=(\sg(v))_{x_{0},r}$ to obtain
\begin{align}
\mathbf{II} = \int_{\ball(x_{0},r_{k'})}f_{a}(\sg(v)-a)\dif x \leq c_{0}\,\excenew(v;x_{0},2r). 
\end{align}
\emph{Ad $\mathbf{III}$.} By our choice of $r_{k'},t_{k'}$ in step 2, cf.~\eqref{eq:rtestimate},  we use $\excenew(u;x_{0},R)<1$ to bound $\mathbf{III}$ by
\begin{align}
\mathbf{III} \leq \frac{50}{\lambda_{\con}\sqrt{n}}\Big(1+\frac{R}{r}\Big)(\excesso(u;x_{0},R))^{\frac{1}{2n+8\alpha}}\int_{\ball(x_{0},R)}f_{a}(\E\widetilde{u}_{a}).
\end{align}
\emph{Ad $\mathbf{IV}$.} This step of the proof crucially utilises the convolution inequality from Section~\ref{sec:poincare}, and to this end, we employ Corollary~\ref{cor:convolutioncontrol} with $L=\tfrac{1}{s_{k'}-t_{k'}}$. In combination with \eqref{eq:smallerthanone}, we hereafter obtain 
\begin{align}\label{eq:convolutionestrewritten}
\begin{split}
\int_{\mathcal{A}(x_{0};t_{k'},s_{k'})}V\Big(\frac{\widetilde{u}_{a}-\widetilde{v}_{a}}{s_{k'}-t_{k'}}\Big)\dif x & \leq c(n,\lambda_{\con})\excesso(u;x_{0},R)^{\frac{1}{2n+8\alpha}}\times \\ 
& \;\;\;\; \times\int_{\overline{\mathcal{A}(x_{0};t_{k'}-2\lambda_{\con}\sqrt{n}\varepsilon,s_{k'}+2\lambda_{\con}\sqrt{n}\varepsilon)}}V(\E\widetilde{u}_{a}).
\end{split}
\end{align}
We then arrive at the following string of inequalities: 
\begin{align*}
\mathbf{IV} &\stackrel{\text{Lemma}~\ref{lem:eest}\ref{item:eest2},\,\eqref{eq:loweruppermain}}{\leq} 4c_{0}\Big(\int_{\mathcal{A}(x_{0};t_{k'},s_{k'})}V(\E\widetilde{v}_{a})+V(\E\widetilde{u}_{a})+\int_{\mathcal{A}(x_{0};t_{k'},s_{k'})}V\Big(\frac{\widetilde{u}_{a}-\widetilde{v}_{a}}{s_{k'}-t_{k'}}\Big)\dif x\Big)\\
& \;\;\;\;\;\;\;\;\;\;\;\;\,\leq 8c_{0}\Big(\int_{\overline{\mathcal{A}(x_{0};t_{k'}-2\varepsilon,s_{k'}+2\varepsilon)}}V(\E\widetilde{u}_{a})+\int_{\mathcal{A}(x_{0};t_{k'},s_{k'})}V\Big(\frac{\widetilde{u}_{a}-\widetilde{v}_{a}}{s_{k'}-t_{k'}}\Big)\dif x\Big)\\
& \;\;\;\;\;\;\;\;\;\;\,\stackrel{\eqref{eq:convolutionestrewritten}}{\leq} 8c_{0}\Big(\int_{\overline{\mathcal{A}(x_{0};t_{k'}-2\varepsilon,s_{k'}+2\varepsilon)}}V(\E\widetilde{u}_{a})\Big. \\ & \Big. \;\;\;\;\;\;\;\;\;\;\;\;\;\;\;\;\;\;\;\;+c(n,\lambda_{\con})\excesso(u;x_{0},R)^{\frac{1}{2n+8\alpha}}\int_{\overline{\mathcal{A}(x_{0};t_{k'}-2\lambda_{\con}\sqrt{n}\varepsilon,s_{k'}+2\lambda_{\con}\sqrt{n}\varepsilon)}}V(\E\widetilde{u}_{a})\Big)\\
& \;\;\;\;\;\;\;\;\;\;\,\stackrel{\eqref{eq:loweruppermain}}{\leq} 8c_{0}^{2}\Big(\int_{\mathcal{A}(x_{0};t_{k'}-2\lambda_{\con}\sqrt{n}\varepsilon,s_{k'}+2\lambda_{\con}\sqrt{n}\varepsilon)}f_{a}(\E\widetilde{u}_{a}) \Big. \\ & \Big. \;\;\;\;\;\;\;\;\;\;\;\;\;\;\;\;\;\;\;\;\;\;\;\;\;\;\;\;\;\;\;\;\;\;\;\;\;\;\;\;\;\;\;\;\;\;\;\;\;\;\;\;\;\;\;\;+c(n,\lambda_{\con})\excesso(u;x_{0},R)^{\frac{1}{2n+8\alpha}}\int_{\ball(x_{0},R)}f_{a}(\E\widetilde{u}_{a})\Big)\\
&\;\;\;\;\;\;\;\;\;\;\,\stackrel{\eqref{eq:annulusest2}}{\leq} c(n,\lambda_{\con},c_{0})\excesso(u;x_{0},R)^{\frac{1}{2n+8\alpha}}\int_{\ball(x_{0},R)}f_{a}(\E\widetilde{u}_{a})
\end{align*}
where, in the final two steps, we have used that $
s_{k'}+2\lambda_{\con}\sqrt{n}\varepsilon < R$ as established in step 2. We may now gather the estimates for $\mathbf{I},...,\mathbf{IV}$ to obtain with a constant $c=c(n,\lambda_{\con},c_{0})>0$ 
\begin{align}\label{eq:almostherePR}
\excenew(u;x_{0},r) \leq c\Big(\excenew(v;x_{0},2r)+ \Big(1+\left(\frac{R}{r}\right) \Big)(\excesso(u;x_{0},R))^{\frac{1}{2n+8\alpha}}\int_{\ball(x_{0},R)}f_{a}(\E\widetilde{u}_{a})\Big).
\end{align}

\emph{Step 4. Conclusion.} In order to arrive at the requisite form of the preliminary decay estimate \eqref{eq:partial}, we estimate by succesive application of Jensen's inequality and \eqref{eq:loweruppermain}:
\begin{align}\label{eq:conclusionestimate1}
\begin{split}
\int_{\ball(x_{0},R)}f_{a}(\E\, & \widetilde{u}_{a}) \leq c_{0}\int_{\ball(x_{0},R)} V(\E u-a) \\ & \leq 2c_{0}\Big(\int_{\ball(x_{0},R)}V(\E u-\xi_{0})+\mathscr{L}^{n}(\ball(x_{0},R))V(\xi_{0}-a)\Big)\\ 
&\leq c(n,c_{0}) \Big(\int_{\ball(x_{0},R)}V(\E u-\xi_{0})+R^{n}\dashint_{\ball(x_{0},r)}V(\sg(v)-\xi_{0})\dif x\Big)\\ 
&\leq c(n,c_{0}) \Big(\int_{\ball(x_{0},R)}V(\E u-\xi_{0})+\left(\frac{R}{r}\right)^{n}\int_{\ball(x_{0},r)}V(|\E u_{\varepsilon,\varepsilon}-\xi_{0}|)\Big)\\ 
& \leq c(n,c_{0}) \Big(1+\left(\frac{R}{r}\right)^{n}\Big)\excenew(u;x_{0},R), 
\end{split}
\end{align}
the ultimate estimate being valid due to our choice $\xi_{0}=(\E u)_{x_{0},R}$ and $r+2\varepsilon<R$. Combining this estimate with \eqref{eq:almostherePR}, we obtain \eqref{eq:partial}, and the proof is complete. 
\end{proof}

\begin{proposition}\label{prop:propdecay}
In the situation of Proposition \ref{prop:main1} we have 
\begin{align}\label{eq:deviationestimate}
\devi(v;x_{0},\tfrac{R}{2})\leq c\Big(1+\Big(\frac{R}{r}\Big)^{n}\Big)\excesso(u;x_{0},R)^{\frac{1}{2n+8\alpha}}\excenew(u;x_{0},R). 
\end{align} 
\end{proposition}
\begin{proof}
Adopting the terminology of step 1 of the previous proof, we leave the setting unchanged up to formula \eqref{eq:akdef}. Instead of Lemma~\ref{lem:convex}\ref{item:AG2} we use Lemma~\ref{lem:convex}\ref{item:AG1} to find, for each $k\in\{1,...,N\}$, a number $t_{k}\in (\kappa_{8k-1},\kappa_{8k})$ such that 
\begin{align}\label{eq:newest11}
\int_{\ball(x_{0},t_{k'})}f_{a}(\E\widetilde{v}_{a})-\int_{\ball(x_{0},t_{k'})}f_{a}(\E\widetilde{u}_{a}) \leq \frac{50}{\lambda_{\con}\sqrt{n}}\widetilde{\Phi}(u;x_{0},R)^{\frac{1}{2n+8\alpha}}\int_{\ball(x_{0},R)}f_{a}(\E\widetilde{u}_{a}),
\end{align}
providing the requisite substitute for formula \eqref{eq:rtestimate}. Equally, we find $s_{k'}=t_{k'}+\frac{R}{500}\widetilde{\Phi}(u;x_{0},R)^{\frac{1}{2n+8\alpha}}$ such that  
\begin{align}\label{eq:newest12}
\int_{\mathcal{A}_{k'}}f_{a}(\E\widetilde{u}_{a}) \leq \frac{1}{15}(\widetilde{\Phi}(u;x_{0},R))^{\frac{1}{2n+8\alpha}}\int_{\ball(x_{0},R)}f_{a}(\E\widetilde{u}_{a}), 
\end{align}
the annulus $\mathcal{A}_{k'}$ now being defined as in \eqref{eq:annulidefine} with the obvious change of $t_{k'}$ and $s_{k'}$. Let $\theta>0$ be arbitrary. We then put 
\begin{align*}
&\mathcal{C}_{1}:=\{\varphi\in \sobo^{1,\infty}(\ball(x_{0},t_{k'});\R^{n})\colon\; \varphi = \widetilde{v}_{a}\;\text{on}\;\partial\!\ball(x_{0},t_{k'})\}\\
&\mathcal{C}_{2}:=\{\varphi\in \sobo^{1,\infty}(\mathcal{A}(x_{0};t_{k'},s_{k'});\R^{n})\colon\; \varphi = \widetilde{v}_{a}\;\text{on}\;\partial\!\mathcal{A}(x_{0};t_{k'},s_{k'})\}
\end{align*}
and find $\varphi_{1}\in\mathcal{C}_{1}$, $\varphi_{2}\in\mathcal{C}_{2}$ such that 
\begin{align}\label{eq:almostminimalPR}
\begin{split}
&\int_{\ball(x_{0},t_{k'})}f_{a}(\sg(\varphi_{1}))\dif x \leq \inf_{\varphi\in\mathcal{C}_{1}}\int_{\ball(x_{0},t_{k'})}f_{a}(\sg(\varphi))\dif x + \frac{\theta}{2} \\
&\int_{\mathcal{A}(x_{0};t_{k'},s_{k'})}f_{a}(\sg(\varphi_{2}))\dif x \leq \inf_{\varphi\in\mathcal{C}_{2}}\int_{\mathcal{A}(x_{0};t_{k'},s_{k'})}f_{a}(\sg(\varphi))\dif x + \frac{\theta}{2}. 
\end{split}
\end{align}
Let us note that, employing an integration by parts, for all $\varphi\in\mathcal{C}_{1}$ there holds 
\begin{align*}
\int_{\ball(x_{0},t_{k'})}f_{a}(\sg(\widetilde{v}_{a}))-f_{a}(\sg(\varphi))\dif x = \int_{\ball(x_{0},t_{k'})}f(\sg( v))-f(\sg(\varphi) + a)\dif x.
\end{align*}
By definition of $f_{a}$, we then obtain 
\begin{align}\label{eq:devirewrite}
\begin{split}
\devi(v;x_{0},t_{k'}) & =\int_{\ball(x_{0},t_{k'})}f_{a}(\sg(\widetilde{v}_{a}))\dif x \\ & -\inf\left\{\int_{\ball(x_{0},t_{k'})}f_{a}(\sg(\psi))\dif x\colon\begin{array}{c}
\psi\in \hold^{1,\alpha}(\overline{\ball(x_{0},t_{k'})};\R^{n}) \\ \psi = \widetilde{v}_{a}\;\text{on}\;\partial\!\ball(x_{0},t_{k'})
\end{array}\right\}.
\end{split}
\end{align}
Since $\varphi_{1},\varphi_{2}$ are Lipschitz and coincide on $\partial\!\ball(x_{0},t_{k'})$, we deduce that the glued map $\varphi_{3}:=\mathbbm{1}_{\overline{\ball(x_{0},t_{k'})}}\varphi_{1}+\mathbbm{1}_{\mathcal{A}(x_{0};t_{k'},s_{k'})}\varphi_{2}$ belongs to $\sobo^{1,\infty}(\ball(x_{0},s_{k'});\R^{n})$. We then obtain, using that $t\mapsto\devi(v;x_{0},t)$ is non-decreasing in the first inequality,
\begin{align*}
\devi(v;x_{0},\tfrac{R}{2})&\;\;\;\;\;\;\;\;\;\;\;\;\;\;\;\,\leq \devi(v;x_{0},t_{k'}) \\ & \stackrel{\eqref{eq:devirewrite},\,\hold^{1,\alpha}\subset\sobo^{1,\infty},\,\eqref{eq:almostminimalPR}_{1}}{\leq} \int_{\ball(x_{0},t_{k'})}f_{a}(\sg(\widetilde{v}_{a}))\dif x - \int_{\ball(x_{0},t_{k'})}f_{a}(\sg(\varphi_{1}))\dif x + \frac{\theta}{2} \\
&\;\;\;\;\;\;\;\;\;\;\;\;\;\stackrel{\eqref{eq:almostminimalPR}_{2}}{\leq} \int_{\ball(x_{0},s_{k'})}f_{a}(\sg(\widetilde{v}_{a}))\dif x - \int_{\ball(x_{0},s_{k'})}f_{a}(\sg(\varphi_{3}))\dif x + \theta\\
&\;\;\;\;\;\;\;\;\;\;\;\;\;\;\;\;= \Big(\int_{\ball(x_{0},s_{k'})}\big(f_{a}(\sg(\widetilde{v}_{a})\mathscr{L}^{n})-f_{a}(\E\widetilde{u}_{a})\big)\Big) \\ &\;\;\;\;\;\;\;\;\;\;\;\;\;\;\;\;\;\;\;\;\;\;\;\; + \Big(\int_{\ball(x_{0},s_{k'})}\big(f_{a}(\E\widetilde{u}_{a})-f_{a}(\sg(\varphi_{3})\mathscr{L}^{n})\big)\Big) + \theta \\ & \;\;\;\;\;\;\;\;\;\;\;\;\;\;\;\;=: \mathbf{V}+\mathbf{VI}+\theta.
\end{align*}
\emph{Ad $\mathbf{V}$.} Splitting $\ball(x_{0},s_{k'})=\ball(x_{0},t_{k'})\cup\mathcal{A}(x_{0};s_{k'},t_{k'})$ and employing \eqref{eq:newest12}, we obtain 
\begin{align*}
\mathbf{V} & \stackrel{\eqref{eq:newest12}}{\leq} \Big( \int_{\mathcal{A}(x_{0};t_{k'},s_{k'})}f_{a}(\E\widetilde{v}_{a})-f_{a}(\E\widetilde{u}_{a})\Big) + \frac{1}{15}\widetilde{\Phi}(u;x_{0},R)^{\frac{1}{2n+8\alpha}}\int_{\ball(x_{0},R)}f_{a}(\E\widetilde{u}_{a}) \\ 
& \stackrel{\eqref{eq:newest11}}{\leq} c(c_{0},n,\lambda_{\con})\widetilde{\Phi}(u;x_{0},R)^{\frac{1}{2n+8\alpha}}\int_{\ball(x_{0},R)}f_{a}(\E\widetilde{u}_{a}).
\end{align*}
\emph{Ad $\mathbf{VI}$.} Different from step 3 of the proof of Proposition~\ref{prop:main1}, we now use the comparison map $\overline{\psi}:=\varphi_{3}+\rho(\widetilde{u}_{a}-\widetilde{v}_{a})$, $\rho$ still being defined by \eqref{eq:cutoff} but now with the new choices of $t_{k'}$ and $s_{k'}$. In advance, we note that $\overline{\psi}=\widetilde{u}_{a}$ $\mathscr{H}^{n-1}$-a.e. on $\partial\!\ball(x_{0},s_{k'})$. Since 
\begin{align}\label{eq:newest13}
\begin{split}
\int_{\mathcal{A}(x_{0};t_{k'},s_{k'})}& V(\sg(\varphi_{3})\mathscr{L}^{n}) \stackrel{\eqref{eq:loweruppermain}}{\leq} c_{0}\int_{\mathcal{A}(x_{0};t_{k'},s_{k'})}f_{a}(\E\varphi_{3})\dif x \\ & \!\!\!\!\!\!\!\stackrel{\eqref{eq:almostminimalPR}_{2}}{\leq} c_{0}\Big(\int_{\mathcal{A}(x_{0};t_{k'},s_{k'})}f_{a}(\E\widetilde{v}_{a})\dif x + \frac{\theta}{2}\Big) \stackrel{\text{Jensen}}{\leq} c_{0}\Big(\int_{\mathcal{A}_{k'}}f_{a}(\E\widetilde{u}_{a}) + \frac{\theta}{2}\Big)
\end{split}  
\end{align}
Generalised local minimality of $u$ for $F$ and $\rho|_{\ball(x_{0},t_{k'})}=0$ yields 
\begin{align*}
\mathbf{VI} & \leq \int_{\mathcal{A}(x_{0},t_{k'},s_{k'})}f_{a}(\E\,(\varphi_{3}+\rho(\widetilde{u}_{a}-\widetilde{v}_{a})))-f_{a}(\sg(\varphi_{3})\mathscr{L}^{n}) \\
& \leq 4c_{0}\Big( \int_{\mathcal{A}(x_{0},t_{k'},s_{k'})}V(\sg(\varphi_{3})\mathscr{L}^{n})+V(\E\widetilde{u}_{a})+V(\sg(\widetilde{v}_{a})\mathscr{L}^{n})\Big) \\ & \;\;\;\;\;\;\;\;\;\;\;\;\;\;\;\;\;\;\;\;\;\;\;\; + 4c_{0}\int_{\mathcal{A}(x_{0},t_{k'},s_{k'})}V\left(\frac{\widetilde{u}_{a}-\widetilde{v}_{a}}{s_{k'}-t_{k'}}\right)\dif x \\ 
& \!\!\!\!\!\!\!\!\!\!\!\!\!\!\!\!\!\!\!\!\!\stackrel{\eqref{eq:newest13},\,\text{Cor.}~\ref{cor:convolutioncontrol},\,\eqref{eq:newest12}}{\leq} c(c_{0},n,\lambda_{\con})\Big(\int_{\mathcal{A}_{k'}}f_{a}(\E\widetilde{u}_{a}) + \widetilde{\Phi}(u;x_{0},R)^{\frac{1}{2n+8\alpha}}\int_{\mathcal{A}_{k'}}f_{a}(\E\widetilde{u}_{a})+\frac{\theta}{2} \Big) \\
& \!\!\!\!\stackrel{\eqref{eq:newest12}}{\leq} c(c_{0},n,\lambda_{\con})\Big((\widetilde{\Phi}(u;x_{0},R))^{\frac{1}{2n+8\alpha}}\int_{\ball(x_{0},R)}f_{a}(\E\widetilde{u}_{a}) + \frac{\theta}{2} \Big)
\end{align*}
Combining the estimates for $\mathbf{V}$ and $\mathbf{VI}$ yields 
\begin{align*}
\devi(v;x_{0},\frac{R}{2}) \leq c(c_{0},n,\lambda_{\con})\Big((\widetilde{\Phi}(u;x_{0},R))^{\frac{1}{2n+8\alpha}}\int_{\ball(x_{0},R)}f_{a}(\E\widetilde{u}_{a}) + \theta \Big). 
\end{align*}
Now we employ \eqref{eq:conclusionestimate1} and send $\theta\searrow 0$ to conclude. The proof is complete.  
\end{proof}
\begin{remark}\label{rem:discussexponent}
As mentioned after Proposition~\ref{prop:convest}, the crude but easier obtainable estimate \eqref{eq:betaestimate} is \emph{not sufficient} for applications in the proof of Propositions~\ref{prop:main1} and \ref{prop:propdecay}. In fact, by the above proof we are bound to set $L=\frac{1}{s_{k'}-t_{k'}}$. With the particular choice of $\varepsilon$ by \eqref{eq:epsilonchoosemain}, we then find that $L\varepsilon^{\beta}$ cannot be suitably bounded to still arrive at the requisite decay estimate. On the other hand, one might redefine $\varepsilon$, but then estimates of the remaining proof cannot be obtained in the requisite form and the decay estimate cannot be inferred. 
\end{remark}
\begin{corollary}[$\varepsilon$-regularity]\label{cor:epsreg}
Let $f\in\hold^{2}(\rsym)$ be a convex integrand with \eqref{eq:lingrowth1} and suppose that there exist $z_{0}\in\rsym$ and $\varrho_{z_{0}}>0$ such that the following hold:
\begin{enumerate}[label={(C\arabic{*})}]
\item\label{C1} $m_{z_{0},\varrho_{z_{0}}}>0$, 
\item\label{C2} For all $\xi,\xi'\in\mathbb{B}(z_{0},\varrho_{z_{0}})$ there holds  $|f''(\xi)-f''(\xi')|\leq \omega(|\xi-\xi'|)$ with a bounded and non-decreasing function $\omega\colon\R_{\geq 0}\to\R_{\geq 0}$ such that $\lim_{t\searrow 0}\omega(t)=0$.  
\end{enumerate}
For any $\alpha\in (0,1)$ there exist $\varepsilon_{0}\in (0,1]$ and $\sigma\in(0,1)$ such that the following holds for all $u\in\gm_{\locc}(F)$: If $x\in\Omega$ and $R>0$ are such that $\ball(x,R)\Subset\Omega$ and 
\begin{enumerate}
\item\label{item:smallness1} $|(\E u)_{x,R}-z_{0}|<\varrho_{z_{0}}/3$, 
\item\label{item:smallness2} $\widetilde{\Phi}(u;x,R)<\varepsilon_{0}$, 
\end{enumerate}
then there holds $\widetilde{\Phi}(u;x,\sigma^{j}R)\leq \sigma^{2\alpha j}\widetilde{\Phi}(u;x,R)$ for all $j\in\mathbb{N}_{0}$. In particular, $\varepsilon_{0}$ and $\sigma$ only depend on $n,\lambda_{\con},\varrho_{z_{0}},m_{z_{0},\varrho_{z_{0}}},\omega,c_{1},c_{2}$ and $\sup_{\mathbb{B}(z_{0},\varrho_{z_{0}})}|f''|$. 
\end{corollary}
\begin{proof}
Let $\alpha\in (0,1)$, $x\in\Omega$ and $R>0$ be such that $\ball(x,R)\subset\Omega$. Note that, if $\xi_{0}\in\mathbb{B}(z_{0},2\varrho_{\xi_{0}})$ with $\varrho_{\xi_{0}}:=\varrho_{z_{0}}/3$, then $m_{\xi_{0},\varrho_{\xi_{0}}}>0$ and $|f''(\xi)-f''(\xi_{0})|\leq \omega_{\xi_{0},\varrho_{\xi_{0}}}(|\xi-\xi_{0}|)$ for all $\xi\in\mathbb{B}(\xi_{0},\varrho_{\xi_{0}})$ with $\omega_{\xi_{0},\varrho_{\xi_{0}}}=\omega$. We put $\xi_{0}:=(\E u)_{x,R}$.  

We pick the constants $\Theta,c>0$ from Proposition~\ref{prop:main1} with $x_{0}=x$, fix the mollification parameter $\varepsilon$ as in \eqref{eq:epsilonchoosemain} and let $\varepsilon_{0}\in (0,\Theta)$ to be fixed later on. Thus, for all $0<r<\frac{R}{4}$, \eqref{eq:partial} is in action with $v=u_{\varepsilon,\varepsilon}$. From Lemma~\ref{lem:adjust} applied to the radius $R$, we obtain $\mathbf{t}_{\alpha,\xi_{0}}(v;x,\frac{R}{2})\leq c(n,\alpha)\widetilde{\Phi}(u;x,R)^{\frac{\alpha}{n+4\alpha}}$, $c(n,\alpha)>0$ denoting the constant from Lemma \ref{lem:adjust}. Thus, diminishing $\varepsilon_{0}$, we may assume that $c(n,\alpha)\varepsilon_{0}^{\frac{\alpha}{n+4\alpha}}<\min\{\frac{\varrho_{\xi_{0}}}{c_{\comp}},\frac{1}{2}\}$, $c_{\comp}>0$ being the constant from Proposition~\ref{prop:1}. This entails $\mathbf{t}_{\alpha,\xi_{0}}(v;x,\frac{R}{2})<\frac{1}{2}$  and thus $|\sg(v)(y)-(\sg(v))_{x,r}|\leq 1$ for all $0<r<\frac{R}{4}$ and $y\in\ball(x,r)$. Therefore, a consecutive application of Proposition~\ref{prop:1},  Lemma~\ref{lem:eest}\ref{item:eest3}, Corollary~\ref{cor:adjust} and Proposition~\ref{prop:propdecay} yields 
\begin{align*}
\Phi(v;x,2r)  & \leq c\Big(\Big(\frac{r}{R}\Big)^{n+2} + \Big(1+\Big(\frac{R}{r}\Big)^{n+1}\Big)\times \Big. \\ & \Big. \;\;\;\;\; \times\Big(\vartheta(c(n,\alpha)\widetilde{\Phi}(u;x,R)^{\frac{\alpha}{n+4\alpha}}) + \widetilde{\Phi}(u;x,R)^{\frac{1}{2n+8\alpha}}\Big)\Big)\times \Phi(u;x,R).
\end{align*}
for all $0<r<\frac{R}{4}$. In conclusion, \eqref{eq:partial} yields the existence of $\varepsilon_{0}^{(1)}\in (0,1]$ and $c_{\decay}>0$ such that there holds
\begin{align}\label{eq:iterationdecay}
\widetilde{\Phi} (u;x,r) \leq c_{\decay}\Big(\frac{r}{R}\Big)^{2}\Big(1 + \Big(\frac{R}{r}\Big)^{2n+3}H(\widetilde{\Phi}(u;x,R)) \Big)\widetilde{\Phi}(u;x,R)
\end{align}
for all $0<r<\frac{R}{4}$, the non-negative function $H\colon\R_{\geq 0}\to\R_{\geq 0}$ being given by 
\begin{align*}
H(t):= \vartheta(c(n,\alpha)t^{\frac{\alpha}{n+4\alpha}})+ t^{\frac{1}{2n+8\alpha}},  
\end{align*}
cf.~Proposition~\ref{prop:1} for the introduction of $\vartheta$. Tracking dependencies, $\varepsilon_{0}^{(1)}$, $c_{\decay}$ and $H$ only depend on $n,\lambda_{\con},\varrho_{z_{0}},m_{z_{0},\varrho_{z_{0}}},\omega,\Lip(f)$ and $\sup_{\mathbb{B}(z_{0},\varrho_{z_{0}})}|f''|$. We now define
\begin{align}\label{eq:defsigmacontrol}
\sigma := \min\left\{\sqrt[\alpha]{\tfrac{\sqrt{\sqrt{3}(\sqrt{2}-1)}}{2\sqrt{12}}},\tfrac{\varrho_{z_{0}}}{6},\sqrt[2(1-\alpha)]{\tfrac{1}{2c_{\decay}}} \right\}
\end{align}
and, using that $\lim_{t\searrow 0}H(t)=0$,  choose $\varepsilon_{0}^{(2)}>0$ such that there holds 
\begin{align}\label{eq:epsiloniterationchoose}
\varepsilon_{0}^{(2)}<\sigma^{n+2}\;\;\;\text{and}\;\;\;\sup\{H(t)\colon 0<t<\varepsilon_{0}^{(2)}\}\leq \sigma^{2n+3}. 
\end{align}
We now define $\varepsilon_{0}:=\min\{\varepsilon_{0}^{(1)},\varepsilon_{0}^{(2)}\}$ and claim that, if $x\in\Omega$ and $R>0$ are such that $\ball(x,R)\Subset\Omega$ with $|(\E u)_{x,R}-z_{0}|<\varrho_{z_{0}}/3$ and $\widetilde{\Phi}(u;x,R)<\varepsilon_{0}$, then there holds $\widetilde{\Phi}(u;x,\sigma^{j}R)\leq \sigma^{\alpha j}\widetilde{\Phi}(u;x,R)$ for all $j\in\mathbb{N}_{0}$. To conclude the proof by iteration, we put 
\begin{align*}
\xi_{0}^{(j)}:=\frac{\E u(\ball(x,\sigma^{j}R))}{\mathscr{L}^{n}(\ball(x,\sigma^{j}R))},\qquad j\in\mathbb{N}_{0},
\end{align*}
and establish validity of 
\begin{align}\label{eq:induction1}\tag{$\text{Dec}_{j}$}
\excesso(u;x,\sigma^{j} R)\leq \sigma^{2\alpha j}\excesso(u;x,R)\;\;\;\text{and}\;\;\;|z_{0}-\xi_{0}^{(j)}|\leq \frac{1}{3}\varrho_{z_{0}}\sum_{i=0}^{j}\frac{1}{2^{i}} 
\end{align}
for all $j\in\mathbb{N}_{0}$. This is trivial for $j=0$. Now assume validity of \eqref{eq:induction1} for some $j\in\mathbb{N}_{0}$, the second part of which implies $\xi_{0}^{(j)}\in\mathbb{B}(z_{0},\tfrac{2}{3}\varrho_{z_{0}})$ so that \ref{C1} and \ref{C2} continue to hold in $\mathbb{B}(\xi_{0}^{(j)},\frac{1}{3}\varrho_{z_{0}})$. Moreover, the first part of \eqref{eq:induction1} yields $\widetilde{\Phi}(u;x,\sigma^{j}R)<\varepsilon_{0}^{(1)}$. Therefore,  
\begin{align*}
\widetilde{\Phi}(u;x,\sigma^{j+1}R) & \stackrel{\eqref{eq:iterationdecay}}{\leq}  \sigma^{2\alpha}(c_{\decay}\sigma^{2-2\alpha}) \Big(1 + \frac{H(\widetilde{\Phi}(u;x,\sigma^{j}R))}{\sigma^{2n+3}} \Big)\widetilde{\Phi}(u;x,\sigma^{j}R) \\ 
& \stackrel{\eqref{eq:epsiloniterationchoose}}{\leq} \sigma^{2\alpha}(2c_{\decay}\sigma^{2-2\alpha})\widetilde{\Phi}(u;x,\sigma^{j}R) \stackrel{\eqref{eq:defsigmacontrol},\,\eqref{eq:induction1}}{\leq} \sigma^{2\alpha(j+1)}\widetilde{\Phi}(u;x,R).
\end{align*} 
Toward the second part of \eqref{eq:induction1}, it suffices to establish $|\xi_{0}^{(j+1)}-\xi_{0}^{(j)}|\leq \frac{1}{3}\varrho_{z_{0}}2^{-j-1}$. Observe that 
\begin{align}\label{eq:inductionV}
\begin{split}
V(|\xi_{0}^{(j+1)}-\xi_{0}^{(j)}|) & \leq \dashint_{\ball(x,\sigma^{j+1}R)}V(|\E u - \xi_{0}^{(j)}|) \\ & \leq \frac{1}{\sigma^{n}}\dashint_{\ball(x,\sigma^{j}R)}V(|\E u - \xi_{0}^{(j)}|) \\ & \!\!\!\!\! \stackrel{\eqref{eq:induction1}}{\leq}  \sigma^{2\alpha j-n}\excesso(u;x,R) \leq \sigma^{2\alpha j +2}\frac{\varepsilon_{0}}{\sigma^{n+2}} \stackrel{\eqref{eq:epsiloniterationchoose},\,0<\sigma<1}{<} 1, 
\end{split}
\end{align}
which, by definition of $V$, entails $|\xi_{0}^{(j+1)}-\xi_{0}^{(j)}|<\sqrt{3}$. Therefore, by Lemma~\ref{lem:eest}\ref{item:eest1},~\ref{item:eest3}, 
\begin{align*}
|\xi_{0}^{(j+1)}-\xi_{0}^{(j)}| & \leq \sqrt{\frac{12}{\sqrt{3}(\sqrt{2}-1)}}\sqrt{V(|\xi_{0}^{(j+1)}-\xi_{0}^{(j)}|)} \\ &  \!\!\!\!\stackrel{\eqref{eq:inductionV}}{\leq} \sqrt{\frac{12}{\sqrt{3}(\sqrt{2}-1)}} \sigma^{\alpha j+1} \stackrel{\eqref{eq:defsigmacontrol}}{\leq} \frac{1}{3}\rho_{z_{0}}2^{-j-1}.
\end{align*}
The proof of the corollary is thereby complete.  
\end{proof}
\subsection{Proof of Theorem~\ref{thm:PR}}\label{sec:proofPR}
We can now proceed to the 
\begin{proof}[Proof of Theorem~\ref{thm:PR}]
Let $u\in\gm(F;u_{0})$ and $(x_{0},z_{0})\in\Omega\times\rsym$ be such that $f''(z_{0})$ is positive definite and  \eqref{eq:theimportantcondition} is satisfied. Since $f''$ is continuous, there exists $\varrho_{z_{0}}>0$ such that \ref{C1} and \ref{C2} from Corollary~\ref{cor:epsreg} are satisfied. Let $\varepsilon_{0}>0$ be as in Corollary~\ref{cor:epsreg}. By \eqref{eq:theimportantcondition}, $\lim_{R\searrow 0}(|\mathscr{E}u-z_{0}|)_{x_{0},R}+|(\E^{s}u)_{x_{0},R}|=0$, and since $V(\cdot)\leq|\cdot|$, 
\begin{align*}
\widetilde{\Phi}(u;x_{0},R) & \leq  2\Big(\dashint_{\ball(x_{0},R)}|\mathscr{E}u-z_{0}|\dif\mathscr{L}^{n}+ \frac{|\E u|(\ball(x_{0},R))}{\mathscr{L}^{n}(\ball(x_{0},R))}\Big)\stackrel{\eqref{eq:theimportantcondition}}{\longrightarrow} 0,\qquad R\searrow 0.
\end{align*}
By \eqref{eq:theimportantcondition} and $\mathscr{E}u\in\lebe_{\locc}^{1}(\Omega;\rsym)$, we conclude that there exists some $R_{0}>0$ and an open neighbourhood $U_{1}$ of $x_{0}$ such that $\ball(x_{0},2R_{0})\Subset\Omega$ and
\begin{align*}
\left\vert\dashint_{\ball(x,R_{0})}\mathscr{E}u\dif\mathscr{L}^{n}-z_{0}\right\vert < \min\Big\{\frac{\varepsilon_{0}}{4},\frac{\varrho_{z_{0}}}{6}\Big\},\;\;\;\frac{|\E^{s}u|(\ball(x_{0},2R_{0}))}{\mathscr{L}^{n}(\ball(x_{0},2R_{0}))}<\frac{1}{2^{n+1}}\min\Big\{\frac{\varepsilon_{0}}{4},\frac{\varrho_{z_{0}}}{3}\Big\}
\end{align*}
hold for all $x\in U_{1}$. Diminishing $U_{1}$ if necessary, we can assume that $U_{1}\subset \ball(x_{0},R_{0})$. Let $x\in U_{1}$, so that $\ball(x,R_{0})\subset\ball(x_{0},2R_{0})$. Thus,
\begin{align*}
\left\vert \frac{\E u(\ball(x,R_{0}))}{\mathscr{L}^{n}(\ball(x,R_{0}))}-z_{0}\right\vert & \leq \left\vert\dashint_{\ball(x,R_{0})}\mathscr{E}u\dif\mathscr{L}^{n}-z_{0}\right\vert + 2^{n}\frac{|\E^{s}u|(\ball(x_{0},2R_{0}))}{\mathscr{L}^{n}(\ball(x_{0},2R_{0}))}< \frac{\varrho_{z_{0}}}{3}.
\end{align*}
On the other hand, since $V(\cdot)\leq|\cdot|$ and $V^{\infty}(\cdot)=|\cdot|$, 
\begin{align*}
\excesso(u;x,R_{0}) & \leq 2\dashint_{\ball(x,R_{0})}|\mathscr{E}u-z_{0}|\dif\mathscr{L}^{n} + 2^{n+1}\frac{|\E^{s}u|(\ball(x_{0},2R_{0}))}{\mathscr{L}^{n}(\ball(x_{0},2R_{0}))}<\varepsilon_{0}. 
\end{align*}
As a conclusion, conditions \ref{item:smallness1} and \ref{item:smallness2} from Corollary~\ref{cor:epsreg} are satisfied for all $x\in U_{1}$ and $R=R_{0}$. Therefore, there exists $C>0$ such that $\widetilde{\Phi}(u;x,r)\leq C (r/R_{0})^{2\alpha}\widetilde{\Phi}(u;x,R_{0})$ holds for all $x\in U_{1}$ and $0<r<R_{0}/4$. By definition of $\widetilde{\Phi}$, this implies $\E^{s}u\equiv 0$ in $U_{1}$ and hence $\E u\mres U_{1}= \sg(u)\mathscr{L}^{n}\mres U_{1}$. Now, for all such $x$ and $r$,
\begin{align*}
V\Big(\dashint_{\ball(x,r)}|\sg(u)-(\sg(u))_{x,r}|\dif\mathscr{L}^{n}\Big)\leq \dashint_{\ball(x,r)}V(\sg(u)-(\sg( u))_{x,r})\dif\mathscr{L}^{n} \leq C(R_{0},\varepsilon_{0})r^{2\alpha} 
\end{align*}
and so, Lemma~\ref{lem:eest}\ref{item:eest4} yields a constant $c(R_{0},\varepsilon_{0})>0$ such that  
\begin{align*}
\dashint_{\ball(x,r)}|\sg(u)-(\sg(u))_{x,r}|\dif\mathscr{L}^{n}\leq \Big(c\dashint_{\ball(x,r)}V(\sg(u)-(\sg( u))_{x,r})\dif\mathscr{L}^{n}\Big)^{\frac{1}{2}} \leq cr^{\alpha}<1.  
\end{align*}
Now, the usual Campanato-Meyers characterisation of H\"{o}lder continuity \cite[Thm.~2.9]{Giusti} implies that $\sg(u)$ is of class $\hold^{0,\alpha}$ and hence $\lebe^{2}$  in a neighbourhood $\widetilde{U}$ of $x_{0}$. Thus, by Lemma~\ref{lem:Cianchi}~\ref{item:Cianchi0}, 
\begin{align*}
\dashint_{\ball(x,r)}|\nabla u-(\nabla u)_{x,r}|^{2}\dif\mathscr{L}^{n} \leq c \dashint_{\ball(x,r)}|\sg(u)-(\sg(u))_{x,r}|^{2}\dif\mathscr{L}^{n} \leq cr^{2\alpha}. 
\end{align*}
We again invoke the Campanato-Meyers characterisation of H\"{o}lder continuity to conclude that $u$ is of class $\hold^{1,\alpha}$ in an open neighbourhood of $x_{0}$. Finally, by the Lebesgue differentiation theorem for Radon measures, cf.~\eqref{eq:Lebesgue}, $\mathscr{L}^{n}$-a.e. $x_{0}\in\Omega$ satisfies \eqref{eq:theimportantcondition}, and the proof of the theorem is complete.
\end{proof}

\section{Remarks and extensions}\label{sec:extensions}
We conclude the paper with some remarks on possible generalisations of Theorems~\ref{thm:W11reg} and \ref{thm:PR} with focus on non-autonomous problems. First, by the very nature of the proofs, Theorem~\ref{thm:W11reg} and \ref{thm:PR} straightforwardly generalise to local generalised minima. Second, in analogy with \cite[Sec.~6]{AG}, if $f\colon\Omega\times\rsym\to\R$ and $g\colon\Omega\times\R^{n}\to\R$ are such that 
\begin{align*}
\begin{cases}
z\mapsto f(x,z)\;\text{is of class $\hold^{2}$}&\; \text{for all $x\in\Omega$},\\
c_{1}|z|-\gamma\leq f(x,z) \leq c_{2}(1+|z|)&\;\text{for all}\;x\in\Omega,\;z\in\rsym,\\
|f(x_{1},z)-f(x_{2},z)|\leq c_{3}|x_{1}-x_{2}|^{\mu}(1+|z|)&\;\text{for all}\;x_{1},x_{2}\in\Omega,\;z\in\rsym,\\
|g(x_{1},y_{1})-g(x_{2},y_{2})|\leq c_{4}| |x_{1}-x_{2}|+|y_{1}-y_{2}||^{\mu}&\;\text{for all}\;x_{1},x_{2}\in\Omega,\;y_{1},y_{2}\in\R^{n}, 
\end{cases}
\end{align*}
for some $c_{1},...,c_{4}>0$, $\gamma>0$ and $0<\mu<1$, then Theorem~\ref{thm:PR} generalises to functionals 
\begin{align}\label{eq:xdep}
F[u]:=\int_{\Omega}f(x,\sg(u))\dif x + \int_{\Omega}g(x,u)\dif x.
\end{align}
More precisely, let $u\in\gm_{\locc}(F)$ and suppose that $(x_{0},z_{0})\in\Omega\times\rsym$ is such that $z_{0}$ is the Lebesgue value of $\E u$ at $x_{0}$, Moreover, assume that there exists $\lambda>0$ such that $\lambda|z|^{2}\leq \langle \D_{z}^{2}f(x,z_{0})z,z\rangle$ holds for all $z\in\rsym$ uniformly in an open neighbourhood of $x_{0}$. Then there exists an open neighbourhood $U$ of $x_{0}$ such that $u$ has H\"{o}lder continuous full gradients in $U$. Let us, however, note that a corresponding result is far from clear if the overall variational integrand $(x,y,z)\mapsto f(x,z)+g(x,y)$ does not possess the splitting structure but is of the general form $(x,y,z)\mapsto f(x,y,z)$. 

Namely, here one usually invokes Caccioppoli's inequality in conjunction with Gehring's lemma on higher integrability to conclude that minima of elliptic problems belong to some $\sobo_{\locc}^{1,r}$, $p>r$, where $p$ is the Lebesgue exponent of the natural energy space $\sobo^{1,p}$. As explained in \cite{GK2}, there exist linear growth integrands and generalised minimisers $u\in\bv\setminus\sobo^{1,1}$ which \emph{do} satisfy a Caccioppoli type inequality. This easily carries over to the $\bd$-situation, and hereby rules out any integrability boost by virtue of Gehring. On the other hand, even for semiautonomous integrands $(x,z)\mapsto f(x,z)$, a well-known counterexample due to \textsc{Bildhauer} \cite[Thm.~4.39]{Bi1} asserts that if $f\in\hold^{2}(\overline{\Omega}\times \R^{N\times n})$ satisfies a uniform variant of \eqref{eq:ellipticity} for $a>3$, then generalised minima might in fact belong to $\bv\setminus\sobo^{1,1}$. In particular, the Caccioppoli inequality itself \emph{cannot} yield higher integrability results in the linear growth setting.

On the other hand, the approach of Section~\ref{sec:PR} is robust enough \emph{to apply to integrands} $(x,y,z)\mapsto f(x,y,z)$ indeed if suitable superlinear growth in the last variable is imposed and thus the Gehring-type improvement is available (cf.~\cite[Thm.~6.1]{AG}):
\begin{remark}[$p$-growth functionals: Partial regularity] Let $1<p<\infty$, $0<\mu<1$ and let $f\colon\Omega\times\R^{n}\times\rsym\to\R$ be a variational integrand that satisfies 
\begin{align*}
\begin{cases}
z\mapsto f(x,y,z)\;\text{is of class $\hold^{2}$},\\
|\D_{z}f(x,y,z)|\leq c_{1}(1+|z|)^{p-1},\\
c_{2}|z|^{p}-\gamma\leq f(x,y,z) \leq c_{3}(1+|z|^{p}),\\
|f(x_{1},y_{1},z)-f(x_{2},y_{2},z)|\leq c_{4}|y_{1}|^{\mu}(|x_{1}-x_{2}|+|y_{1}-y_{2}|)^{\mu}(1+|z|^{p})
\end{cases}
\end{align*}
for all $x,x_{1},x_{2}\in\Omega$, $y,y_{1},y_{2}\in\R^{n}$ and $z\in\rsym$ and constants $c_{1},...,c_{4}>0$, $\gamma>0$. Let $u\in\sobo_{\locc}^{1,p}(\Omega;\R^{n})$ be a local minimiser of the variational integral corresponding to $f$. Moreover, let $(x_{0},y_{0},z_{0})\in\Omega\times\R^{n}\times\rsym$ is such that $x_{0}$ is a Lebesgue point for both $u$ and $\sg(u)$, with Lebesgue values $y_{0}$ or $z_{0}$, respectively. If there exists $\lambda>0$ such that $\lambda|z|^{2}\leq \langle \D_{z}^{2}f(x,y,z_{0})z,z\rangle$ holds for all $z\in\rsym$ uniformly in an open neighbourhood of $(x_{0},y_{0})$, then $u$ has H\"{o}lder continuous gradients in an open neighbourhood of $x_{0}$. 
\end{remark}
In view of partial regularity, we have omitted  \emph{symmetric quasiconvex} functionals throughout. In fact, at present it is not known how to modify the method exposed in Section~\ref{sec:PR} even in the \emph{full gradient case} (also see the discussion in \cite{AG,Schmidt1}). The only result available in the $\bv$-full gradient, strongly quasiconvex case is due to \textsc{Kristensen} and the author \cite{GK2}, and the case of strongly symmetric-quasiconvex functionals on $\bd$ is due to the author \cite{G3}.  If the condition of strong symmetric quasiconvexity  pro forma is introduced for \emph{convex} $\hold^{2}$-integrands, then it translates to $3$-elliptic integrands in the sense of \eqref{eq:ellipticity} and does not apply to the very degenerate ellipticity regime covered by Theorem~\ref{thm:PR}. Whereas the main obstructions in \cite{GK2,G3} stem from the weakened convexity notion, they moreover require higher regularity of the variational integrands, namely, $\hold_{\locc}^{2,\mu}$ for some $\mu>1-\frac{2}{n}$. In this sense, the results of \cite{G3} and Theorem~\ref{thm:PR} are independent. 

As to Sobolev regularity, the case of non-autonomous integrands $(x,z)\mapsto f(x,z)$ which satisfy the obvious modification of \eqref{eq:ellipticity} uniformly in $x$, however, is more intricate. Even if $f$ is of class $\hold^{2}$ in the joint variable and satisfies the estimates corresponding to \cite[Ass.~4.22]{Bi2}, it is not fully clear to arrive at the decoupling estimates that eliminate the divergence as done in the proof of Theorem~\ref{thm:regdual}. Whereas for partial regularity $\hold^{0,\alpha}$-H\"{o}lder continuous $x$-dependence of $\D_{z}f$ still suffices, the corresponding Sobolev regularity theory is \emph{far from clear} when aiming at an ellipticity regime beyond $1<a<1+\frac{1}{n}$ (also see \textsc{Baroni, Colombo \& Mingione} \cite{BaCoMi} for the related borderline case $\frac{q}{p}=1+\frac{\alpha}{n}$ in the superlinear growth regime). Namely, in this case the Euler-Lagrange equations satisfied by (generalised) minima \emph{cannot} be differentiated. In the full gradient, superlinear growth regime, this setting has been extensively studied by \textsc{Mingione} \cite{Mingione0,Mingione2,Mingione3} and  \textsc{Kristensen \& Mingione} \cite{KrMi1,KrMi2,KrMi3}. Here, Nikolski\u{\i} estimates are employed but -- as a matter of fact -- do \emph{not use} any information apart from the Euler-Lagrange equation itself and the continuity properties of $f$ with respect to its first variable.  Such a strategy has been pursued in \cite{GK1} for autonomous functionals (in the regime $1<a<1+\frac{1}{n}$). However, if we wish to amplify the ellipticity regime as is done in Theorem~\ref{thm:W11reg}, then we ought to use the instrumental identities for the minimisers that come out as a byproduct of \emph{second order estimates}, cf.~Theorem~\ref{thm:regdual}. As the latter are essentially obtained by differentiating the first variation-style perturbed Euler-Lagrange equation, the approach presented in Section~\ref{sec:W11reg} requires modification, an objective which we intend to deal with in the future.
 
Lastly, let $\A[D]$ be a first order, constant-coefficient differential operator $\A[D]$ on $\R^{n}$ between the two finite dimensional vector spaces $V,W$. Then the canonical Dirichlet problem \eqref{eq:varprin} has a relaxed minimiser in $\bv^{\A}(\Omega):=\{v\in\lebe^{1}(\Omega;V)\colon\;\A[D]u\in\mathscr{M}(\Omega;W)\}$ if $\A$ is $\mathbb{C}$-elliptic (and hence $\lebe^{1}$-traces of $\bv^{\A}$-maps are definable), cf.~\cite[Thm.~5.3]{BDG}. By means of Hilbert-Nullstellensatz-techniques \cite{Smith,GRVS}, the splitting strategy underlying Theorem~\ref{thm:regdual} -- yet being technically more demanding -- is likely to work as well. On the other hand, based on the Poincar\'{e}-type inequalities from \cite{BDG}, the partial regularity result from Theorem~\ref{thm:PR} hinges on the existence of a mollifier $\rho$ such that $\rho*\pi=\pi$ for all $\pi\in\ker(\A[D])$. This is a consequence of the Bramble-Hilbert lemma, and we shall pursue this elsewhere.

\section{Appendix A: On uniqueness and the structure of $\gm(F;u_{0})$}\label{sec:uniquenessappendix}
In Section~\ref{sec:cors} we addressed some uniqueness assertions and the structure of the set of generalised minimisers. Working from the assumption that generalised minima are unique up to rigid deformations, we here complete the proof of Corollary~\ref{cor:uniqueness} with
\begin{proposition}\label{prop:uniqueness}
Let $\Omega\subset\R^{n}$ be open and bounded with Lipschitz boundary $\partial\Omega$ and let $u_{0}\in\ld(\Omega)$. Moreover, suppose that $f\colon\rsym\to\R$ is convex integrand with \eqref{eq:lingrowth1} such that for each $\nu\in\R^{n}\setminus\{0\}$ the map $f_{\nu}^{\infty}\colon\R^{n}\ni z \mapsto f^{\infty}(z\odot\nu)$ has strictly convex sublevel sets (in the sense of Section~\ref{sec:cors}) and every two generalised minima differ by a rigid deformation. Then the following hold: 
\begin{enumerate}
\item\label{item:uniqueness1} If there exists one generalised minimiser $u\in\gm(F;u_{0})$ with $u=u_{0}$ $\mathscr{H}^{n-1}$-a.e. on $\partial\Omega$, then $\gm(F;u_{0})=\{u\}$. 
\item\label{item:uniqueness2} If $\partial\Omega$ moreover satisfies for all $a\in\R$
\begin{align}\label{eq:uniquenesscond}
\mathscr{H}^{n-1}(\{x\in\partial\Omega\colon\;x_{i}=a\})= 0\qquad\text{for all}\;i\in\{1,...,n\},
\end{align}
then there exists $u\in\gm(F;u_{0})$ and $\pi\in\mathcal{R}(\Omega)$ such that 
\begin{align}\label{eq:oneparameter}
\gm(F;u_{0})=\{u+\lambda\pi\colon\;\lambda\in[-1,1]\}. 
\end{align}
\end{enumerate}
\end{proposition}
Note that, the hypotheses of Corollary~\ref{cor:uniqueness0} imply those of the preceding proposition. For the rest of this section, we tacitly assume that the hypotheses of Proposition~\ref{prop:uniqueness} are in action.

We begin with some preliminary considerations. Given $u_{0}\in\ld(\Omega)$ and a convex integrand with \eqref{eq:lingrowth1}, we start by noting that for any $u\in\gm(F;u_{0})$, the set 
\begin{align}\label{eq:ccb}
\begin{split}
\mathcal{R}_{u}:=\big\{\pi\in\mathscr{R}(\Omega)\colon\; & u+\pi\in\gm(F;u_{0})\big\}\\ & \;\;\;\;\;\text{is convex, closed and bounded in}\;\mathscr{R}(\Omega). 
\end{split}
\end{align}
Convexity of $\mathcal{R}_{u}$ is a direct consequence of convexity of $\overline{F}_{u_{0}}[-;\Omega]$ on $\bd(\Omega)$. If $(\pi_{j})\subset\mathcal{R}_{u}$ satisfies $\pi_{j}\to\pi$ in $\mathscr{R}(\Omega)$, then Lipschitz continuity\footnote{The recession function is convex and of linear growth, thus Lipschitz, too.} of $f^{\infty}$ readily implies that $\overline{F}_{u_{0}}[u+\pi;\Omega]=\lim_{j\to\infty}\overline{F}_{u_{0}}[u+\pi_{j};\Omega]=\min \overline{F}_{u_{0}}[\bd(\Omega);\Omega]$ and hence $\pi\in\mathcal{R}_{u}$, too. Lastly, if $\mathcal{R}_{u}$ were not bounded, we would find $(\pi_{j})\subset\mathcal{R}_{u}$ with $\|\pi_{j}\|\to\infty$ for an arbitrary norm $\|\cdot\|$ on $\mathscr{R}(\Omega)$. There exists a constant $c=c_{n}>0$ such that $c|a|\,|b|\leq |a\odot b|$ for all $a,b\in\R^{n}$. Since, by \eqref{eq:lingrowth1}, $c_{1}|z|\leq f^{\infty}(z)$ for all $z\in\rsym$, we find 
\begin{align*}
C\int_{\Omega}|\trace_{\partial\Omega}(u-u_{0}-\pi_{j})|\dif\mathscr{H}^{n-1} & \leq \int_{\Omega}f^{\infty}(\trace_{\partial\Omega}(u_{0}-u-\pi_{j})\odot\nu_{\partial\Omega})\dif\mathscr{H}^{n-1} \\ & \leq (\min\overline{F}_{u_{0}}[\bd(\Omega)]) - f[\E u](\Omega) < \infty,
\end{align*}
so that the triangle inequality and equivalence of all norms on $\mathscr{R}(\Omega)$ yields the contradictory $\sup_{j\in\mathbb{N}}\|\pi_{j}\|<\infty$. In consequence, \eqref{eq:ccb} follows.

 As an adaptation of \cite[Lem.~6.2]{BS1}, we now establish that whenever $\pi\in\mathscr{R}(\Omega)$ is such that $u+\pi\in\gm(F;u_{0})$, then there exists a $\mathscr{H}^{n-1}$-measurable function $\beta\colon\partial\Omega\to \R\setminus (0,1)$ such that $\trace_{\partial\Omega}(u_{0})(x)=\trace_{\partial\Omega}(u)(x)+\beta(x)\pi(x)$ for $\mathscr{H}^{n-1}$-a.e. $x\in\partial\Omega$. 
 
In fact, if $u+\pi\in\gm(F;u_{0})$, then by \eqref{eq:ccb}, $u+\beta\pi\in\gm(F;u_{0})$ for all $\beta\in [0,1]$. In particular, we find 
\begin{align*}
 2f[\E u](\Omega) & + \int_{\partial\Omega}f^{\infty}(\trace_{\partial\Omega}(u_{0}-u-\pi)\odot\nu_{\partial\Omega})\dif\mathscr{H}^{n-1}\\ & + \int_{\partial\Omega}f^{\infty}(\trace_{\partial\Omega}(u_{0}-u)\odot\nu_{\partial\Omega})\dif\mathscr{H}^{n-1} = 2 \min \overline{F}_{u_{0}}[\bd(\Omega)] \\ 
& \leq 2f[\E u](\Omega) + 2\int_{\partial\Omega}f^{\infty}\Big(\Big(\trace_{\partial\Omega}(u_{0}-u-\frac{\pi}{2})\Big)\odot\nu_{\partial\Omega}\Big)\dif\mathscr{H}^{n-1},
\end{align*}
and since 
\begin{align}\label{eq:uniquenessinequality}
\begin{split}
2f^{\infty}(\trace_{\partial\Omega}(u_{0}-u-\frac{\pi}{2}) \odot \nu_{\partial\Omega}) & \leq f^{\infty}(\trace_{\partial\Omega}(u_{0}-u)\odot\nu_{\partial\Omega}) \\ & \!\!\!\!\!\!\!\!\!\!\!\!\!\!\!\!\!\!\!\!\!\!\!\!+ f^{\infty}(\trace_{\partial\Omega}(u_{0}-u-\pi))\odot\nu_{\partial\Omega})\;\;\;\;\text{$\mathscr{H}^{n-1}$-a.e. on $\partial\Omega$}, 
\end{split}
\end{align}
we deduce that we have equality in \eqref{eq:uniquenessinequality} $\mathscr{H}^{n-1}$-a.e. on $\partial\Omega$. Because the map $z\mapsto f^{\infty}(z\odot\nu_{\partial\Omega}(x))$ has strictly convex level sets for $\mathscr{H}^{n-1}$-a.e. $x\in\partial\Omega$, by \cite[Lem.~4.8]{Schmidt2}, for $\mathscr{H}^{n-1}$-a.e. $x\in\partial\Omega$ there exists $R(x)\geq 0$ such that
\begin{align*}
\trace_{\partial\Omega}(u_{0}(x)-u(x)-\pi(x))=R(x)\trace_{\partial\Omega}(u_{0}(x)-u(x))\qquad\text{for $\mathscr{H}^{n-1}$-a.e. $x\in\partial\Omega$}. 
\end{align*}
Clearly, on $\{x\in\partial\Omega\colon\;\pi(x)=0\}$ we must have $R=1$. Conversely, on $\{x\in\partial\Omega\colon\;\pi(x)\neq 0\}$, we have $R\neq 1$, $\trace_{\partial\Omega}((1-R)(u_{0}-u)-\pi)=0$ and hence $\trace_{\partial\Omega}(u_{0}-u)=\frac{1}{1-R}\trace_{\partial\Omega}(\pi)$. We may thus define 
\begin{align*}
\beta(x):=\begin{cases} 1 &\;\text{where}\;\pi(x)=0,\\
\frac{1}{1-R(x)}&\;\text{otherwise},  
\end{cases}
\end{align*}
so that $\beta(x)\in\R\setminus (0,1)$, and it is easily seen that $\beta$ has the required properties. 

\begin{proof}[Proof of Proposition~\ref{prop:uniqueness}~\ref{item:uniqueness1}] In \cite{GK1} this has been established for \emph{convex} domains, and we here give the general case. Suppose that $v\in\gm(F;u_{0})$ is a generalised minimiser. Then $v=u+\pi$, and generalised minimality of $v$ together with $\trace_{\partial\Omega}(u)=\trace_{\partial\Omega}(u_{0})$ $\mathscr{H}^{n-1}$-a.e. on $\partial\Omega$ yields  
\begin{align*}
\int_{\partial\Omega}f^{\infty}(\pi(x)\odot\nu_{\partial\Omega}(x))\dif\mathscr{H}^{n-1}(x)=0. 
\end{align*}
Since $f^{\infty}(a\odot b)\geq C|a|\,|b|$ for some $C>0$ and all $a,b\in\R^{n}$, $\pi=0$ $\mathscr{H}^{n-1}$-a.e. (and thus, by continuity, everywhere) on $\partial\Omega$. Write $\pi(x)=Ax+b$ with $A\in\rscew$ and $b\in\R^{n}$. Clearly, for $\Omega$ is open and bounded, $\partial\Omega$ cannot be contained in an $(n-1)$-dimensional affine hyperplane. If $\dim(\ker(A))\leq n-1$, then $\ker(A)$ is contained in an $(n-1)$-dimensional hyperplane $H$. We have, for some $x_{0}\in\partial\Omega$, $\{y\colon\;Ay=-b\}=x_{0}+\ker(A)$. Since $\partial\Omega\not\subset x_{0}+\ker(A)$, we find $x_{1}\in\partial\Omega\cap(x_{0}+\ker(A))^{c}$. Then, however, $\pi(x_{1})=0$ implies $Ax_{1}=-b$ and so $x_{1}\in x_{0}+\ker(A)$, a contradiction. In consequence, necessarily $\dim(\ker(A))=n$, in which case $A=0$ so that, because of $\pi\equiv 0$ on $\partial\Omega$, $b=0$ and hence $\pi\equiv 0$ on $\R^{n}$. In conclusion, $u=v$ and hence $\gm(F;u_{0})=\{u\}$. 
\end{proof}
We now establish Proposition~\ref{prop:uniqueness}\ref{item:uniqueness2} for $n=2$; the higher dimensional case is similar and is left to the reader. 
\begin{proof}[Proof of Proposition~\ref{prop:uniqueness}~\ref{item:uniqueness2}]
By assumption, $\gm(F;u_{0})=u+\mathcal{R}_{u}$, $\mathcal{R}_{u}$ being defined as in \eqref{eq:ccb}. Suppose that $\mathcal{R}_{u}$ contains two linearly independent elements $\pi_{1},\pi_{2}$. Then, by the above discussion, we may write $\trace_{\partial\Omega}(u_{0})=\trace_{\partial\Omega}(u)+\beta_{1}\pi_{1}=\trace_{\partial\Omega}(u)+\beta_{2}\pi_{2}$ $\mathscr{H}^{1}$-a.e. on $\partial\Omega$ for some suitable $\beta_{1,2}\colon\partial\Omega\to\R\setminus (0,1)$. Therefore, $\beta_{1}\pi_{1}-\beta_{2}\pi_{2}=0$ $\mathscr{H}^{1}$-a.e. on $\partial\Omega$. We write $\pi_{1}(x)=A_{1}x+b_{1}$, $\pi_{2}(x)=A_{2}x+b_{2}$, where 
\begin{align*}
&A_{1}=\left(\begin{matrix} 0 & \lambda \\ -\lambda & 0 \end{matrix}\right),b_{1}=\left(\begin{matrix} \mathbf{b}_{11} \\ \mathbf{b}_{12}\end{matrix} \right)\;\;\;\text{and}\;\;\; A_{2}=\left(\begin{matrix} 0 & \mu \\ -\mu & 0 \end{matrix}\right),b_{2}=\left(\begin{matrix} \mathbf{b}_{21} \\ \mathbf{b}_{22}\end{matrix} \right)
\end{align*}
for some suitable $\lambda,\mu\in\R$, $b_{1},b_{2}\in\R^{2}$; in two dimensions, every rigid deformation is of this form. Now suppose that $\beta_{1}\pi_{1}-\beta_{2}\pi_{2}=0$ $\mathscr{H}^{1}$-a.e. on $\partial\Omega$, and denote $\Theta\subset\partial\Omega$ the set where equality holds; hence, $\mathscr{H}^{1}(\partial\Omega\setminus\Theta)=0$. Then for any $x=(x_{1},x_{2})\in\Theta$, 
\begin{align}\label{eq:uniquenessrewrite1}
\beta_{1}(x)\left(\begin{matrix} \lambda x_{2} + \mathbf{b}_{11} \\ -\lambda x_{1} + \mathbf{b}_{12} \end{matrix}\right) = \beta_{2}(x)\left(\begin{matrix} \mu x_{2} + \mathbf{b}_{21} \\ -\mu x_{1} + \mathbf{b}_{22} \end{matrix} \right).
\end{align}
Denote $\Gamma:=\{x\in\Theta\colon\;\beta_{1}(x)\neq 0\}$. Our aim is to establish $\mathscr{H}^{1}(\Gamma)=0$. We split 
\begin{align*}
\Gamma & = \Gamma_{1}\cup\Gamma_{2}\cup\Gamma_{3}\cup\Gamma_{4} \\ 
& := \{x\in\Gamma\colon \mu x_{2}+\mathbf{b}_{21} = 0\;\text{and}\;-\mu x_{1}+\mathbf{b}_{22} = 0\} \\ 
& \cup \{x\in\Gamma\colon \mu x_{2}+\mathbf{b}_{21} = 0\;\text{and}\;-\mu x_{1}+\mathbf{b}_{22} \neq 0\} \\ 
& \cup \{x\in\Gamma\colon \mu x_{2}+\mathbf{b}_{21} \neq 0\;\text{and}\;-\mu x_{1}+\mathbf{b}_{22} = 0\} \\ 
& \cup \{x\in\Gamma\colon \mu x_{2}+\mathbf{b}_{21} \neq 0\;\text{and}\;-\mu x_{1}+\mathbf{b}_{22} \neq 0\}.
\end{align*}
For $\Gamma_{1}$, note that if $\mu\neq 0$, then $\Gamma_{1}$ consists at most of one single point and hence $\mathscr{H}^{1}(\Gamma_{1})=0$. If $\mu=0$ and $\mathscr{H}^{1}(\Gamma_{1})>0$, then $\Gamma_{1}\neq\emptyset$ implies $\mathbf{b}_{21}=\mathbf{b}_{22}=0$ and hence, in total, by $\mu=0$, $\pi_{2}\equiv 0$, which is ruled out by linear independence of $\pi_{1},\pi_{2}$. Hence, $\mathscr{H}^{1}(\Gamma_{1})=0$.

Now, for $x\in\Gamma$, we may put $\gamma(x):=\frac{\beta_{2}(x)}{\beta_{1}(x)}$ and obtain from \eqref{eq:uniquenessrewrite1}
\begin{align}\label{eq:uniquenessrewrite2}
\left(\begin{matrix} \lambda x_{2} + \mathbf{b}_{11} \\ -\lambda x_{1} + \mathbf{b}_{12} \end{matrix}\right) = \gamma(x)\left(\begin{matrix} \mu x_{2} + \mathbf{b}_{21} \\ -\mu x_{1} + \mathbf{b}_{22} \end{matrix} \right).
\end{align} 
\begin{itemize}
\item The treatment of $\Gamma_{2}$ and $\Gamma_{3}$ is symmetric (interchange the roles of $x_{1}$ and $x_{2}$). So suppose that $\mathscr{H}^{1}(\Gamma_{2})>0$. If $\mu\neq 0$, then $\Gamma_{2}\subset \{x\in\Gamma\colon\;x_{2}=\frac{\mathbf{b}_{21}}{\mu}\}$ and hence $\mathscr{H}^{1}(\Gamma_{2})=0$ by \eqref{eq:uniquenesscond}, a contradiction. Thus $\mu=0$. From \eqref{eq:uniquenessrewrite2} we deduce that $\lambda x_{2}+\mathbf{b}_{11}=0$ for all $x\in\Gamma_{2}$. Again, if $\lambda\neq 0$, then $\Gamma_{2}\subset \{x\in\Gamma\colon\; x_{2}=-\frac{\mathbf{b}_{11}}{\lambda}\}$ and hence $\mathscr{H}^{1}(\Gamma_{2})=0$ by \eqref{eq:uniquenesscond}. Hence $\lambda=0$, and so $\pi_{1}=b_{1}$, $\pi_{2}=b_{2}$. In this situation, linear independence of $\pi_{1},\pi_{2}$  and hereafter of $b_{1},b_{2}$ implies that $\beta_{1}=\beta_{2}=0$ $\mathscr{H}^{1}$-a.e. on $\Gamma$, a contradiction to $\beta_{1}\neq 0$ $\mathscr{H}^{1}$-a.e. on $\Gamma$. As a conclusion, $\mathscr{H}^{1}(\Gamma_{2})=0$, and similarly, now invoking the first part of \eqref{eq:uniquenesscond}, $\mathscr{H}^{1}(\Gamma_{3})=0$. 
\item Suppose that $\mathscr{H}^{1}(\Gamma_{4})>0$. For $x\in\Gamma_{4}$, we have  $\mu x_{2}+\mathbf{b}_{21}\neq 0$ and $-\mu x_{1}+\mathbf{b}_{22}\neq 0$. From here we deduce 
\begin{align}\label{eq:fourthcase}
\frac{\lambda x_{2}+\mathbf{b}_{11}}{\mu x_{2}+\mathbf{b}_{21}}= \gamma(x) = \frac{\lambda x_{1}-\mathbf{b}_{12}}{\mu x_{1}-\mathbf{b}_{22}}\qquad\text{for all}\;x\in\Gamma_{4}. 
\end{align}
Therefore, $\gamma(x)$ must be independent of $x_{1},x_{2}$ and thus is constant. Hence, there exists $a\in\R$ such that $\pi_{1}=a\pi_{2}$ on $\Gamma_{4}$. The affine-linear map $\pi_{1}-a\pi_{2}$ thus vanishes on a set of positive $\mathscr{H}^{1}$-measure. Therefore, it necessarily vanishes on a line $\ell\subset\R^{2}$. In other words, 
\begin{align}\label{eq:uniquesolvability}
(A_{1}-aA_{2})x=a(b_{2}-b_{1})\qquad\text{for}\;x\in\ell. 
\end{align}
There are three options: If $A_{1}-aA_{2}$ is invertible, then $(A_{1}-aA_{2})x=a(b_{2}-b_{1})$ has a unique solution and thus contradicts \eqref{eq:uniquesolvability} for all $x\in\ell$. Thus, $A_{1}-a A_{2}$ is not invertible, and by the structure of $A_{1},A_{2}$, this implies $\lambda = a\mu$. Either $a=0$, in which case \eqref{eq:fourthcase} yields $\mathbf{b}_{11}=\mathbf{b}_{12}=0$. Then $\pi_{1}\equiv 0$, contradicting the linear independence of $\pi_{1},\pi_{2}$. If $a\neq 0$, then \eqref{eq:uniquenessrewrite2} yields $\mathbf{b}_{11}=a\mathbf{b}_{21}$ and $\mathbf{b}_{12}=a\mathbf{b}_{22}$. In conclusion, $\pi_{1}=a\pi_{2}$, again contradicting the linear independence of $\pi_{1}$ and $\pi_{2}$. Therefore, $\mathscr{H}^{1}(\Gamma_{4})=0$. 
\end{itemize}
In conclusion, $\mathscr{H}^{1}(\Gamma)=0$ so that $\beta_{1}=0$ $\mathscr{H}^{1}$-a.e. on $\partial\Omega$. Then we obtain from $\trace_{\partial\Omega}(u_{0})=\trace_{\partial\Omega}(u)+\beta_{1}\pi_{1}=\trace_{\partial\Omega}(u)$ $\mathscr{H}^{1}$-a.e. on $\partial\Omega$ that $u\in\gm(F;u_{0})$ is a generalised minimiser which attains the correct boundary data $u_{0}$ $\mathscr{H}^{1}$-a.e.. In this situation, Proposition~\ref{prop:uniqueness}\ref{item:uniqueness1} yields $\gm(F;u_{0})=\{u\}$. In total, $\gm(F;u_{0})\subset u + \R\pi$ for some suitable $\pi\in\mathscr{R}(\Omega)$. Since in this situation $\gm(F;u_{0})$ is a closed and bounded interval by \eqref{eq:ccb}, the statement of  Proposition~\ref{prop:uniqueness} for $n=2$ follows.  
\end{proof}
Proposition~\ref{prop:uniqueness} rises the question under which minimal geometric assumptions on $\partial\Omega$ the representation \eqref{eq:oneparameter} continues to hold, an issue that we intend to pursue elsewhere.
\section{Appendix B: Proofs of auxiliary results}\label{sec:proofaux}
We now collect here the proofs of some minor auxiliary results used in the main part of the paper.
\subsection{On the $\lebe^{q}$-stability \eqref{eq:Lpstability}}\label{sec:stability}
We start by justifying \eqref{eq:Lpstability}. Let $x_{0}\in\R^{n}$ and $r>0$. Pick an $\lebe^{2}$-orthonormal basis $\{\pi_{1},...,\pi_{N}\}$ of $\mathscr{R}(\ball(0,1))$ and consider the orthonormal projection $\Pi_{\ball(0,1)}\colon\lebe^{2}(\ball(0,1);\R^{n})\to\mathscr{R}(\ball(0,1))$ given by $\Pi_{\ball(0,1)}v:=\sum_{k=1}^{N}\langle v,\pi_{k}\rangle_{\lebe^{2}}\pi_{k}$. For $\mathscr{R}(\ball(0,1))$ consists of polynomials, it is clear that we may also admit $v\in\lebe^{1}(\ball(0,1);\R^{n})$ in the last formula. By \eqref{eq:rigidscaling}, this yields the estimate 
\begin{align*}
\|\Pi_{\ball(0,1)}v\|_{\lebe^{1}(\ball(0,1);\R^{n})}\leq c(n)\|v\|_{\lebe^{1}(\ball(0,1);\R^{n})}
\end{align*}
 for $v\in\lebe^{1}(\ball(0,1);\R^{n})$ so that $\Pi_{\ball(0,1)}$ extends to a bounded linear operator from $\lebe^{1}(\ball(0,1);\R^{n})$ to $\mathscr{R}(\ball(0,1))$. Now \eqref{eq:Lpstability} follows by rescaling. 
\subsection{Proof of Lemma~\ref{lem:shifted}}\label{sec:proofshifted}
Let $a\in\rsym$ be fixed and let $\xi\in\rsym$ be arbitrary. Assertion \ref{item:shifted1} follows by differentiation, and $f_{a}\geq 0$ is a consequence of convexity of $f$. As to \ref{item:shifted2}, since $f$ is Lipschitz by Lemma~\ref{lem:boundbelow} and because of $\mathbb{B}(a,\tfrac{\varrho_{\xi_{0}}}{2})\subset\mathbb{B}(\xi_{0},\varrho_{\xi_{0}})$, 
\begin{align*}
f_{a}(\xi) & = \int_{0}^{1}\langle f'(a+t\xi)-f'(a),\xi\rangle\dif t \leq \begin{cases} (\sup_{\mathbb{B}(\xi_{0},\varrho_{\xi_{0}})}|f''|)|\xi|^{2}&\;\text{for}\;|\xi|\leq \tfrac{\varrho_{\xi_{0}}}{2},\\
2\Lip(f)|\xi|&\;\text{for}\;|\xi|>\tfrac{\varrho_{\xi_{0}}}{2}.
\end{cases}
\end{align*}
Therefore, if $|\xi|>\tfrac{\varrho_{\xi_{0}}}{2}$, we may successively apply Lemma~\ref{lem:eest}\ref{item:eest3} and \ref{item:eest1} to find
\begin{align*}
|\xi| = \frac{\varrho_{\xi_{0}}}{2}\left\vert \frac{2}{\varrho_{\xi_{0}}}\xi\right\vert \leq \frac{\varrho_{\xi_{0}}}{2}\frac{1}{\sqrt{2}-1}V\Big(\frac{2}{\varrho_{\xi_{0}}}\xi \Big)\leq \frac{8}{\sqrt{2}-1}\frac{1}{\varrho_{\xi_{0}}}V(\xi).
\end{align*}
Thus, by Lemma~\ref{lem:eest}\ref{item:eest4} with $\ell=\tfrac{\varrho_{\xi_{0}}}{2}$ and the corresponding constant $c=c(\ell)=c(\tfrac{\varrho_{\xi_{0}}}{2})$
\begin{align*}
f_{a}(\xi) \leq \Big(c(\tfrac{\varrho_{\xi_{0}}}{2})\sup_{\mathbb{B}(\xi_{0},\varrho_{\xi_{0}})}|f''| +\frac{16\Lip(f)}{(\sqrt{2}-1)\varrho_{\xi_{0}}}\Big)V(\xi).
\end{align*}
For the lower bound, we observe that by \eqref{eq:positivedefinitebound} and $\mathbb{B}(a,\tfrac{\varrho_{\xi_{0}}}{2})\subset\mathbb{B}(\xi_{0},\varrho_{\xi_{0}})$, 
\begin{align*}
f_{a}(\xi) & = \int_{0}^{1}\int_{0}^{1}\langle f''(a+st\xi)\xi,\xi\rangle\dif s\dif t \geq m_{\xi_{0},\varrho_{\xi_{0}}}|\xi|^{2}\qquad\text{for all}\;\xi\in\mathbb{B}(0,\tfrac{\varrho_{\xi_{0}}}{2}). 
\end{align*}
Similarly, if $\tfrac{\varrho_{\xi_{0}}}{2}\leq|\xi|$, then positive definiteness of $f''$ on $\rsym$ yields
\begin{align*}
f_{a}(\xi) = f_{a}(\xi)-f_{a}(0) \geq \int_{0}^{\sqrt{\frac{\varrho_{\xi_{0}}}{2|\xi|}}}\int_{0}^{\sqrt{\frac{\varrho_{\xi_{0}}}{2|\xi|}}}\left\langle f''(a+st\xi)\xi,\xi\right\rangle\dif t\dif s \geq m_{\xi_{0},\varrho_{\xi_{0}}}\frac{\varrho_{\xi_{0}}}{2|\xi|}|\xi|^{2}.
\end{align*}
Hence, we obtain for all $\xi\in\rsym$ by Lemma~\ref{lem:eest}\ref{item:eest3} and monotonicity of $\R\ni t\mapsto V(t)$,
\begin{align*}
f_{a}(\xi) & \geq m_{\xi_{0},\varrho_{\xi_{0}}}V(\xi)\mathbbm{1}_{\{|\xi|<\varrho_{\xi_{0}}/2\}}(\xi)+ m_{\xi_{0},\varrho_{\xi_{0}}}\Big(\frac{\varrho_{\xi_{0}}}{2}\Big)^{2}V\Big(\frac{2\xi}{\varrho_{\xi_{0}}}\Big)\mathbbm{1}_{\{|\xi|\geq\varrho_{\xi_{0}}/2\}}(\xi) \\ & \geq m_{\xi_{0},\varrho_{\xi_{0}}}\Big(\frac{\varrho_{\xi_{0}}}{2}\Big)^{2}V(\xi).
\end{align*}
The proof is complete. 

\subsection{Linear comparison estimates and the proof of Proposition~\ref{prop:1}}\label{sec:LCE}
Let $\Omega\subset\R^{n}$ be an open and bounded domain with smooth boundary.  For $w\in\sobo^{1,2}(\Omega;\R^{n})$, consider the variational principle
\begin{align}\label{eq:linearproblem}
\text{to minimise}\;\;G[v]:=\int_{\Omega}g(\sg(v))\dif x\qquad \text{over}\;v\in w+\sobo_{0}^{1,2}(\Omega;\R^{n}), 
\end{align}
where $g(z):=\mathscr{A}[z,z] + \langle b,z\rangle + c$ is a polynomial of degree two on $\rsym$ with a symmetric bilinear form $\mathscr{A}\colon\rsym\times\rsym\to\R$, $b\in\rsym$ and $c\in\R$. We moreover assume that $\mathscr{A}$ is elliptic in the sense that there exists $\ell_{1},\ell_{2}>0$ such that $\ell_{1}|z|^{2}\leq \mathscr{A}[z,z] \leq \ell_{2} |z|^{2}$ holds for all $z\in\R_{\sym}^{n\times n}$.
\begin{lemma}\label{lem:lineardecayestimates}
There exists a unique solution $u\in w+\sobo_{0}^{1,2}(\Omega;\R^{n})$ of \eqref{eq:linearproblem}. Moreover, this solution satisfies the following: 
\begin{enumerate}
\item\label{item:linearcomparison1} There exists a constant $c=c(n,\ell_{1},\ell_{2})>0$ such that if $\ball(x_{0},R)\Subset\Omega$, then for all $0<r<R/2$ there holds 
\begin{align*}
\int_{\ball(x_{0},r)}|\sg(u)-(\sg(u))_{x_{0},r}|^{2}\dif x \leq c\left(\frac{r}{R}\right)^{n+2}\int_{\ball(x_{0},R/2)}|\sg(u)-(\sg(u))_{x_{0},R/2}|^{2}\dif x. 
\end{align*}
\item\label{item:linearcomparison2} If $\Omega=\ball(x_{0},R)$ for some $x_{0}\in\R^{n}$ and $R>0$, then for any $\alpha\in (0,1)$ there exists a constant $c=c(n,\alpha,\ell_{1},\ell_{2})>0$ such that if $w\in\hold^{1,\alpha}(\overline{\ball(x_{0},R)};\R^{n})$, then 
\begin{align*}
[\sg(u)]_{\hold^{0,\alpha}(\overline{\ball(x_{0},R)};\rsym)}\leq c [\sg(w)]_{\hold^{0,\alpha}(\overline{\ball(x_{0},R)};\rsym)}.
\end{align*}
\end{enumerate}
\end{lemma}
\begin{proof}
\textsc{Korn}'s inequality $\|\nabla\varphi\|_{\lebe^{2}(\Omega;\R^{n\times n})}\leq c\|\sg(\varphi)\|_{\lebe^{2}(\Omega;\R^{n\times n})}$ for all $\varphi\in\sobo_{0}^{1,2}(\Omega;\R^{n})$ implies that minimising sequences are bounded in $\sobo^{1,2}(\Omega;\R^{n})$ (as the Dirichlet datum $w$ is fixed). From here, the existence of minima is standard by convexity of $g$, and uniqueness follows from strict convexity of $g$. The proof of \ref{item:linearcomparison1} follows along the lines of \cite[Lem.~3.0.5]{FS2}. For~\ref{item:linearcomparison2}, consider the symmetric bilinear form $\mathscr{B}\colon\R^{n\times n}\times\R^{n\times n}\to\R$ defined by $\mathscr{B}[z,\xi]:=\mathscr{A}[z^{\sym},\xi^{\sym}]$ for $z,\xi\in\R^{n\times n}$. Then $\mathscr{B}$ is strongly elliptic in the sense of Legendre-Hadamard: For all $a,b\in\R^{n}$ there holds 
\begin{align*}
\mathscr{B}[a\otimes b,a\otimes b] = \mathscr{A}[a\odot b,a\odot b] \geq c(n,\ell_{1},\ell_{2})|a|^{2}|b|^{2}, 
\end{align*}
and since trivially $|\mathscr{B}[z,\xi]|\leq c(n,\ell_{1},\ell_{2})|z||\xi|$ for all $z,\xi\in\R^{n\times n}$, $\mathscr{B}$ is a strongly elliptic bilinear form on $\R^{n\times n}$. By minimality of $u$ for $G$, $u$ satisfies the Euler-Lagrange equation 
\begin{align}
\begin{cases}
-\di(\mathscr{B}[\nabla u,\cdot]) = 0&\;\text{in}\;\Omega, \\ \;\;\;\;\;\;\;\;\;\;\;\;\;\;\;\;\;\;\;\;\;\;\;\;
u = w&\;\text{on}\;\partial\Omega. 
\end{cases}
\end{align}
Therefore, by the classical Schauder estimates for strongly elliptic systems and scaling, there exists a constant $c=c(n,\ell_{1},\ell_{2})>0$ such that 
\begin{align*}
[\nabla u]_{\hold^{0,\alpha}(\overline{\ball(x_{0},R)};\R^{n\times n})}\leq c [\nabla w]_{\hold^{0,\alpha}(\overline{\ball(x_{0},R)};\R^{n\times n})}.
\end{align*}
Trivially, $[\sg(u)]_{\hold^{0,\alpha}(\overline{\ball(x_{0},R)};\rsym)}\leq [\nabla u]_{\hold^{0,\alpha}(\overline{\ball(x_{0},R)};\R^{n\times n})}$. By the simple geometry of $\ball(x_{0},R)$, $\mathcal{L}^{2,n+2\alpha}(\ball(x_{0},R);\rsym)\simeq \hold^{0,\alpha}(\overline{\ball(x_{0},R)};\rsym)$ with the Campanato spaces $\mathcal{L}^{p,\lambda}$. We then estimate, using Lemma~\ref{lem:Cianchi}\ref{item:Cianchi0} in the third step and scaling, 
\begin{align*}
[\nabla w & ]_{\hold^{0,\alpha}(\overline{\ball(x_{0},R)};\R^{n\times n})} \leq c [\nabla w]_{\mathcal{L}^{2,n+2\alpha}(\ball(x_{0},R);\R^{n\times n})}\\ 
& = c\sup_{x\in\Omega}\sup_{0<r<2R}\Big(\frac{1}{r^{n+2\alpha}}\int_{\ball(x,r)\cap\ball(x_{0},R)}|\nabla w - (\nabla w)_{\ball(x,r)\cap\ball(x_{0},R)}|^{2}\dif\mathscr{L}^{n}\Big)^{\frac{1}{2}}\\ 
& \leq c\sup_{x\in\Omega}\sup_{0<r<2R}\Big(\frac{1}{r^{n+2\alpha}}\int_{\ball(x,r)\cap\ball(x_{0},R)}|\sg(w) - (\sg(w))_{\ball(x,r)\cap\ball(x_{0},R)}|^{2}\dif\mathscr{L}^{n}\Big)^{\frac{1}{2}} \\ 
& \leq c[\sg(u)]_{\hold^{0,\alpha}(\overline{\ball(x_{0},R)};\rsym)}, 
\end{align*} 
where still $c=c(n,\ell_{1},\ell_{2})$. This yields \ref{item:linearcomparison2}, and the proof is complete. 
\end{proof}
The key in the above proof is that an easy reduction to the strongly elliptic bilinear forms applied to the \emph{full} gradients is possible. This is \emph{not} the case for elliptic bilinear forms. Clearly, in  \ref{item:linearcomparison2} we could have allowed for more general domains, but this is not needed for the 
\begin{proof}[Proof of Proposition~\ref{prop:1}] 
We split the proof into two steps, linearisation and comparison estimates. 

\emph{Step 1. Linearisation.} We begin by defining the auxiliary integrand $g\colon\rsym\to\R$ by
\begin{align*}
g(\xi):=f(\xi_{0})+\langle f'(\xi_{0}),(\xi-\xi_{0})\rangle+\tfrac{1}{2}\langle f''(\xi_{0})(\xi-\xi_{0}),(\xi-\xi_{0})\rangle,\;\;\;\xi\in\rsym.  
\end{align*}
Using a Taylor expansion of $f$ up to order two around $\xi_{0}$, we deduce by \eqref{eq:modcon} that 
\begin{align}\label{eq:Taylorsecond}
|f(\xi)-g(\xi)|\leq\omega_{\xi_{0},\varrho_{\xi_{0}}}(|\xi-\xi_{0}|)|\xi-\xi_{0}|^{2},\;\;\;\xi\in\mathbb{B}(\xi_{0},\varrho_{\xi_{0}}). 
\end{align}
By Lemma \ref{lem:lineardecayestimates}, the unique solution $h$ of the auxiliary variational principle
\begin{align}\label{eq:auxprob}
\text{to minimise}\;\int_{\ball(x_{0},R/2)}g(\sg(w))\dif x\;\;\text{over all}\;w\in v + \sobo_{0}^{1,2}(\ball(x_{0},\tfrac{R}{2});\R^{n}), 
\end{align}
belongs to $\hold^{1,\alpha}(\overline{\ball(x_{0},R/2)};\R^{n})$. By Lemma~\ref{lem:lineardecayestimates}~\ref{item:linearcomparison1}, there exists $c=c(n,\lambda,\Lambda)>0$ such that
\begin{align}\label{eq:compare}
\int_{\ball(x_{0},r)}|\sg(h)-(\sg(h))_{x_{0},r}|^{2}\dif x \leq c\left(\frac{r}{R}\right)^{n+2}\int_{\ball(x_{0},R/2)}|\sg(h)-(\sg(h))_{x_{0},R/2}|^{2}\dif x
\end{align}
for all $0<r<R/2$. Moreover, enlarging $c>0$ if necessary, Lemma \ref{lem:lineardecayestimates}~\ref{item:linearcomparison2} gives
\begin{align}\label{eq:Schauder}
[\sg(h)]_{\hold^{0,\alpha}(\overline{\ball(x_{0},R/2)};\rsym)}\leq c [\sg(v)]_{\hold^{0,\alpha}(\overline{\ball(x_{0},R/2)};\rsym)}.
\end{align}
Since $h$ is a solution of the variational principle \eqref{eq:auxprob}, the bounds of \eqref{eq:ellipticcomparison} yield that \begin{align}\label{eq:L2compas}
\|\sg(h)-\xi_{0}\|_{\lebe^{2}(\ball(x_{0},R/2);\rsym)} \leq c\|\sg(v)-\xi_{0}\|_{\lebe^{2}(\ball(x_{0},R/2);\rsym)},
\end{align}
where still $c=c(n,\lambda,\Lambda)>0$. Therefore we deduce for every $x\in\ball(x_{0},\tfrac{R}{2})$ that
\begin{align*}
|\sg(h)(x)-\xi_{0}| & \;\;\;\;\;\;\;\leq \sup_{\ball(x_{0},R/2)}|\sg(h)-(\sg(h))_{x_{0},R/2}| + |(\sg(h))_{x_{0},R/2}-\xi_{0}|\\
& \;\;\stackrel{\eqref{eq:Schauder},\eqref{eq:L2compas}}{\leq} c R^{\alpha}[\sg(v)]_{\hold^{0,\alpha}(\overline{\ball(x_{0},R/2)};\rsym)}+\Big(\dashint_{\ball(x_{0},R/2)}|\sg(v)-\xi_{0}|^{2}\dif x \Big)^{\frac{1}{2}}\\
& \;\;\;\;\;\;\,\,\leq c R^{\alpha}[\sg(v)]_{\hold^{0,\alpha}(\overline{\ball(x_{0},R/2)};\rsym)}+\sup_{\ball(x_{0},R/2)}|\sg(v)-\xi_{0}|\\
&  \;\;\;\;\;\;\;=: c_{\comp}\mathbf{t}_{\alpha,\xi_{0}}(v;x_{0},\tfrac{R}{2}),
\end{align*}
where $c_{\comp}=c_{\comp}(\lambda,\Lambda,n)>1$ shall be the constant claimed in the proposition, and so
\begin{align}\label{eq:supremumbound}
\sup_{\ball(x_{0},R/2)}|\sg(h)-\xi_{0}| \leq c_{\comp}\mathbf{t}_{\alpha,\xi_{0}}(v;x_{0},\tfrac{R}{2}). 
\end{align}
\emph{Step 2. Comparison estimates.} We will now compare $v$ with $h$. To this end, we first notice that by Jensen's inequality,  \eqref{eq:compare} and $0<r<\tfrac{R}{2}$,
\begin{align*}
\int_{\ball(x_{0},r)}|\sg(v)- (\sg(v))_{x_{0},r}|^{2}\dif x & \leq c\Big(\int_{\ball(x_{0},r)}|\sg(v)-\sg(h)|^{2}\dif x
 \Big. \\ & \Big. \;\;\;\;\;\;\;\;\;\;\;\;\;\;\;\;\;\;\;\;\;\;\;\;\;\;\;\;\;\; + \int_{\ball(x_{0},r)}|\sg(h)-(\sg(h))_{x_{0},r}|^{2}\dif x\Big) \\
& \leq c\Big(\int_{\ball(x_{0},r)}|\sg(v)-\sg(h)|^{2}\dif x \Big. \\ & \Big.\;\;\;\;\;\;\;\;+ \left(\frac{r}{R}\right)^{n+2}\int_{\ball(x_{0},R/2)}|\sg(h)-(\sg(h))_{x_{0},R/2}|^{2}\dif x\Big)\\
& \leq c\Big( \int_{\ball(x_{0},R/2)}|\sg(v)-\sg(h)|^{2}\dif x \Big. \\ & \Big.\;\;\;\;\;\;\;\;+ \left(\frac{r}{R}\right)^{n+2}\int_{\ball(x_{0},R/2)}|\sg(v)-(\sg(v))_{x_{0},R/2}|^{2}\dif x\Big), 
\end{align*}
where $c=c(n,\lambda,\Lambda)>0$. In view of \eqref{eq:smoothes}, we thus need to control the first term on the very right hand side of the previous inequality. Since $h$ solves \eqref{eq:auxprob} and $v-h\in\sobo_{0}^{1,2}(\ball(x_{0},\tfrac{R}{2});\R^{n})$, an elementary integration by parts establishes
\begin{align*}
\frac{1}{2}  \int_{\ball(x_{0},R/2)}  \langle f''(\xi_{0})(\sg(v)  -\sg(h)) ,(\sg(v)-\sg(h))\rangle\dif x = \int_{\ball(x_{0},R/2)}g(\sg(v))-g(\sg(h))\dif x. 
\end{align*}
Using this equality in the second step, we then deduce 
\begin{align*}
\int_{\ball(x_{0},R/2)}|\sg(v)-\sg(h)|^{2} \dif x &\; \stackrel{\eqref{eq:ellipticcomparison}}{\leq} \frac{1}{\lambda} \int_{\ball(x_{0},R/2)}\langle f''(\xi_{0})(\sg(v)-\sg(h)),(\sg(v)-\sg(h))\rangle\dif x \\ 
& =\frac{2}{\lambda} \int_{\ball(x_{0},R/2)}g(\sg(v))-g(\sg(h))\dif x \\ 
& = \frac{2}{\lambda}\Big( \int_{\ball(x_{0},R/2)}g(\sg(v))-f(\sg(v))\dif x \Big. \\ &\Big. + \int_{\ball(x_{0},R/2)}f(\sg(v))-f(\sg(h))\dif x \Big. \\ & \Big. + \int_{\ball(x_{0},R/2)}f(\sg(h))-g(\sg(h))\dif x\Big) \\ & =: \frac{2}{\lambda}\big(\mathbf{I}_{1}+\mathbf{I}_{2}+\mathbf{I}_{3}\big), 
\end{align*}
the single terms $\mathbf{I}_{1},\mathbf{I}_{2},\mathbf{I}_{3}$ being defined in the obvious manner.

Ad $\mathbf{I}_{1}$. Since $c_{\comp}>1$ and by virtue of our assumption $\mathbf{t}_{\alpha,\xi_{0}}(v;x_{0},R/2)<\varrho_{\xi_{0}}/c_{\comp}$, we  obtain $\sg(v)(x)\in\mathbb{B}(\xi_{0},\varrho_{\xi_{0}})$ for all $x\in\ball(x_{0},R/2)$. In consequence, by \eqref{eq:Taylorsecond}, the definition of $\mathbf{t}_{\alpha,\xi_{0}}$ and because $\omega_{\xi_{0},\varrho_{\xi_{0}}}$ is non-decreasing, 
\begin{align*}
\mathbf{I}_{1}=\int_{\ball(x_{0},R/2)}g(\sg(v))-f(\sg(v))\dif x \leq \omega_{\xi_{0},\varrho_{\xi_{0}}}(\mathbf{t}_{\alpha,\xi_{0}}(v;x_{0},R/2))\int_{\ball(x_{0},R/2)}|\sg(v)-\xi_{0}|^{2}\dif x.
\end{align*}

Ad $\mathbf{I}_{2}$. Here we invoke the definition of $\devi$ and minimality of $h$ for \eqref{eq:auxprob}, yielding $\mathbf{I}_{2}\leq \devi(v;x_{0},R/2)$.

Ad $\mathbf{I}_{3}$. By our choice \eqref{eq:supremumbound} of $c_{\comp}>1$ and $\mathbf{t}_{\alpha,\xi_{0}}(v;x_{0},R/2)<\varrho_{\xi_{0}}/c_{\comp}(<1)$, \eqref{eq:supremumbound} implies that  $\sg(h)(x)\in\mathbb{B}(\xi_{0},\varrho_{\xi_{0}})$ for all $x\in\ball(x_{0},R/2)$. Hence, by \eqref{eq:Taylorsecond}, 
\begin{align*}
|f(\sg(h)(x))-g(\sg(h)(x))|\leq \omega_{\xi_{0},\varrho_{\xi_{0}}}(|\sg(h)(x)-\xi_{0}|)|\sg(h)(x)-\xi_{0}|^{2}
\end{align*}
for all $x\in\ball(x_{0},R/2)$. Now, because $\omega_{\xi_{0},\varrho_{\xi_{0}}}$ is non-decreasing, \eqref{eq:supremumbound} and \eqref{eq:L2compas} imply
\begin{align*}
\mathbf{I}_{3} & = \int_{\ball(x_{0},R/2)}f(\sg(h))-g(\sg(h))\dif x \\ & \leq \omega_{\xi_{0},\varrho_{\xi_{0}}}(c_{\comp}\mathbf{t}_{\alpha,\xi_{0}}(v;x_{0},R/2))\int_{\ball(x_{0},R/2)}|\sg(h)-\xi_{0}|^{2}\dif x\\
&  \leq c \omega_{\xi_{0},\varrho_{\xi_{0}}}(c_{\comp}\mathbf{t}_{\alpha,\xi_{0}}(v;x_{0},R/2))\int_{\ball(x_{0},R/2)}|\sg(v)-\xi_{0}|^{2}\dif x, 
\end{align*}
where $c=c(n,\lambda,\Lambda)>0$. In conclusion, we find with some constant $c=c(n,\lambda,\Lambda)>0$
\begin{align*}
\int_{\ball(x_{0},r)}|\sg(v)-(\sg(v)&)_{x_{0},r}|^{2}\dif x  \leq c\Big(\left(\frac{r}{R}\right)^{n+2}\int_{\ball(x_{0},R/2)}|\sg(v)-(\sg(v))_{z,R/2}|^{2}\dif x \Big. \\ & \Big. + \devi(v;x_{0},R/2) + \vartheta(\mathbf{t}_{\alpha,\xi_{0}}(v;x_{0},R/2))\int_{\ball(x_{0},R/2)}|\sg(v)-\xi_{0}|^{2}\dif x\Big), 
\end{align*}
where $\vartheta(t):=\omega_{\xi_{0},\varrho_{\xi_{0}}}(t)+\omega_{\xi_{0},\varrho_{\xi_{0}}}(c_{\comp}t)$ meets the required properties. This is \eqref{eq:smoothes}, and the proof is complete.  
\end{proof}

\end{document}